\documentclass[11pt, leqno]{amsart}

\usepackage{
  amsmath,
  amsthm,
  amssymb,
  tikz,
  fancyhdr,
  enumitem,
  multirow,
}

\usetikzlibrary{arrows,calc,matrix,shapes}

\usepackage[colorlinks=true, pdfstartview=FitV,linkcolor=blue,citecolor=black,urlcolor=black]{hyperref}

\usepackage[vcentermath,enableskew]{youngtab}
\usepackage{mathrsfs}
\usepackage[center,small,sc]{caption}

\usepackage{multicol}

\usepackage[all]{xy}

\usepackage[normalem]{ulem}  
\usepackage{cancel}
\usepackage[colorinlistoftodos]{todonotes}

\usepackage{graphicx}
\newcommand\sbullet[1][.5]{\mathbin{\vcenter{\hbox{\scalebox{#1}{$\bullet$}}}}}

\theoremstyle{plain}
\newtheorem{theorem}{Theorem}[section]
\newtheorem{lemma}[theorem]{Lemma}
\newtheorem{proposition}[theorem]{Proposition}
\newtheorem{corollary}[theorem]{Corollary}
\newtheorem{conjecture}[theorem]{Conjecture}
\newtheorem{convention}[theorem]{Convention}

\theoremstyle{definition}
\newtheorem{definition}[theorem]{Definition}
\newtheorem{example}[theorem]{Example}
\newtheorem{remark}[theorem]{Remark}
\numberwithin{equation}{section} \numberwithin{figure}{section}
\numberwithin{table}{section}

\newcommand{\seteq}{\mathbin{:=}}
\newcommand{\GL}{\operatorname{GL}}
\newcommand{\Sp}{\operatorname{Sp}}
\newcommand{\USp}{\operatorname{USp}}
\newcommand{\GSp}{\operatorname{GSp}}
\newcommand{\U}{\operatorname{U}}
\newcommand{\SU}{\operatorname{SU}}

\newcommand{\tr}{\operatorname{tr}}

\newenvironment{red}{\relax\color{red}}{\relax}
\newenvironment{blue}{\relax\color{blue}}{\hspace*{.5ex}\relax}
\newenvironment{jaune}{\relax\color{green}}{\hspace*{.5ex}\relax}

\newcommand{\ber}{\begin{red}}
\newcommand{\er}{\end{red}}
\newcommand{\beb}{\begin{blue}}
\newcommand{\eb}{\end{blue}}
\newcommand{\bjn}{\begin{jaune}}
\newcommand{\ejn}{\end{jaune}}

\newcommand{\Z}{\mathbb{Z}}

\newcommand{\Pset}{\Phi_{\mathtt A_1 \times \mathtt A_1}}
\newcommand{\Oone}{\overline{1}}
\newcommand{\Otwo}{\overline{2}}
\newcommand{\Otk}{1^k}
\newcommand{\xs}{x^s}
\newcommand{\Okb}{1^{k-b}}
\newcommand{\OTb}{\Otwo^b}
\newcommand{\Ob}{1^b}
\newcommand{\Tb}{2^b}
\newcommand{\Ott}{1^t}
\newcommand{\Ttk}{2^k}
\newcommand{\Ttt}{2^t}
\newcommand{\OTtk}{\Otwo^k}
\newcommand{\OTtt}{\Otwo^t}
\newcommand{\Ooneck}{\Oone^{c-k}}
\newcommand{\Oonet}{\Oone^{t}}

\newcommand{\soplus}{\mathop{\mbox{\normalsize$\bigoplus$}}\limits}

\newcommand{\tablone}
{
\xy
(0,0)*{};(48,0)*{} **\dir{-};
(0,-6)*{};(0,6)*{} **\dir{-};
(8,-6)*{};(8,6)*{} **\dir{-};
(16,-6)*{};(16,6)*{} **\dir{-};
(24,-6)*{};(24,6)*{} **\dir{-};
(32,-6)*{};(32,6)*{} **\dir{-};
(40,0)*{};(40,6)*{} **\dir{-};
(48,0)*{};(48,6)*{} **\dir{-};
(0,6)*{};(48,6)*{} **\dir{-};
(0,-6)*{};(32,-6)*{} **\dir{-};
(44,3)*{\Ooneck};(36,3)*{\Ttk};(28,3)*{\Ttt};(20,3)*{2};(12,3)*{\Ott};(4,3)*{\Otk};
(28,-3)*{\Oonet};(20,-3)*{\Otwo};(12,-3)*{\OTtt};(4,-3)*{\OTtk};
\endxy
}

\newcommand{\tabltwo}
{
\xy
(0,0)*{};(40,0)*{} **\dir{-};
(0,-6)*{};(0,6)*{} **\dir{-};
(8,-6)*{};(8,6)*{} **\dir{-};
(16,-6)*{};(16,6)*{} **\dir{-};
(24,-6)*{};(24,6)*{} **\dir{-};
(32,0)*{};(32,6)*{} **\dir{-};
(40,0)*{};(40,6)*{} **\dir{-};
(0,6)*{};(40,6)*{} **\dir{-};
(0,-6)*{};(24,-6)*{} **\dir{-};
(36,3)*{\Ooneck};(28,3)*{\Ttk};(20,3)*{\Ttt};(12,3)*{\Ott};(4,3)*{\Otk};
(20,-3)*{\Oonet};(12,-3)*{\OTtt};(4,-3)*{\OTtk};
\endxy
}

\newcommand{\tablthree}
{
\xy
(0,0)*{};(32,0)*{} **\dir{-};
(0,-6)*{};(0,6)*{} **\dir{-};
(8,-6)*{};(8,6)*{} **\dir{-};
(16,0)*{};(16,6)*{} **\dir{-};
(24,0)*{};(24,6)*{} **\dir{-};
(32,0)*{};(32,6)*{} **\dir{-};
(0,6)*{};(32,6)*{} **\dir{-};
(0,-6)*{};(8,-6)*{} **\dir{-};
(28,3)*{\Ooneck};(20,3)*{\Tb};(12,3)*{\Okb};(4,3)*{\Ob};
(4,-3)*{\OTb};
\endxy
}

\newcommand{\tablas}
{
\xy
(0,-3)*{};(6,-3)*{} **\dir{-};
(0,-3)*{};(0,3)*{} **\dir{-};
(6,-3)*{};(6,3)*{} **\dir{-};
(0,3)*{};(6,3)*{} **\dir{-};
(3,0)*{\xs};
\endxy
}

\newcommand{\tablasd}
{
\xy
(0,-3)*{};(24,-3)*{} **\dir{-};
(6,-3)*{};(6,3)*{} **\dir{-};
(18,-3)*{};(18,3)*{} **\dir{-};
(0,-3)*{};(0,3)*{} **\dir{-};
(24,-3)*{};(24,3)*{} **\dir{-};
(0,3)*{};(24,3)*{} **\dir{-};
(3,0)*{x}; (12,0)*{\cdots}; (21,0)*{x};
\endxy
}
\begin{document}

\title[Sato--Tate distributions and symplectic characters]
{Auto-correlation functions of Sato--Tate distributions and identities of symplectic characters}

\author[K.-H. Lee]{Kyu-Hwan Lee}
\address{Department of Mathematics, University of Connecticut, Storrs, CT 06269, U.S.A.}
\email{khlee@math.uconn.edu}

\author[S.-j. Oh]{Se-jin Oh$^{\dagger}$}
\address{Department of Mathematics, Ewha Womans University, Seoul 120-750, South Korea}
\email{sejin092@gmail.com}
\thanks{$^{\dagger}$The research of S.-j.\ Oh was supported by the Ministry of Education of the Republic of Korea and the National Research Foundation of Korea (NRF-2019R1A2C4069647)}

\date{\today}

\begin{abstract}
The Sato--Tate distributions for genus $2$ curves (conjecturally) describe
the statistics of numbers of rational points on the curves. In this paper,
we explicitly compute the auto-correlation functions of Sato--Tate distributions for genus 2 curves  as sums of irreducible characters of symplectic groups.
Our computations bring about families of identities involving irreducible characters of symplectic groups $\Sp(2m)$ for all $m \in \Z_{\ge1}$,
which have interest in their own rights.
\end{abstract}

\maketitle

\section{Introduction}

The connection between random matrices and number theory was first conjectured by H. L. Montgomery \cite{Mont} in 1973 from the observation made by himself and F. J. Dyson that the two-point correlation function of the zeros of the Riemann zeta function is the same as the two-point correlation function of the eigenvalues of random matrices. Since then extensive research on correlation functions of $L$-functions and random matrices has been performed. For example, see \cite{BK,BK-1,Hej,KS,Rub,RS}.

In 2000, Keating and Snaith \cite{KeaSn,KeaSn-1} studied averages of characteristic polynomials of random matrices motivated in part by this connection to number theory and in part by the importance of these averages in quantum chaos \cite{AS}. Over the years it has become clear that averages of characteristic polynomials are fundamental for random matrix models, and many important developments have been made to the theory \cite{BDS,BS,BH,BH-2,FS,FS-1,MN}.
In particular, the auto-correlation functions of the distributions of
characteristic polynomials in the compact classical groups were computed by
Conrey, Farmer, Keating, Rubinstein and Snaith \cite{CFKRS,CFKRS-1} in connection with
conjectures for integral moments of zeta and $L$-functions, and by
Conrey, Farmer and Zirnbauer~\cite{CFZ,CFZ-1} in connection with conjectures
for ratios of $L$-functions. Later, Bump and Gamburd \cite{BG} obtained
different derivations of such formulae applying the symmetric function theory
and (analogues of) the dual Cauchy identity along with classical results due
to Weyl and Littlewood. Their results show that the auto-correlation functions
are actually combinations of characters of classical groups.

More precisely, Conrey, Farmer, Keating, Rubinstein and Snaith \cite{CFKRS} computed Selberg's integrals to obtain the following formula (and many other formulae):
\begin{equation} \label{eq-sp} \int_{\USp(2g)}  \left ( \prod_{j=1}^m \det (I+ x_j \gamma) \right ) d \gamma = \sum_{\lambda:\ \textrm{even}} s_\lambda (x_1, \cdots , x_m), \end{equation}
where  $s_\lambda$ is the Schur function and the sum is over partitions $\lambda=(\lambda_1, \dots , \lambda_m)$ with all parts $\lambda_j$ even and $2g \ge \lambda_1 \ge \cdots \ge \lambda_m \ge 0$.
The left-hand side of \eqref{eq-sp} is the average of the products of characteristic polynomials of random matrices in $\USp(2g)$ as $d\gamma$ is the Haar probability measure on $\USp(2g)$. Here, since $-\gamma \in \USp(2g)$ for $\gamma \in \USp(2g)$, we may use $\det (I+ x_j \gamma)$.

In a totally different way, Bump and Gamburd \cite{BG} could compute the same integral to obtain
\begin{equation} \label{eq-ssspp}
\int_{\USp(2g)}  \left ( \prod_{j=1}^m \det (I+ x_j \gamma) \right ) d \gamma = (x_1 \dots x_m)^g \, \chi^{\Sp(2m)}_{(g^m)} (x_1^{\pm 1}, \dots , x_m^{\pm 1}), \end{equation}
where $\chi^{\Sp(2m)}_{(g^m)}$ is the irreducible character of $\Sp(2m) \footnote{By this notation we mean $\Sp(2m, \mathbb C)$ or $\USp(2m)$ accodring to the context. We keep this ambiguity for notational convenience.}$ associated with the rectangular partition $(g^m)$.  The equality of the right-hand sides of \eqref{eq-sp} and \eqref{eq-ssspp} can be shown directly through branching of the character $\chi^{\Sp(2m)}_{(g^m)}$.
In fact, Bump and Gamburd used  an analogue of the dual Cauchy identity due to Jimbo--Miwa \cite{JM} and Howe \cite{Howe}:
\begin{equation} \label{eqn-111} \prod_{i=1}^m \prod_{j=1}^g (x_i +x_i^{-1}+t_j+t_j^{-1}) = \sum_{\lambda \trianglelefteq (g^m)} \chi_\lambda^{\Sp(2m)}(x_1^{\pm 1}, \dots , x_m^{\pm 1})
\chi_{\tilde \lambda}^{\Sp(2g)}(t_1^{\pm 1}, \dots , t_g^{\pm 1}) ,\end{equation}
where $\tilde \lambda=(m-\lambda'_g, \dots , m-\lambda'_1)$ with $\lambda' =(\lambda'_1, \dots , \lambda'_g)$ the transpose of $\lambda$. This identity can be considered as a reflection of Howe duality.

While most of the connections of number theory to random matrices are still conjectural, the celebrated Sato--Tate conjecture for elliptic curves is directly related to random matrices in nature and has been proved (under some conditions) by the works of R. Taylor, jointly with L. Clozel, M. Harris, and N. Shepherd-Barron \cite{CHT,T,HSBT}.
The conjecture concerns the distribution of Euler factors of an elliptic curve over a number field. It specifically predicts that this distribution always takes one of the three forms, one occurring whenever the elliptic curve fails to have complex multiplication, which is the generic case, and two exceptional cases arising for curves with complex multiplication, and furthermore says that all three distributions are the same as the distributions of eigenvalues of random matrices in the compact groups $\SU(2)$, $\U(1)$ and $N(\U(1))$, respectively, where $N(\U(1))$ is the normalizer of $\U(1)$ in $\SU(2)$.

This amazing structural randomness in arithmetic data is not
expected to be restricted to the case of elliptic curves. Indeed,
J.-P. Serre, N. Katz and P. Sarnak proposed a generalized Sato--Tate
conjecture for curves of higher genera  \cite{Ser,KS}. Pursuing this
direction, K. S. Kedlaya and A. V. Sutherland \cite{KS09} and later
together with F. Fit{\' e} and V. Rotger \cite{FKRS}  made a list of
55 compact subgroups of $\USp(4)$ called {\em Sato--Tate groups}
that would classify all the distributions of Euler factors for
abelian surfaces and showed that at most 52 of them can actually
arise from abelian surfaces\footnote{In \cite{FKS-1}, an additional
Sato--Tate axiom is included for abelian varieties of dimension $\le
3$, and it eliminates those $3$ groups which do not arise for
abelian surfaces.}.

For example, the hyperelliptic curve $y^2 = x^5 - x + 1$ over $\mathbb Q$ falls into the generic case and has the distribution given by the group $\USp(4)$, while the curve $y^2 = x^6 + 2$ over $\mathbb Q(\sqrt{-3})$ has the distribution given by the subgroup $\langle \U(1), \pmb{\zeta}_{12} \rangle$  of $\USp(4)$, where $\pmb{\zeta}_{12}$ is a primitive $12^{\mathrm{th}}$ root of unity (embedded into $\USp(4)$). Thus $\USp(4)$ and  $\langle \U(1), \pmb{\zeta}_{12} \rangle$ are examples of the Sato--Tate groups for genus $2$ curves.

 Through enormous amount of computer computation \cite{FKRS}, they found abelian surfaces whose Euler factors have the same distributions as the distributions of characteristic polynomials of the $52$ Sato--Tate groups.
 Since then, the Sato--Tate conjecture for abelian surfaces defined over $\mathbb Q$, which covers 34 Sato--Tate groups, has been established by C. Johansson and N. Taylor \cite{Joh17, Tay20} except for the generic case $\USp(4)$.
Actually, those three groups, which do not arise for abelian surfaces, and two other subgroups of $\USp(4)$ appear when certain motives of weight $3$ are considered in \cite{FKS}. Thus all the 57 groups are interesting in number theory and arithmetic geometry, and we will call all of them Sato--Tate groups in what follows.

\medskip

In this paper, we describe the distributions of characteristic polynomials of random matrices in Sato--Tate groups $H \le \USp(4)$ by computing  the auto-correlation functions
\begin{equation} \label{eq-sp-1} \int_{H}  \prod_{j=1}^m \det (I+ x_j \gamma)d \gamma, \qquad m \in \mathbb Z_{\ge 1},
\end{equation}
as sums of irreducible characters of $\Sp(2m)$. We consider all
the 57  Sato--Tate groups $H \le \USp(4)$, and obtain the following result.

\begin{theorem} \label{thm-genus-2-1}
Let $H \le \USp(4)$ be a Sate--Tate group. Then, for any $m \in \mathbb Z_{\ge 1}$, we have
\begin{align*}
\int_{H}  \prod_{j=1}^m \det (I+ x_j \gamma)  d
\gamma & =  (x_1 \cdots x_m)^2 \sum_{b=0}^{m} \sum_{z=0}^{\lfloor \frac {m-b} 2 \rfloor} \mathfrak m_{(b+2z,b)} \chi_{(2^{m-b-2z}, 1^{2z})}^{\Sp(2m)} ,
\end{align*}
where the coefficients $\mathfrak m_{(b+2z,b)}$ are the multiplicities of the trivial representation in the restrictions $\chi^{\Sp(4)}_{(b+2z,b)}\big |_H$ and are explicitly given in Table \ref{tab-1}.
\end{theorem}

As the statement of the theorem manifests, this result can also be interpreted as a result on branching rules in representation theory. Indeed, our approach starts with the identity \eqref{eqn-111} and converts the problem to the branching rules of irreducible characters of $\USp(4)$ restricted to Sato--Tate groups $H$. Though a few of the Sato--Tate groups invoke classical branching rules, almost all of them call for a new study because they are disconnected and involve twists by automorphisms.

In classical situations, branching rules invoke many interesting combinatorial questions (e.g. \cite{King,KT}).
With Sato--Tate groups, we also meet various combinatorial problems, and adopt {\em crystals} as our main combinatorial tools for the problems. One can find details about crystals, for example, in \cite{BuSch, HK, Kash}.
However, in some cases, we need more concrete realizations of representations of $\USp(4)$.

\medskip

Moreover, since most of the Sato--Tate groups are disconnected, we can decompose the integral
\eqref{eq-sp-1} according to coset decompositions, and find that the characteristic polynomials over some cosets are \emph{independent} of the elements of the cosets. Combining this observation with the computations of branching rules, we obtain families of non-trivial identities of irreducible characters of $\Sp(m)$ for all $m \in \Z_{\ge1}$ as follows.

\begin{theorem} \label{thm-identities-1}
Let $ \kappa_1=-2 , \kappa_2=2, \kappa_3=1, \kappa_4=0, \text{ and } \kappa_6=-1$.
Then,
for any $m \in \mathbb Z_{\ge 1}$ and $n=1,2,3,4,6$, we have
\begin{align*} \tag{1}
\prod_{i=1}^m (x_i^2+\kappa_n+x_i^{-2}) &= \sum_{b=0}^m \sum_{z=0}^{\lfloor \frac {m-b} 2 \rfloor} \psi_n(z,b) \chi_{(2^{m-b-2z}, 1^{2z})}^{\Sp(2m)} ,
\end{align*}
where $\psi_n(z,b) \in \mathbb Z$ are respectively defined in \eqref{def-psi-1} and  \eqref{def-psi-2} and in Table \ref{tab-3};
\begin{align*} \tag{2}
\sum_{k=0}^m \binom{m-k}{\lfloor (m-k)/2 \rfloor}
\sum_{1 \le i_1 < \cdots <i_k \le m} \prod_{j=1}^k (x_{i_j}^2+x_{i_j}^{-2})
&=\sum_{\ell=0}^{\lfloor \frac m 2 \rfloor} \sum_{z=0}^{\lfloor \frac m 2 \rfloor-\ell}  (-1)^z \chi_{(2^{m-2\ell-2z}, 1^{2z})}^{\Sp(2m)},
\end{align*}
where the summation over $i_1< \dots < i_k$ is set to be equal to $1$ when $k=0$;
\begin{align*} \tag{3}
\prod_{i=1}^m (x_i+x_i^{-1})  \sum_{j=0}^{\lfloor \frac m 2 \rfloor} \chi_{(1^{m-2j})}^{\Sp(2m)}
&=\sum_{b=0}^m \sum_{z=0}^{\lfloor \frac {m-b} 2 \rfloor} \xi_2(z,b) \chi_{(2^{m-b-2z}, 1^{2z})}^{\Sp(2m)},
\end{align*}
where $\xi_2(z,b) \in \mathbb Z$ is defined in Table \ref{tab-4}.
\end{theorem}

Notice that the irreducible characters $\chi_\lambda^{\Sp(2m)}$ are symmetric functions with the number of terms growing very fast as $m$ increases,
but that the coefficients (e.g. $\psi_n(z,b)$) are independent of $m$.
These identities seem intriguing from the viewpoint of representation theory and algebraic combinatorics.
Without the motivation coming from the Sato--Tate distributions which led to the computations in this paper,
it might have been difficult for us to expect that such identities exist.


\subsection{Organization of the paper}

In Section \ref{sec-stg}, we present backgrounds for Sato--Tate groups. In Section
\ref{sec-dci},
the dual Cauchy identity for symplectic groups will be used to convert computation of auto-correlation functions into that of branching rules of $\USp(4)$.
Section \ref{sec-g-1} is devoted to
the genus one case. This case will demonstrate basic ideas which apply to the genus two case. In Section \ref{sec-g-2}, the main theorems are stated and proved by considering each of the Sato--Tate groups using the results of Section \ref{sec-branching}, where we study branching rules for Sato--Tate groups through crystals and other methods.

\begin{convention}
Throughout this paper, we keep the following conventions.
\begin{enumerate}
\item[{\rm (i)}] For a statement $P$, $\delta(P)$ is equal to $1$ or $0$ according to whether $P$ is true or not.
\item[{\rm (ii)}] For a partition $\lambda$, we denote by $\lambda'=(\lambda'_1 \ge \ldots \ge\lambda'_g \ge 0)$  the transpose of $\lambda$.
\item[{\rm (iii)}]  The term for $k=0$ in a summation over $1 \le i_1< \dots < i_{k}\le m$ is set to be equal to $1$.
\item[{\rm (iv)}]  For $m,m' \in \Z$, we write $m\equiv_k m'$ if $k$ divides $m-m'$, and $m \not\equiv_k m'$ otherwise.
\end{enumerate}
\end{convention}

\subsection*{Acknowledgments} We are very grateful to Daniel Bump for his help and guidance for this project, and thank Seung Jin Lee for useful discussions. We also thank Andrew Sutherland for his helpful comments.

\section{Sato--Tate groups} \label{sec-stg}

In this section, we briefly overview backgrounds of Sato--Tate groups and their auto-correlation functions of characteristic polynomials. More details can be found in \cite{KS09,FKRS}.

\medskip
Let $C$ be a smooth, projective, geometrically irreducible algebraic curve of genus $g$ defined over $\mathbb Q$. 
For each prime $p$ where $C$ has good reduction,
we define the zeta
function $Z(C/\mathbb F_p; T)$ by
\[ Z(C/\mathbb F_p; T) = \exp \left ( \sum_{k=1}^\infty N_kT^k/k \right ) , \]
where $N_k$ is the number of the points on $C$ over $\mathbb F_{p^k}$. It is well-known \cite{Art} that $Z(C /\mathbb F_p;T)$ is a rational function of the form
\[ Z(C/\mathbb F_p; T) = \frac{L_p(T)}{(1-T)(1-pT)} , \] where $L_p \in \mathbb Z[T]$ is a  polynomial of degree $2g$ with constant term $1$. For example, when $C$ is an elliptic curve, i.e. when $g=1$, we have $L_p(T)=1-a_pT+pT^2$ and $L_p(1)$ is equal to the number of points on $C$ over $\mathbb F_p$.

Set $\bar L_p(T):=L_p(p^{-1/2}T)$ and write
\[ \bar{L}_p(T)=T^{2g}+a_{1,p}T^{2g-1}+ a_{2,p}T^{2g-2}+ \cdots + a_{2,p} T^2+a_{1,p}T+1. \]
Let $P_C(N)$ be the set of primes $p \le N$ for which the curve $C$ has good reduction.
\begin{definition}
For $1\le k \le g$ and $m\ge 0$, define $a_k(m;g)$ to be the average value of $a_{k,p}^m$ over $p \in P_C(N)$ as $N \rightarrow \infty$.
\end{definition}
The values $a_k(m;g)$, $m \ge 0$, are  the $m^{\mathrm{th}}$ moments of the distribution of $a_{k,p}$, and we are interested in how to describe $a_k(m;g)$.
The generalized Sato--Tate conjecture expects  that curves of fixed genus $g$ are classified into certain families and that  $a_k(m;g)$ are all the same for curves in each family. In particular, there is a generic family of curves for each genus $g$.

Let $\ell$ be a prime and $T_\ell(C)$ be the {\em Tate module}, i.e., the inverse limit of the $\ell^n$-torsion subgroups ($n \in \mathbb Z_{\ge 1}$) of the Jacobian $J(C)$ of $C$. Then we obtain the representation
\[ \rho_\ell : \mathrm{Gal}(\overline {\mathbb Q}/\mathbb Q) \rightarrow \mathrm{Aut}(T_\ell(C)) \cong \GL(2g, \mathbb Z_\ell) .\]
We say that the curve $C$ has {\em large Galois image} if the image of $\rho_\ell$ is Zariski dense in $\GSp(2g,\mathbb Z_\ell) \subset \GL(2g,\mathbb Z_\ell)$ for any $\ell$. The curves with large Galois image form the generic family of curves with a fixed genus $g$.

\begin{example} \label{ex-ST}
The following curves are from the generic families (\cite{KS09,FKRS}).
\begin{itemize}
 \item $g=1$, \quad $C$: $y^2=x^3+x+1$
\[a_1(m;1):\ 1,0,1,0,2,0,5,0,14,0,42,0,132, \dots  \text{(Catalan numbers)} \]
 \item $g=2$, \quad $C$: $y^2=x^5-x+1$
\begin{align*}
 a_1(m;2) :&\quad  1,0,1,0,3,0,14,0,84,0,594,0,4719, \dots \\
a_2(m;2):&\quad  1,1,2,4,10,27,82,268,940, \dots
\end{align*}
 \item $g=3$, \quad $C$: $y^2=x^7-x+1$
\begin{align*}
a_1(m;3): &\quad  1,0,1,0,3,0,15,0,104,0,909,0,9449, \dots \\
a_2(m;3): &\quad  1,1,2,5,16,62,282,1459,8375, \dots \\
a_3(m;3): & \quad 1,0,2,0,23,0, 684,0, 34760, \dots
\end{align*}
\end{itemize}
\end{example}

The generalized Sato--Tate conjecture predicts that these distributions are actually the same as the distributions of eigenvalues of random matrices. To be precise, let us consider  the group $\USp(2g)$ with the Haar probability measure.  Let \[ \det (I - x\gamma)=x^{2g}+c_{1}x^{2g-1}+ c_{2}x^{2g-2}+ \cdots + c_{2} x^2+c_{1}x+1\] be the characteristic polynomial of a random matrix $\gamma$ of $\USp(2g)$.
\begin{definition}
For each $k=1,2, \dots, g$, let $X_k$ be the random variable corresponding to the coefficient $c_k$ and define $c_k(m;g)$ to be the $m^{\rm{th}}$ moment $\mathbf E[X_k^m]$, $m \in \mathbb Z_{\ge 0}$, of the random variable $X_k$.
\end{definition}
The following is the generalized Sato--Tate conjecture for the case that $C$ is in the generic family.

\begin{conjecture}[\cite{KS}] \label{conj-ST}
Let $C$ be a smooth projective curve of genus $g$.
Assume that $C$ is in the generic family.
 Then, for each $k=1, 2, \dots, g$ and $m \ge 0$, we have
 \[ a_k(m;g)= c_k(m;g). \]
\end{conjecture}

Therefore it is important to compute the distribution of the characteristic polynomials of random matrices in $\USp(2g)$. Ultimately, the auto-correlation functions
\begin{equation} \label{eqn-Usp}  \int_{\USp(2g)}  \left ( \prod_{j=1}^m \det (I+ x_j \gamma) \right ) d \gamma \end{equation}
for $m \in \mathbb Z_{\ge 0}$ will describe the distribution completely.  The number $c_k(m;g)$ will appear as the coefficient of $(x_1 \cdots x_m)^k$ in \eqref{eqn-Usp}.

As mentioned in the previous section, Conrey--Farmer--Keating--Rubinstein--Snaith \cite{CFKRS} and
Bump--Gamburd \cite{BG} obtained
\begin{equation}  \label{eqn-done-1} \int_{\USp(2g)}  \left ( \prod_{j=1}^m \det (I+ x_j \gamma) \right ) d \gamma = (x_1 \dots x_m)^g \, \chi^{\Sp(2m)}_{(g^m)} (x_1^{\pm 1}, \dots , x_m^{\pm 1}), \end{equation}
where $\chi^{\Sp(2m)}_{(g^m)}$ is the irreducible character of $\Sp(2m)$ associated with the rectangular partition $(g^m)$.
By computing the coefficients of $(x_1 \cdots x_m)^k$ in the right-hand side of \eqref{eqn-done-1}, which are nothing but weight multiplicities, one can check \[ a_k(m;g)= c_k(m;g) \] for the sequences $a_k(m;g)$ in Example \ref{ex-ST}. Thus we see  validity of the generalized Sato--Tate conjecture for the curves in the example.

Aside from the generic family of curves whose distribution is (expected to be) given by $\USp(2g)$, there are exceptional families of curves. For elliptic curves, there are two exceptional families (only one over $\mathbb Q$) which consist of elliptic curves with complex multiplication. The Sato--Tate conjecture, which is proven much earlier  for these exceptional families of $g=1$ \cite{Deu}, tells us that the moment sequences $a_1(m;1)$ are the same as those of $N(\U(1))$ and $\U(1)$, respectively. Here $N(\U(1))$ is the normalizer of $\U(1)$ in $\SU(2) \cong  \USp(2)$.

For genus $2$ curves, there are a lot more of exceptional families. Kedlaya and Sutherland \cite{KS09} and later with Fit{\' e} and  Rotger \cite{FKRS} made a conjectural, exhaustive list of 55 compact subgroups of $\USp(4)$ that would classify all the distributions of Euler factors for abelian surfaces, and called the groups {\em Sato--Tate groups}. Later, when they considered certain motives of weight $3$ \cite{FKS}, two other groups were added to the list of Sato--Tate groups that are subgroups of $\USp(4)$. They determined the moment sequences $c_k(m;2)$, $k=1,2$, for each Sato--Tate group by expressing them as combinations of some sequences. In the process they investigated a huge number of abelian surfaces to see that Euler factors have the same distributions as the Sato--Tate distributions, supporting their refined, generalized Sato--Tate conjecture for abelian surfaces.

\section{Dual Cauchy identity} \label{sec-dci}

In this section, we use the dual Cauchy identity for symplectic groups to convert computation of auto-correlation functions into that of branching rules.

\medskip

We recall the dual Cauchy identity for symplectic groups:

\begin{proposition}[ \cite{JM, Howe}]
For $m, g \in \mathbb Z_{\ge 1}$, we have
\begin{equation} \label{eqn-1} \prod_{i=1}^m \prod_{j=1}^g (x_i +x_i^{-1}+t_j+t_j^{-1}) = \sum_{\lambda \trianglelefteq (g^m)} \chi_\lambda^{\Sp(2m)}(x_1^{\pm 1}, \dots , x_m^{\pm 1})
\chi_{\tilde \lambda}^{\Sp(2g)}(t_1^{\pm 1}, \dots , t_g^{\pm 1}) ,\end{equation}
where we set
$$\tilde \lambda \seteq (m-\lambda'_g, \dots , m-\lambda'_1)$$  with $\lambda' =(\lambda'_1, \dots , \lambda'_g)$ the transpose of $\lambda$.
\end{proposition}

Let $H$ be a compact  subgroup of $\USp(2g)$. For each $m \in \mathbb Z_{\ge 1}$, we want to compute the auto-correlation of distribution of characteristic polynomials of $H$:
\[ \int_{H} \prod_{i=1}^m \det (I+ x_i \gamma)  d \gamma .\]

Fix $g$ for the time being. For $\gamma \in \USp(2g)$ with eigenvalues $t_1^{\pm 1}, \dots , t_g^{\pm 1}$,   
we have
\[ \prod_{i=1}^m \det(I+x_i \gamma) = \prod_{i=1}^m \prod_{j=1}^g (1+x_i t_j)(1+x_i t_j^{-1}) .\]
We also write $\chi_{\tilde \lambda}^{\Sp(2g)} (\gamma) =\chi_{\tilde \lambda}^{\Sp(2g)}(t_1^{\pm 1}, \dots , t_g^{\pm 1})$ for simplicity of notations.

\begin{definition} \label{def-m-lambda}
Let $\mathfrak m_{\tilde \lambda}(H)$ denote the multiplicity of the trivial representation $1_H$ of $H$ in the restriction of $\chi_{\tilde \lambda}^{\Sp(2g)}$ to $H$.  That is,
\[ \chi_{\tilde \lambda}^{\Sp(2g)} |_H = \mathfrak m_{\tilde \lambda}(H)\, 1_H + \text{sum of nontrivial irreducible characters of $H$} .\]
\end{definition}

The following proposition shows that the decomposition multiplicities $\mathfrak m_{\tilde \lambda}(H)$ are the coefficients of character expansion of auto-correlation functions.

\begin{proposition} \label{prop-main}
For each $m \in \mathbb Z_{\ge 1}$, the auto-correlation function of the distribution of characteristic polynomials of $H$ is given by
\[\int_{H} \prod_{i=1}^m \det (I+ x_i \gamma)  \, d \gamma = (x_1 \cdots x_m)^{g} \sum_{\lambda \trianglelefteq (g^m)} \mathfrak m_{\tilde \lambda}(H)\, \chi_\lambda^{\Sp(2m)}(x_1^{\pm 1}, \dots , x_m^{\pm 1}). \]
\end{proposition}

\begin{proof}
Since we have
\[ (1+x_i t_j)(1+x_i t_j^{-1})=x_i (x_i +x_i^{-1}+t_j+t_j^{-1}),\]
it follows from the dual Cauchy identity \eqref{eqn-1} that
\begin{align*} 
&  (x_1 \cdots x_m)^{-g} \int_{H}  \prod_{j=1}^m \det (I+ x_j \gamma)  d
\gamma = (x_1 \cdots x_m)^{-g} \int_{H}  \prod_{i=1}^m \prod_{j=1}^g (1+x_i t_j)(1+x_i t_j^{-1})  \allowdisplaybreaks \\
&= \int_{H} \prod_{i=1}^m \prod_{j=1}^g (x_i +x_i^{-1}+t_j+t_j^{-1}) d\gamma = \int_{H} \sum_{\lambda \trianglelefteq (g^m)} \chi_\lambda^{\Sp(2m)}(x_1^{\pm 1}, \dots , x_m^{\pm 1})
\chi_{\tilde \lambda}^{\Sp(2g)}(\gamma)  d\gamma  \nonumber \allowdisplaybreaks \\
&=  \sum_{\lambda \trianglelefteq (g^m)}  \chi_\lambda^{\Sp(2m)}(x_1^{\pm 1}, \dots , x_m^{\pm 1}) \int_{H}
\chi_{\tilde \lambda}^{\Sp(2g)}(\gamma)  d\gamma .   \nonumber
\end{align*}
From Schur orthogonality (for example, \cite{Bu}), the integral $\int_{H}
\chi_{\tilde \lambda}^{\Sp(2g)}(\gamma)  d\gamma$
is equal to the multiplicity of the trivial representation $1_H$ of $H$ in the restriction of $\chi_{\tilde \lambda}^{\Sp(2g)}$ to $H$, which is  $\mathfrak m_{\tilde \lambda}(H)$ by definition.
\end{proof}

When $H$ is clear from the context, we will simply write $\mathfrak m_{\tilde \lambda}$ for $\mathfrak m_{\tilde \lambda}(H)$. For simplicity, we also write
\[ \chi^{\Sp(2m)}_{\lambda} =\chi^{\Sp(2m)}_{\lambda} (x_1^{\pm 1}, \dots , x_m^{\pm 1}) .\]

As a special case, we obtain the identity \eqref{eq-ssspp} where $H$ is equal to the generic Sato--Tate group $\USp(2g)$.
\begin{corollary}[\cite{BG}]\label{cor-generic}  When $H=\USp(2g)$, we obtain
\begin{equation*}  
\int_{\USp(2g)}  \prod_{j=1}^m \det (I+ x_j \gamma) \, d \gamma = (x_1 \dots x_m)^g \, \chi^{\Sp(2m)}_{(g^m)} (x_1^{\pm 1}, \dots , x_m^{\pm 1}).
\end{equation*}
\end{corollary}

\begin{proof}
Since $H=\USp(2g)$, 
we have $\mathfrak m_{\tilde \lambda}(H)=0$ unless $\chi_{\tilde \lambda}^{\Sp(2g)}$ itself is trivial. If $\chi_{\tilde \lambda}^{\Sp(2g)}$ is trivial, we get $\tilde \lambda=\emptyset$, $\mathfrak m_{\tilde \lambda}(H)=1$ and $\lambda = (g^m)$.
\end{proof}

\section{Prototype: case $g=1$} \label{sec-g-1}

In this section, we compute $\mathfrak m_{\tilde \lambda}(H)$ for non-generic Sato--Tate groups $H \lneqq \USp(2)$ when $g=1$. The computations will demonstrate our approach which extends to the case $g=2$ in Sections \ref{sec-g-2} and \ref{sec-branching}.

\medskip

Let $\U(1)=\{ u \in \mathbb C^\times : |u|=1 \}$ be the circle group. We embed $\U(1)$ into $\USp(2)$ by \[u \longmapsto \mathrm{diag}(u,u^{-1}). \] Thus $\U(1)$ is a (maximal) torus of $\USp(2)$.
Under this embedding the non-generic Sato--Tate groups are $\U(1)$ and its normalizer $N(\U(1))$.

The auto-correlation functions for $g=1$ are explicitly given in the following theorem. 

\begin{theorem} \label{thm-genus-1}
For each $m \in \mathbb Z_{\ge 1}$, we have
\begin{align}  \int_{H}  \prod_{j=1}^m \det (I+ x_j \gamma)  d
\gamma & =  (x_1 \cdots x_m) \sum_{j=0}^{\lfloor \frac m 2 \rfloor} \mathfrak m_{(2j)} \chi_{(1^{m-2j})}^{\Sp(2m)} \label{eqn-re-11} ,
\end{align}
where
\begin{center}{\small
\begin{tabular}{ |c|c| }
\hline $H$ & $\mathfrak m_{(2j)}$  \\ \hline \vspace*{- 0.35 cm} & \\ \vspace*{- 0.35 cm}
$\USp(2)$ &
$\delta(j =0)$\\ & \\
\vspace*{- 0.35 cm}
$\U(1)$ & $1$  \\ & \\ 
$N(\U(1))$ & $\delta(j \equiv_2 0)$
 \\
 \hline
\end{tabular} }.
\end{center}
\end{theorem}

\begin{proof}
If $\lambda \trianglelefteq (1^m)$, we can write
\[ \lambda = (1^k)\quad \text{ and } \quad \tilde \lambda=(m-k) \quad \text{ for } \ 0 \le k \le m. \]

First, assume that $H=\U(1)$. Let $v_1=(1,0)$ and $v_2=(0,1)$ be the standard unit vectors of $V:=\mathbb C^2$, and consider the standard representation of $\Sp(2, \mathbb C)$ on $V$. Consider the irreducible representation $\mathrm{Sym}^{m-k}(V)$ of $\Sp(2, \mathbb C)$ with the character $\chi_{(m-k)}^{\Sp(2)}$ of degree $m-k+1$. Then the trivial $\U(1)$-module is generated by $v_1^j v_2^j$ only when $m-k=2j$.
Thus the restriction  $\chi_{(m-k)}^{\Sp(2)}|_{\U(1)}$ has the trivial character with multiplicity $1$ if and only if $m-k$ is even. That is,
\[ \mathfrak m_{(m-k)} = \begin{cases} 1 & \text{ if $m-k$ is even}, \\ 0 & \text{ otherwise.} \end{cases} \]
Thus it follows from Proposition \ref{prop-main} that
\begin{align}  \int_{\U(1)}  \prod_{j=1}^m \det (I+ x_j \gamma)  d
\gamma & =  (x_1 \cdots x_m) \sum_{k=0}^{m} \mathfrak m_{(m-k)} \chi_{(1^{k})}^{\Sp(2m)} = (x_1 \cdots x_m) \sum_{j=0}^{\lfloor \frac m 2 \rfloor} \chi_{(1^{m-2j})}^{\Sp(2m)} , \label{eqn-u1-1} \end{align}
where we change the indices by $m-k=2j$. This proves \eqref{eqn-re-11} for $H=\U(1)$.

Next, assume that  $H=N(\U(1))$.
Let $J={\scriptsize \begin{pmatrix} 0 & 1 \\ -1 & 0 \end{pmatrix}} \in \USp(2)$. Then \[ N(\U(1)) = \U(1) \bigsqcup J\U(1), \]
and $Jv_1 = -v_2$ and $Jv_2= v_1$ on the standard representation of $\Sp(2, \mathbb C)$ on $V$. Consider again the irreducible representation $\mathrm{Sym}^{m-k}(V)$ of $\Sp(2, \mathbb C)$. As noted above, the trivial $\U(1)$-module is generated by $v_1^j v_2^j$ when $m-k=2j$, and $J$ acts trivially on $v_1^j v_2^j$ if and only if $j$ is even.
Consequently, the restriction  $\chi_{(m-k)}^{\Sp(2)}|_{N(\U(1))}$ has
\begin{equation*} 
\mathfrak m_{(m-k)} = \delta( m-k \equiv_4 0).
\end{equation*}

Now it follows from Proposition \ref{prop-main} that
\begin{align}  \int_{N(\U(1))}  \prod_{j=1}^m \det (I+ x_j \gamma)  d
\gamma & =  (x_1 \cdots x_m) \sum_{k=0}^{m} \mathfrak m_{(m-k)} \chi_{(1^{k})}^{\Sp(2m)} = (x_1 \cdots x_m) \sum_{\ell=0}^{\lfloor \frac m 4 \rfloor} \chi_{(1^{m-4 \ell})}^{\Sp(2m)} , \label{eqn-u1-2} \end{align}
where we change the indices by $m-k=4 \ell$. This proves \eqref{eqn-re-1} for $H=N(\U(1))$.

Together with Corollary \ref{cor-generic}, we have completed a proof.
\end{proof}

As corollaries, we obtain the following identities which are interesting in their own rights.

\begin{proposition} \label{prop-first-identity}
For each $m \in \mathbb Z_{\ge 1}$, we have the following identities:
\begin{align}
& \sum_{\ell =0}^{\lfloor m/2 \rfloor} \binom{2\ell}{\ell} \sum_{1 \le i_1 < \cdots <i_{m-2\ell} \le m}  \prod_{j=1}^{m-2\ell} (x_{i_j}+ x_{i_j}^{-1}) = \sum_{j=0}^{\lfloor \frac m 2 \rfloor}
\chi_{(1^{m-2j})}^{\Sp(2m)}, \label{gen-1-id-1} \\
& \prod_{i=1}^m (x_i+x_i^{-1}) =  \sum_{j=0}^{\lfloor \frac m 2 \rfloor} (-1)^j \chi_{(1^{m-2j})}^{\Sp(2m)} \label{gen-1-id-2}.
\end{align}
\end{proposition}

\begin{proof}
Since $\int_{\U(1)} u^k du = \delta(k=0)$ for $k \in \mathbb Z$, we have
\begin{align}
&\int_{\U(1)} \prod_{i=1}^m (x_i+x_i^{-1}+(u+u^{-1})) du \nonumber \allowdisplaybreaks \\
& = \int_{\U(1)} \sum_{k=0}^m \sum_{1 \le i_1 < \cdots <i_k \le m}  \prod_{j=1}^k (x_{i_j}+ x_{i_j}^{-1})(u+u^{-1})^{m-k}du \nonumber \allowdisplaybreaks\\
&=\sum_{k=0}^m \sum_{1 \le i_1 < \cdots <i_k \le m}  \prod_{j=1}^k (x_{i_j}+ x_{i_j}^{-1}) \delta( m \equiv_2 k) \binom{m-k}{(m-k)/2}\nonumber \allowdisplaybreaks\\
\label{uu-1} & =\sum_{\ell =0}^{\lfloor m/2 \rfloor} \binom{2\ell}{\ell} \sum_{1 \le i_1 < \cdots <i_{m-2\ell} \le m}  \prod_{j=1}^{m-2\ell} (x_{i_j}+ x_{i_j}^{-1}),\end{align}
where we put $m-k=2 \ell$ for the last equality. Since $\det(I+ x\gamma)= 1+ (u+u^{-1})x+x^2$ for
$\gamma = {\tiny \begin{pmatrix} u &0 \\ 0 & u^{-1} \end{pmatrix}}$, the identity \eqref{gen-1-id-1} follows from Theorem \ref{thm-genus-1}.

For any $\gamma \in J \U(1)$, we compute to see that
\[  \det (I+x\gamma)= 1+x^2 . \]
Thus we have
\begin{align*}
 \int_{N(\U(1))}  \prod_{i=1}^m \det (I+ x_i \gamma)  \, d
\gamma & = \frac 1 2 \int_{\U(1)}  \prod_{i=1}^m \det (I+ x_i \gamma_1) \, d
\gamma_1 + \frac 1 2  \int_{\U(1)}  \prod_{i=1}^m \det (I+ x_iJ \gamma_1)  \, d
\gamma_1, \\
& =  \frac 1 2 \int_{\U(1)}  \prod_{i=1}^m \det (I+ x_i \gamma) \, d
\gamma_1 + \frac 1 2  \prod_{i=1}^m (1+x_i^2) \\
&=  \frac 1 2 (x_1 \cdots x_m) \sum_{j=0}^{\lfloor \frac m 2 \rfloor} \chi_{(1^{m-2j})}^{\Sp(2m)} + \frac 1 2  \prod_{i=1}^m (1+x_i^2),
\end{align*}
where $d\gamma_1$ is the probability Haar measure on $\U(1)$ and we use \eqref{eqn-u1-1} for the last equality.
By comparing with \eqref{eqn-u1-2},  we obtain
\[  \prod_{i=1}^m (1+x_i^2) =(x_1 \cdots x_m) \sum_{j=0}^{\lfloor \frac m 2 \rfloor} (-1)^j \chi_{(1^{m-2j})}^{\Sp(2m)}   .\]
Diving both sides by $x_1 \cdots x_m$, we obtain the desired identity \eqref{gen-1-id-2}.
\end{proof}

\section{Main results: case $g=2$} \label{sec-g-2}

After fixing notations for Sato--Tate groups, we present the first main theorem of this paper and go over its proof. Much of the computations of branching rules involving crystals will be performed in Section \ref{sec-branching} though we use the results of the branching rules in this section.
In the process we will decompose Sato--Tate groups into cosets and prove various identities of irreducible characters of $\Sp(2m)$ for all $m \ge 1$, which form another set of main results in this paper.

\smallskip

We will adopt the same notations for the Sato--Tate groups as in \cite{FKRS}. To make this paper more self-contained, we recall the definitions of these groups. We take the group $\USp(4)$ to fix the symplectic form $\begin{pmatrix} 0 & I_2 \\ -I_2 & 0 \end{pmatrix}$, where $I_2$ is the $2 \times 2$ identity matrix. Let $E_{ij}$  be the $4 \times 4$ elementary matrix which has  $(i,j)$-entry equal to $1$ and other entries equal to $0$.
We fix a basis for the Lie algebra $\mathfrak{sp}_4 (\mathbb C)$:
\begin{align*}
& e_1=E_{12}-E_{43}, && f_1=E_{21}-E_{34}, && h_1=E_{11}-E_{22}-E_{33}+E_{44},\\
&e_2=E_{24}, &&  f_2=E_{42}, &&h_2=E_{22}-E_{44}.
\end{align*}
Set
\begin{align*}
&\hat{e}_1=E_{13}, && \hat{f}_1=E_{31},&& \hat{h}_1=E_{11}-E_{33},\\
&\hat{e}_2=e_2=E_{24}, && \hat{f}_2=f_2=E_{42},   &&\hat{h}_2=h_2=E_{22}-E_{44} .\end{align*}
Define weights $\epsilon_i$, $i=1,2$, by
$ \epsilon_i(\hat h_j) = \delta_{ij}$.
Then the simple roots $\mathfrak{sp}_4 (\mathbb C)$ are
\[ \alpha_1 = \epsilon_1 - \epsilon_2, \quad \alpha_2 = 2 \epsilon_2 ,\]
and the fundamental weights are
\[ \varpi_1 = \epsilon_1, \quad \varpi_2 = \epsilon_1+\epsilon_2 .\]
A pair of non-negative integers $(a,b)$ with $a \ge b$, or a partition $(a,b)$ of length $\le 2$ will be considered as a weight corresponding to $a \epsilon_1 + b \epsilon_2$.

We embed $\U(1)$ into $\USp(4)$ by
\begin{equation*} 
 u \longmapsto \mathrm{diag}(u,u,u^{-1},u^{-1}) ,
\end{equation*}
and $\SU(2)$ and $\U(2)$ into $\USp(4)$ by
\begin{equation}\label{embed-su2} A \longmapsto \begin{pmatrix} A&0  \\ 0& \overline{A} \end{pmatrix},\end{equation}
where $\overline A$ consists of the complex conjugates of the entries of $A$.

We fix an embedding \begin{equation} \label{embed-su2-su2} \SU(2) \times \SU(2) \hookrightarrow \USp(4)\end{equation} in such a way that the induced Lie algebra embedding $\mathfrak {sl}_2(\mathbb C) \times \mathfrak{sl}_2(\mathbb C) \rightarrow \mathfrak {sp}_4(\mathbb C)$ gives
\[ (h,0) \longmapsto \hat{h}_1 \quad \text{ and }\quad (0,h) \longmapsto \hat{h}_2 .\]
From this, we also induce an embedding \[\U(1)\times \SU(2) \hookrightarrow \USp(4).\]

Identify $\SU(2)$ with the group of unit quaternions via the isomorphism
\[  a+ b \, \pmb{\rm i} + c \, \pmb{\rm j} + d \, \pmb{\rm k}  \mapsto \begin{pmatrix} a+bi & c+di \\ -c+di & a-bi \end{pmatrix},  \quad a,b,c,d \in \mathbb R,   \]
and also identify them with the corresponding elements in $\USp(4)$ through the embedding $\SU(2) \hookrightarrow \USp(4)$ in \eqref{embed-su2}. For example, with this identification, we have
$ \mathbf j = \scriptsize { \begin{pmatrix} 0 & 1 & 0 & 0 \\ -1 & 0 & 0&0 \\ 0 &0&0&1\\ 0& 0&-1&0 \end{pmatrix} }$.

Let $J = \scriptsize { \begin{pmatrix} 0 & 0 & 0 & 1 \\ 0 & 0 & -1&0 \\ 0 &-1&0&0\\ 1& 0&0&0 \end{pmatrix} }$. We write
$\pmb{\zeta}_{2n}= \begin{pmatrix} e^{\pi i/n} & 0 \\ 0 & e^{-\pi i/n} \end{pmatrix} \in \SU(2)$, and its embedded image in $\USp(4)$ will also be written as $\pmb{\zeta}_{2n}$.
Set \begin{align} Q_1&=\{ \pm 1, \pm \mathbf i, \pm \mathbf j, \pm \mathbf k, \tfrac 1 2 ( \pm 1 \pm \mathbf i \pm \mathbf j \pm \mathbf k \} ),\label{q1} \\
\label{q2}
Q_2 &=\left \{ \tfrac 1 {\sqrt 2} ( \pm 1 \pm \mathbf i), \tfrac 1 {\sqrt 2} ( \pm 1 \pm \mathbf j), \tfrac 1 {\sqrt 2} ( \pm 1 \pm \mathbf k),\tfrac 1 {\sqrt 2} ( \pm \mathbf i \pm \mathbf j),\tfrac 1 {\sqrt 2} ( \pm \mathbf i \pm \mathbf k),\tfrac 1 {\sqrt 2} ( \pm \mathbf j \pm \mathbf k) \right \}.
\end{align}

We have embedding\footnote{In \cite{FKRS}, the symplectic form of $\USp(4)$ is changed for $\U(1) \times \U(1)$ and for the groups that contain it. In this paper, we do not change the symplectic form. There is no difference except for switching some indices.} $\U(1) \times \U(1)$ into $\USp(4)$ by
\begin{equation} \label{embed-u1-u1}  (u_1, u_2) \mapsto \mathrm{diag}(u_1, u_2, u_1^{-1}, u_2^{-1}) .\end{equation}
Let
\begin{align}  \mathtt a &= {\tiny \begin{pmatrix}  0&0&1&0 \\  0&1&0&0 \\ -1&0&0&0 \\ 0&0&0&1 \end{pmatrix}},  & \mathtt b& = {\tiny \begin{pmatrix}  1&0&0&0 \\  0&0&0&1 \\ 0&0&1&0 \\ 0&-1&0&0 \end{pmatrix}},
& \mathtt c& = {\tiny \begin{pmatrix}  0&1&0&0 \\  -1&0&0&0 \\ 0&0&0&1 \\ 0&0&-1&0 \end{pmatrix}}. \label{abc} \end{align}

\begin{definition}[Sato--Tate groups] 
(1)
For $n=1,2,3,4,6$, define
\[C_n := \langle \U(1), \pmb{\zeta}_{2n} \rangle . \] For $n=2,3,4,6$, define
\[ D_n: = \langle C_n,  \mathbf j \rangle. \] With $Q_1$ and $Q_2$ in \eqref{q1} and \eqref{q2}, respectively, define \[   T:= \langle \U(1), Q_1 \rangle \quad
\text{ and } \quad O := \langle T, Q_2 \rangle.  \]

(2) Define the groups
\begin{align*} J(C_n)&:=\langle C_n,  J \rangle \quad (n=1,2,3,4,6), & J(D_n)&:=\langle D_n ,J \rangle \quad (n=2,3,4,6),\\ \quad J(T)&:=\langle T, J \rangle , & J(O)&:=\langle O, J \rangle . \end{align*}

(3) For $n=2,4,6$, define
\[ C_{n,1} := \langle \U(1) , J \pmb{\zeta}_{2n} \rangle \quad \text{ and } \quad D_{n,1} := \langle \U(1), J \pmb{\zeta}_{2n},\,  \mathbf j \rangle . \]
For $n=3,4,6$, define
\[ D_{n,2} := \langle \U(1), \pmb{\zeta}_{2n}, J \mathbf j \rangle , \]
and define
\[ O_1 := \langle T, J Q_2 \rangle  \] with $Q_2$ in \eqref{q2}.

(4) For $n=1,2,3,4,6$, define \[ E_n:= \langle \SU(2), e^{\pi i/n} \rangle  \quad \text{ and } \quad J(E_n) := \langle \SU(2), e^{\pi i/n}, J \rangle ,\] where $e^{\pi i/n}$ is identified with  \[\mathrm{diag}(e^{\pi i/n},e^{\pi i/n},e^{-\pi i/n},e^{-\pi i/n}).\]

(5) The image of $\U(2)$ is denoted by the same notation and its normalizer by $N(\U(2))$.

(6) Define $F$ to be the image of $\U(1) \times \U(1)$ under the embedding \eqref{embed-u1-u1}, and define
\begin{align*} F_{\mathtt a}&= \langle F, \mathtt a \rangle, \quad F_{\mathtt c}= \langle F, \mathtt c \rangle, \quad F_{\mathtt{ab}}= \langle F, \mathtt{ab} \rangle,\quad F_{\mathtt{ac}}= \langle F, \mathtt{ac} \rangle, \\
F_{\mathtt{a},\mathtt{b}}&= \langle F, \mathtt{a},\mathtt{b} \rangle, \quad F_{\mathtt{ab},\mathtt{c}}= \langle F, \mathtt{ab},\mathtt{c} \rangle, \quad F_{\mathtt{a},\mathtt{b},\mathtt{c}}= \langle F, \mathtt{a},\mathtt{b},
\mathtt{c} \rangle ,\end{align*}
where $\mathtt{a},\mathtt{b},\mathtt{c}$ are defined in \eqref{abc}.

(7) Define $G_{1,3}$ and $G_{3,3}$ to be the images of $\U(1) \times \SU(2)$ and $\SU(2) \times \SU(2)$ respectively under the embedding \eqref{embed-su2-su2}, and $N(G_{1,3})$ and $N(G_{3,3})$ to be their normalizers in $\USp(4)$.

\end{definition}

The following is one of the main theorems in this paper.

\begin{theorem}[Theorem \ref{thm-genus-2-1}] \label{thm-genus-2}
For each $m \in \mathbb Z_{\ge 1}$, we have
\begin{align}  \int_{H}  \prod_{j=1}^m \det (I+ x_j \gamma)  d
\gamma & =  (x_1 \cdots x_m)^2 \sum_{b=0}^{m} \sum_{z=0}^{\lfloor \frac {m-b} 2 \rfloor} \mathfrak m_{(b+2z,b)} \chi_{(2^{m-b-2z}, 1^{2z})}^{\Sp(2m)} \label{eqn-re-1} ,
\end{align}
where $\mathfrak m_{(b+2z,b)}$ is the multiplicity of the trivial representation in the restriction $\chi^{\Sp(4)}_{(b+2z,b)}\big |_H$ for each Sate--Tate group $H \le \USp(4)$ and is explicitly given in Table \ref{tab-1}.
\end{theorem}
{\scriptsize
\begin{table}\begin{center}
\begin{tabular}{ |c|c||c|c| }
\hline
$H$ & $\mathfrak m_{(b+2z,b)}$ & $H$ & $\mathfrak m_{(b+2z,b)}$ \\
\hline
$C_1$ & $\eta_1$ & $D_{3,2}$ & $\tfrac 1 6 \eta_1 + \tfrac 1 3 \eta_3+ \tfrac 1 2 \psi_2$\\
$C_2$ & $\tfrac 1 2 \eta_1 + \tfrac 1 2 \eta_2$ & $D_{4,2}$ & $\tfrac 1 8 \eta_1 + \tfrac 1 8 \eta_2 + \tfrac 1 4 \eta_4 + \tfrac 1 2 \psi_2$\\
$C_3$ & $\tfrac 1 3 \eta_1 + \tfrac 2 3 \eta_3$  & $D_{6,2}$ & $\tfrac 1 {12} \eta_1 + \tfrac 1 {12} \eta_2 + \tfrac 1 6 \eta_3+ \tfrac 1 6 \eta_6 + \tfrac 1 2 \psi_2$ \\
$C_4$ & $\tfrac 1 4 \eta_1 + \tfrac 1 4 \eta_2 + \tfrac 1 2 \eta_4 $  & $O_1$ & $\tfrac 1 {24} \eta_1 + \tfrac 1 8 \eta_2 + \tfrac 1 3 \eta_3+ \tfrac 1 4 \psi_2 + \tfrac 1 4 \psi_4$ \\
$C_6$ & $\tfrac 1 6 \eta_1 + \tfrac 1 6 \eta_2 + \tfrac 1 3 \eta_3+ \tfrac 1 3 \eta_6 $ & $E_1$ & $b+1$ \\
$D_2$ & $\tfrac 1 4 \eta_1 + \tfrac 3 4 \eta_2$ &  $E_2$ & $(b+1) \, \delta(b \equiv_2 0)$\\
$D_3$ &  $\tfrac 1 6 \eta_1 + \frac 1 2 \eta_2+ \tfrac 1 3 \eta_3$  & $E_3$ & $\lfloor b/3 \rfloor +1 -\delta(b \equiv_3 1)$  \\
$D_4$ &  $\tfrac 1 8 \eta_1 + \tfrac 5 8 \eta_2 + \tfrac 1 4 \eta_4$  & $E_4$ & $(2 \lfloor b/4 \rfloor +1) \,  \delta(b \equiv_2 0)$  \\
$D_6$ &   $\tfrac 1 {12} \eta_1 + \tfrac 7 {12} \eta_2 + \tfrac 1 6 \eta_3+ \tfrac 1 6 \eta_6 $ & $E_6$ & $(2 \lfloor b/6 \rfloor +1) \,  \delta(b \equiv_2 0 )$   \\
$T$ & $\tfrac 1 {12} \eta_1 + \tfrac 1 4 \eta_2 + \tfrac 2 3 \eta_3 $  & $J(E_1)$ & $\tfrac 1 2 (b+1) +\tfrac 1 2  (-1)^z  \, \delta(b \equiv_2 0)$\\
$O$ & $\tfrac 1 {24} \eta_1 + \tfrac 3 8 \eta_2 + \tfrac 1 3 \eta_3 +\tfrac 1 4 \eta_4$  &  $J(E_2)$& $(b/2+ \delta(z \equiv_2 0))  \, \delta(b \equiv_2 0)$\\
$J(C_1)$ & $\theta_1$ & $J(E_3)$& $\tfrac 1 2 (\lfloor b/3 \rfloor +1 -\delta(b \equiv_3 1)) + \tfrac 1 2  (-1)^z  \, \delta(b \equiv_2 0)$ \\
$J(C_2)$ & $\tfrac 1 2 \theta_1 + \tfrac 1 2 \theta_2$ &  $J(E_4)$& $ ( \lfloor b/4 \rfloor+ \delta(z \equiv_2 0))  \, \delta(b \equiv_2 0)$\\
$J(C_3)$ & $\tfrac 1 3 \theta_1 + \tfrac 2 3 \theta_3$  & $J(E_6)$& $ ( \lfloor b/6 \rfloor+ \delta(z \equiv_2 0))  \, \delta(b \equiv_2 0)$\\
$J(C_4)$ & $\tfrac 1 4 \theta_1 + \tfrac 1 4 \theta_2 + \tfrac 1 2 \theta_4$ & $\U(2)$& $\delta(b \equiv_2 0)$ \\
$J(C_6)$ & $\tfrac 1 6 \theta_1 + \tfrac 1 6 \theta_2 + \tfrac 1 3 \theta_3+ \tfrac 1 3 \theta_6 $ &$N(\U(2))$& $\delta(b \equiv_2 0)\, \delta(z \equiv_2 0)$  \\
$J(D_2)$ & $\tfrac 1 4 \theta_1 + \tfrac 3 4 \theta_2$ & $F$ & $\xi_1$\\
$J(D_3)$ & $\tfrac 1 6 \theta_1 + \frac 1 2 \theta_2+ \tfrac 1 3 \theta_3$  & $F_{\mathtt{a}}$ & $\tfrac 1 2 \xi_1 + \tfrac 1 2 \xi_2$\\
$J(D_4)$ &$\tfrac 1 8 \theta_1 + \tfrac 5 8 \theta_2 + \tfrac 1 4 \theta_4$  &  $F_{\mathtt{c}}$ & $\tfrac 1 2 \xi_1 + \tfrac 1 2 \eta_2$\\
$J(D_6)$ & $\tfrac 1 {12} \theta_1 + \tfrac 7 {12} \theta_2 + \tfrac 1 6 \theta_3+ \tfrac 1 6 \theta_6$ & $F_{\mathtt{ab}}$ & $\tfrac 1 2 \xi_1 +\tfrac 1 2 \psi_2$\\
$J(T)$ & $\tfrac 1 {12} \theta_1 + \tfrac 1 4 \theta_2 + \tfrac 2 3 \theta_3 $  & $F_{\mathtt{ac}}$ & $\tfrac 1 4 \xi_1 + \tfrac 1 4 \psi_2 +\tfrac 1 2 \psi_4$\\
$J(O)$ & $\tfrac 1 {24} \theta_1 + \tfrac 3 8 \theta_2 + \tfrac 1 3 \theta_3 +\tfrac 1 4 \theta_4$  &  $F_{\mathtt{a},\mathtt{b}}$  &  $\tfrac 1 4 \xi_1 + \tfrac 1 4 \psi_2 +\tfrac 1 2 \xi_2$\\
$C_{2,1}$ & $\tfrac 1 2 \eta_1 + \tfrac 1 2 \psi_2$  & $F_{\mathtt{ab},\mathtt{c}}$ & $\tfrac 1 4 \xi_1+\tfrac 1 4 \psi_2 +\tfrac 1 2 \eta_2 $\\
$C_{4,1}$ & $\tfrac 1 4 \eta_1 + \tfrac 1 4 \eta_2 + \tfrac 1 2 \psi_4$  &  $F_{\mathtt{a},\mathtt{b},\mathtt{c}}$& $\tfrac 1 8 \xi_1 + \tfrac 1 4 \xi_2+ \tfrac 1 8 \psi_2 +\tfrac 1 4 \psi_4 + \tfrac 1 4 \eta_2$\\
$C_{6,1}$ & $\tfrac 1 6 \eta_1 + \tfrac 1 3 \eta_3 + \tfrac 1 6 \psi_2 + \tfrac 1 3 \psi_6$  & $G_{1,3}$  & $1$\\
$D_{2,1}$& $\tfrac 1 4 \eta_1 +\tfrac 1 4 \eta_2+ \tfrac 1 2 \psi_2$ & $N(G_{1,3})$& $\delta(z \equiv_2 0)$\\
$D_{4,1}$ &$\tfrac 1 8 \eta_1 + \tfrac 3 8 \eta_2 + \tfrac 1 4 \psi_2+ \tfrac 1 4 \psi_4$ & $G_{3,3}$ & $\delta(z=0)$\\
$D_{6,1}$& $\tfrac 1 {12} \eta_1 + \tfrac 1 4 \eta_2+ \tfrac 1 6 \eta_3 + \tfrac 1 3 \psi_2 + \tfrac 1 6 \psi_6$&$N(G_{3,3})$ & $\delta(b\equiv_2 0) \, \delta(z=0)$ \\
&& $\USp(4)$ & $\delta(b=0) \, \delta(z=0)$ \\
 \hline
\end{tabular}
\end{center}
\caption{Coefficients $\mathfrak m_{(b+2z,b)}$}\label{tab-1} \end{table}
} Here we 
define
\begin{align*} \eta_1(z,b) &:=(b+1)(z^2 + zb+ 2z +b/2+1),  \\
\eta_2(z,b)&:= \begin{cases} - \tfrac {b+1} 2  & \text{if $b$ is odd}, \\
\tfrac b 2 + \delta(\text{$z$ is even}) & \text{if $b$ is even}, \end{cases}
\end{align*}
and the functions $\eta_i(z,b)$, $i=3,4,6$ on the congruence classes of $z$ and $b$ as in Table \ref{tab-2};
\begin{table}
\begin{center}
{\scriptsize
\begin{tabular}{|c||c|c|c|}
\hline $z \backslash b$ & 0&1&2 \\ \hline \hline
0&1 & 0& 0 \\ \hline
1&1 & $-1$ & 0 \\ \hline
2&0 & $-1$ & 0 \\ \hline
\end{tabular} } \quad
{\scriptsize
\begin{tabular}{|c||c|c|c|c|}
\hline
$z\backslash b$ & 0&1&2&3\\ \hline \hline
0&1 & 1&0& 0 \\ \hline
1&2&1 & $-1$ & 0 \\ \hline
2&1 & $-1$ & $-2$&0 \\ \hline
3&0 & $-1$ & $-1$& 0 \\ \hline
\end{tabular}} \quad
{\scriptsize
\begin{tabular}{|c||c|c|c|c|c|c|}
\hline
$z\backslash b$& 0&1&2&3&4&5 \\ \hline \hline
0&1 & 2&2&1&0& 0 \\ \hline
1&3&5&4&1 & $-1$ & 0 \\ \hline
2&4 &5&2& $-2$ & $-3$&0 \\ \hline
3&3&2&$-2$ & $-5$ & $-4$& 0 \\ \hline
4&1 & $-1$ &$-4$ & $-5$& $-3$&0 \\ \hline
5&0&$-1$ & $-2$ & $-2$& $-1$&0 \\ \hline
\end{tabular}}
\end{center}
\caption{Functions $\eta_i(z,b)$, $i=3,4,6$}\label{tab-2}
\end{table}
define
\begin{align}
\psi_1(z,b) &:= (-1)^b (b+1)(z+b/2+1) \label{def-psi-1},\\
\psi_2(z,b) &:= \begin{cases} (-1)^z  \tfrac {b+1} 2  & \text{ if $b$ is odd}, \\ (-1)^z (z +b/2+1) & \text{ if $b$ is even}, \end{cases} \label{def-psi-2}
\end{align}
and the functions  $\psi_i(z,b)$, $i=3,4,6$ on the congruence classes of $z$ and $b$  as in Table \ref{tab-3};
\begin{table}
\begin{center} {\scriptsize
\begin{tabular}{|c||c|c|c|c|c|c|}
\hline
$z\backslash b$ & 0&1&2&3&4&5\\ \hline \hline
0&$1$ & $0$&$0$& $-1$& $0$& $0$ \\ \hline
1&$-1$&$1$ & $0$ & $1$ & $-1$& $0$\\ \hline
2& $0$ & $-1$ & $0$&$0$ & $1$& $0$\\ \hline
\end{tabular}} \quad
{\scriptsize
\begin{tabular}{|c||c|c|c|c|}
\hline
$z\backslash b$ & 0&1&2&3\\ \hline \hline
0&$1$ & $-1$&$0$& $0$ \\ \hline
1&$0$&$1$ & $-1$ & $0$ \\ \hline
2& $-1$ & $1$ & $0$&$0$ \\ \hline
3&$0$ & $-1$ & $1$ & $0$ \\ \hline
\end{tabular}} \quad
{\scriptsize
\begin{tabular}{|c||c|c|c|c|c|c|}
\hline
$z\backslash b$ & 0&1&2&3&4&5\\ \hline \hline
0& $1$ & $-2$&$2$& $-1$& $0$& $0$ \\ \hline
1& $1$&$-1$ & $0$ & $1$ & $-1$& $0$\\ \hline
2& $0$ & $1$ & $-2$&$2$ & $-1$& $0$\\ \hline
3& $-1$ & $2$ & $-2$&$1$ & $0$& $0$\\ \hline
4& $-1$ & $1$ & $0$&$-1$ & $1$& $0$\\ \hline
5& $0$ & $-1$ & $2$&$-2$ & $1$& $0$\\ \hline
\end{tabular}}
\end{center}
\caption{Functions $\psi_i(z,b)$, $i=3,4,6$} \label{tab-3}
\end{table}
define \[\xi_1 (z,b) := z(b+1) +\lfloor b/2 \rfloor +1,
\]
and $\xi_2 (z,b) $ on  the congruence classes of $z$ and $b$ as in Table \ref{tab-4};
\begin{table}
\begin{center} {\scriptsize
\begin{tabular}{|c||c|c|c|c|}
\hline
$z\backslash b$ & 0&1&2&3\\ \hline \hline
0&$1$ & $1$&$0$& $0$ \\ \hline
1&$0$&$-1$ & $-1$ & $0$ \\ \hline
\end{tabular}} \end{center}
\caption{Function $\xi_2(z,b)$} \label{tab-4}
\end{table}
finally define
\begin{equation} \label{def-th} \theta_i (z,b) := \frac 1 2 \eta_i (z,b) + \frac 1 2 \psi_i (z,b) , \qquad i=1,2,3,4,6 .\end{equation}
More explicitly,
\begin{align*} \theta_1(z,b) &:= \begin{cases} \tfrac 1 2 z(b+1)(z+b+1)  & \text{ if $b$ is odd}, \\ \tfrac 1 2 (z+1)(b+1)(z+b+2) & \text{ if $b$ is even}, \end{cases}\\
\theta_2(z,b) &: =\begin{cases} -\tfrac 1 2(b+1) & \text{ if $b$ is odd and $z$ is odd}, \\
0 & \text{ if $b$ is odd and $z$ is even}, \\ -\tfrac 1 2(z+1) & \text{ if $b$ is even and $z$ is odd}, \\ \tfrac 1 2 b+ \tfrac 1 2 z+1 & \text{ if $b$ is even and $z$ is even},
\end{cases}
\end{align*}
and $\theta_i(z,b) $, $i=3,4,6$, are determined by the congruence classes of $z$ and $b$ as in Table \ref{tab-5}.
\begin{table}
\begin{center}{\scriptsize
\begin{tabular}{|c||c|c|c|c|c|c|}
\hline
$z\backslash b$ & 0&1&2&3&4&5\\ \hline \hline
0&$1$ & $0$&$0$& $0$& $0$& $0$ \\ \hline
1&$0$&$0$ & $0$ & $1$ & $-1$& $0$\\ \hline
2& $0$ & $-1$ & $0$&$0$ & $0$& $0$\\ \hline
\end{tabular}}\quad
{\scriptsize
\begin{tabular}{|c||c|c|c|c|}
\hline
$z\backslash b$ & 0&1&2&3\\ \hline \hline
0&$1$ & $0$&$0$& $0$ \\ \hline
1&$1$&$1$ & $-1$ & $0$ \\ \hline
2& $0$ & $0$ & $-1$&$0$ \\ \hline
3&$0$ & $-1$ & $0$ & $0$ \\ \hline
\end{tabular}} \quad
{\scriptsize
\begin{tabular}{|c||c|c|c|c|c|c|}
\hline
$z\backslash b$ & 0&1&2&3&4&5\\ \hline \hline
0& $1$ & $0$&$2$& $0$& $0$& $0$ \\ \hline
1& $2$&$2$ & $2$ & $1$ & $-1$& $0$\\ \hline
2& $2$ & $3$ & $0$&$0$ & $-2$& $0$\\ \hline
3& $1$ & $2$ & $-2$&$-2$ & $-2$& $0$\\ \hline
4& $0$ & $0$ & $-2$&$-3$ & $-1$& $0$\\ \hline
5& $0$ & $-1$ & $0$&$-2$ & $0$& $0$\\ \hline
\end{tabular}} \end{center}
\caption{Functions $\theta_i(z,b)$, $i=3,4,6$} \label{tab-5}
\end{table}

\bigskip

In the rest of this section, we prove Theorem \ref{thm-genus-2}. The cases $J(C_n), J(E_n)$ and $F_{\mathtt{a}}$ will lead to Theorems \ref{thm-identities}, \ref{thm-identities-e} and \ref{thm-identities-f},  which are also main results of this paper. We denote by $V_{\tilde \lambda}^{\Sp(4)}$ the irreducible representation of $\Sp(4)$ with  character  $\chi_{\tilde \lambda}^{\Sp(4)}$.

\medskip

From Proposition \ref{prop-main}, we have
\[\int_{H} \prod_{i=1}^m \det (I+ x_i \gamma)  \, d \gamma = (x_1 \cdots x_m)^{2} \sum_{\lambda \trianglelefteq (2^m)} \mathfrak m_{\tilde \lambda}(H) \, \chi_\lambda^{\Sp(2m)}. \]
In particular, $\tilde \lambda=(m-\lambda'_2,  m-\lambda'_1)$ with $\lambda' =(\lambda'_1, \lambda'_2)$ the transpose of $\lambda$.
For $\tilde \lambda = (a,b) \trianglelefteq (m^2)$, we define $z:=(a-b)/2$. Then $\tilde \lambda=(b+2z, b)$ and $\lambda = (2^{m-b-2z}, 1^{2z})$ for $0 \le b \le m$ and   $0 \le z \le  \frac {m-b} 2$.

For each Sato--Tate group $H$, the number $\mathfrak m_{(a,b)}(H)$ is equal to the number of independent weight vectors $v_\mu$ with weight $\mu$ in  $V_{(a,b)}^{\Sp(4)}$ which are fixed by $H$ and satisfy some other conditions. Since each $H$ contains either
\[ \mathrm{diag}(u,u,u^{-1},u^{-1}) \quad \text{ or } \quad \mathrm{diag}(u,u^{-1},u^{-1},u), \quad u \in \U(1),\] one necessary condition for $\mu$ is  that \[ \text{$\mu(\hat h_1+\hat h_2) \equiv 0$ (mod $2$).} \]
If $a-b$ is odd then this condition cannot be satisfied.
Thus, for each of the Sato--Tate groups,
\begin{equation} \label{con-coeff}  \text{$\mathfrak m_{(a,b)}=\mathfrak m_{(b+2z,b)} =0$ unless $a-b$ is even, or equivalently, unless $z$ is an integer.} \end{equation}
This justifies having only integer values for $z$ in \eqref{eqn-re-1}.

\subsection{Groups $C_n$} \label{subsec-cn}

For each $n=1,2,3,4,6$, the group $C_n \le \USp(4)$ consists of the matrices of the form
\[ \begin{pmatrix} A &0 \\0 & \overline A \end{pmatrix} \quad \text{ with } \quad  A=\begin{pmatrix} e^{2\pi i r + s \pi i/n} &0 \\0 & e^{2 \pi i r -s\pi i /n} \end{pmatrix}, \quad r \in [0,1), \ s=0,1,\dots , 2n-1 .\]
By definition the number $\mathfrak m_{\tilde \lambda}(C_n)$ is equal to the multiplicity of the trivial representation in $\chi_{\tilde \lambda}^{\Sp(4)}|_{C_n}$.

Let $v_\mu$ be a vector of weight $\mu$ in $V_{\tilde \lambda}^{\Sp(4)}$. Then we obtain
\begin{align*}
\begin{pmatrix} A &0 \\0 & \overline A \end{pmatrix}v_\mu = (e^{2\pi i r})^{\mu(\hat h_1 + \hat h_2)} (e^{s \pi i /n})^{\mu( \hat h_1 - \hat h_2)} \ v_\mu ,
\end{align*}
where $\hat h_1 = E_{11}-E_{33}$ and $\hat h_2 = E_{22}-E_{44}$ as before.
Thus the number $\mathfrak m_{\tilde \lambda}(C_n)$ is equal to the number of independent vectors with weight $\mu$ in the representation $V_{\tilde \lambda}^{\Sp(4)}$ such that \begin{equation} \label{eq: condition C_n}\text{$\mu(\hat h_1+\hat h_2)=0$ \quad and \quad $\mu(\hat h_1-\hat h_2) \equiv 0$ (mod $2n$),\quad  $n=1,2,3,4,6$.} \end{equation}
The numbers $\mathfrak m_{\tilde \lambda}(C_n)$ are all calculated in Section \ref{sec-H-Cn} (Propositions \ref{prop:cn1}--\ref{prop:cn6}) and they match with the formulae in Table \ref{tab-1}.
This prove Theorem \ref{thm-genus-2} for the groups $C_n$.

\medskip

For convenience in notation, define
\begin{align*}
\widetilde\eta_1 &:= \eta_1,&
\widetilde\eta_2 &:= \tfrac 1 2 \eta_1 + \tfrac 1 2 \eta_2 ,&
\widetilde\eta_3 &:= \tfrac 1 3 \eta_1 + \tfrac 2 3 \eta_3 ,\\
\widetilde\eta_4 &:= \tfrac 1 4 \eta_1 + \tfrac 1 4 \eta_2 + \tfrac 1 2 \eta_4 ,&
\widetilde\eta_6 &:= \tfrac 1 6 \eta_1 + \tfrac 1 6 \eta_2 + \tfrac 1 3 \eta_3+ \tfrac 1 3 \eta_6 .
\end{align*}
Then we have
\begin{equation} \label{phi-cn} \int_{C_n} \prod_{i=1}^m \det(I+x_i \gamma) d\gamma_{C_n}  =(x_1 \cdots x_m)^2 \sum_{b=0}^m \sum_{z=0}^{\lfloor \frac {m-b} 2 \rfloor} \widetilde\eta_n(z,b) \chi_{(2^{m-b-2z}, 1^{2z})}^{\Sp(2m)}  \end{equation}
for $n=1,2,3,4,6$, where $d\gamma_{C_n}$ is the probability Haar measure on $C_n$.

\subsection{Cosets of $C_n$} We need to study the cosets of $C_1$ in $C_n$ to understand other groups. Consider
\[  \gamma = \mathrm{diag}(ue^{\pi i /n} , ue^{-\pi i /n }, u^{-1}e^{-\pi i /n} ,  u^{-1}e^{\pi i /n} ) \in \pmb{\zeta}_{2n} C_1 , \quad u \in \U(1).  \]
Then we have
\begin{align} \det (I +x\gamma) &= (1+x u e^{\pi i /n} )(1+x u e^{-\pi i /n} )(1+x u^{-1} e^{-\pi i /n} )(1+x u^{-1} e^{\pi i /n} ) \nonumber \\ &= (1+ \omega_n xu + x^2u^2) (1+ \omega_n xu^{-1} + x^2u^{-2}) \label{det-u}, \end{align}
where we set $\omega_n =-2, 0, 1, \sqrt 2, \sqrt 3 $ for $n=1,2,3,4,6$, respectively.

\begin{proposition}
For $n=1,2,3,4,6$, we have
\begin{align}
\int_{C_1} \prod_{i=1}^m \det(I+x_i \pmb{\zeta}_{2n} \gamma) d\gamma_{C_1} &= \int_{\U(1)} \prod_{i=1}^m (1+\omega_n x_i  u + x_i^2  u^2)(1+\omega_n x_i  u^{-1} +
 x_i^2   u^{-2}) du \nonumber \\ \label{zeta-cn} & = (x_1 \cdots x_m)^2 \sum_{b=0}^m \sum_{z=0}^{\lfloor \frac {m-b} 2 \rfloor} \eta_n(z,b) \chi_{(2^{m-b-2z}, 1^{2z})}^{\Sp(2m)}.\end{align}
\end{proposition}

\begin{proof}
The case $n=1$ is already checked as a part of Theorem \ref{thm-genus-2}, and we have only to consider $n=2,3,4,6$. To begin with, note that
we have the coset decompositions
\begin{align} C_2& = C_1 \sqcup \pmb{\zeta}_4 C_1, \quad
C_3 =  C_1 \sqcup \pmb{\zeta}_6 C_1 \sqcup \pmb{\zeta}_6^2 C_1,\quad
C_4=  C_1 \sqcup \pmb{\zeta}_8 C_1 \sqcup \pmb{\zeta}_4 C_1 \sqcup \pmb{\zeta}_8^3 C_1, \label{coset-234}\\
C_6&=  C_1 \sqcup \pmb{\zeta}_{12} C_1 \sqcup \pmb{\zeta}_6 C_1 \sqcup \pmb{\zeta}_4 C_1 \sqcup \pmb{\zeta}_6^2 C_1 \sqcup \pmb{\zeta}_{12}^{11} C_1. \label{coset-6}
\end{align}

Let $d\gamma_{C_n}$ and $du$ be the probability Haar measures on $C_n$ and $\U(1)$, respectively. Note that $C_1 \cong \U(1)$. Recall \eqref{phi-cn}:
\begin{align*} & \int_{C_n} \prod_{i=1}^m \det(I+x_i \gamma) d\gamma_{C_n} = (x_1 \cdots x_m)^2 \sum_{b=0}^m \sum_{z=0}^{\lfloor \frac {m-b} 2 \rfloor} \widetilde\eta_n(z,b) \chi_{(2^{m-b-2z}, 1^{2z})}^{\Sp(2m)}.\end{align*}

In what follows we write  \begin{equation} \label{int-nota} \Delta( \pmb{\zeta}  \gamma) := \prod_{i=1}^m\det(I+x_i \pmb{\zeta} \gamma)d\gamma \end{equation} to ease the notation, where $d\gamma$ is the probability Haar measure on the group the integral is over.

When $n=2$, we obtain from \eqref{det-u} and \eqref{coset-234},
\begin{align*}
&\int_{C_2} \Delta(\gamma)= \frac 1 2  \int_{C_1} \Delta(\gamma) + \frac 1 2 \int_{C_1} \Delta(\pmb{\zeta}_4 \gamma) \\ 
&=  \frac 1 2 (x_1 \cdots x_m)^2 \sum_{b=0}^m \sum_{z=0}^{\lfloor \frac {m-b} 2 \rfloor} \eta_1(z,b) \chi_{(2^{m-b-2z}, 1^{2z})}^{\Sp(2m)} + \frac 1 2 \int_{\U(1)} \prod_{i=1}^m (1+ x_i^2u^2) (1+ x_i^2u^{-2})du. \nonumber
\end{align*}
Since $\widetilde\eta_2 = \frac 1 2 \eta_1+ \frac 1 2 \eta_2$, we obtain \eqref{zeta-cn} for $n=2$.

For $n=3$, since $-u \in \U(1)$ for $u \in \U(1)$, a similar computation to \eqref{det-u} yields \[ \int_{C_1} \Delta(\pmb{\zeta}_6 \gamma) = \int_{C_1}\Delta(\pmb{\zeta}_6^2 \gamma) .\] Thus we obtain from \eqref{det-u} and \eqref{coset-234},
\begin{align*}
&\int_{C_3} \Delta(\gamma)= \frac 1 3 \int_{C_1} \Delta(\gamma) + \frac 2 3 \int_{C_1} \Delta(\pmb{\zeta}_6 \gamma) \nonumber  
\allowdisplaybreaks\\
&=  \frac 1 3 (x_1 \cdots x_m)^2 \sum_{b=0}^m \sum_{z=0}^{\lfloor \frac {m-b} 2 \rfloor} \eta_1(z,b) \chi_{(2^{m-b-2z}, 1^{2z})}^{\Sp(2m)} \nonumber \allowdisplaybreaks\\ &
\phantom{LLLLLLLLLLLLLLLL}+ \frac 2 3\int_{\U(1)} \prod_{i=1}^m (1+x_iu+ x_i^2u^2) (1+x_iu^{-1}+ x_i^2u^{-2})du. \nonumber
\end{align*}
Since $\widetilde\eta_3 = \frac 1 3 \eta_1+ \frac 2 3\eta_3$, we obtain \eqref{zeta-cn} for $n=3$.

When $n=4$, we get \[ \int_{C_1} \Delta(\pmb{\zeta}_8 \gamma)= \int_{C_1} \Delta(\pmb{\zeta}_8^3 \gamma).\] Thus we obtain from \eqref{det-u} and \eqref{coset-234},
\begin{align*}
&\int_{C_4} \Delta(\gamma)= \frac 1 4 \int_{C_1} \Delta(\gamma)  + \frac 1 4 \int_{C_1} \Delta(\pmb{\zeta}_4 \gamma)   +\frac 1 2 \int_{C_1} \Delta(\pmb{\zeta}_8\gamma)  \nonumber 
\allowdisplaybreaks\\
&=  \frac 1 4 (x_1 \cdots x_m)^2 \sum_{b=0}^m \sum_{z=0}^{\lfloor \frac {m-b} 2 \rfloor} (\eta_1+\eta_2)(z,b) \chi_{(2^{m-b-2z}, 1^{2z})}^{\Sp(2m)} \nonumber \allowdisplaybreaks\\
& \phantom{LLLLLLLLLLLLL}+ \frac 1 2 \int_{\U(1)} \prod_{i=1}^m (1+ \sqrt{2} x_iu+x_i^2u^2) (1+ \sqrt{2}x_iu^{-1}+x_i^2u^{-2})du. \nonumber
\end{align*}
Since $\widetilde\eta_4 = \frac 1 4 \eta_1 + \frac 1 4 \eta_2 + \frac 1 2 \eta_4$, we obtain \eqref{zeta-cn} for $n=4$.

The case $n=6$ is similar to the previous cases, and
we obtain
\begin{equation*} 
\int_{C_6} \Delta(\gamma)= \frac 1 6\int_{C_1}\Delta(\gamma) + \frac 1 6 \int_{C_1} \Delta(\pmb{\zeta}_4 \gamma) + \frac 1 3 \int_{C_1} \Delta(\pmb{\zeta}_6 \gamma) + \frac 1 3 \int_{C_1} \Delta(\pmb{\zeta}_{12}\gamma) .
\end{equation*}
Since $\widetilde\eta_6= \frac 1 6 \eta_1 + \frac 1 6 \eta_2+ \frac 1 3 \eta_3+ \frac 1 3 \eta_6$,  we obtain \eqref{zeta-cn} for $n=6$.
\end{proof}

\begin{corollary} \label{eta-2-special}
For any $m \in \mathbb Z_{\ge 1}$, we have
\begin{align*}
 \sum_{b=0}^m \sum_{z=0}^{\lfloor \frac {m-b} 2 \rfloor} \eta_2(z,b) \chi_{(2^{m-b-2z}, 1^{2z})}^{\Sp(2m)}&= \sum_{\ell =0}^{\lfloor m/2 \rfloor} \binom{2\ell}{\ell} \sum_{1 \le i_1 < \cdots <i_{m-2\ell} \le m}  \prod_{j=1}^{m-2\ell} (x_{i_j}^2+ x_{i_j}^{-2})  \\  & = \sum_{j=0}^{\lfloor \frac m 2 \rfloor} \chi_{(1^{m-2j})}^{\Sp(2m)} (x_1^{\pm 2}, \dots , x_m^{\pm 2}).
\end{align*}
\end{corollary}

\begin{proof}
The first equality comes from \eqref{zeta-cn} and \eqref{uu-1}, and the second equality comes from \eqref{gen-1-id-1}.
\end{proof}

\subsection{Groups $D_n$}
We have the coset decomposition $D_n=C_n \sqcup \mathbf j C_n$ with $ \mathbf j = \scriptsize { \begin{pmatrix} 0 & 1 & 0 & 0 \\ -1 & 0 & 0&0 \\ 0 &0&0&1\\ 0& 0&-1&0 \end{pmatrix} }$ for $n=2,3,4,6$. Consider
\[  \gamma = {\scriptsize \begin{pmatrix} 0 & ue^{-\pi i /n} &&\\   -ue^{\pi i /n} & 0 &&\\ && 0 & u^{-1}e^{\pi i /n} \\ && -u^{-1}e^{-\pi i /n} &0 \end{pmatrix}} \in \mathbf j C_n, \quad u \in \U(1).\]
We see that
\[ \det (I +x\gamma) = (1+x^2 u^2) (1+x^2 u^{-2}) \]
for any $n=2,3,4,6$.
Then we have
\begin{align*}  \int_{D_n} \Delta(\gamma) &= \frac 1 2 \int_{C_n} \Delta(\gamma) + \frac 1 2 \int_{C_n} \Delta(\mathbf j \gamma) \allowdisplaybreaks\\
&=  \frac 1 2 \int_{C_n} \Delta(\gamma)  + \frac 1 2 \int_{\U(1)} \prod_{i=1}^m (1+ x_i^2u^2) (1+ x_i^2u^{-2})du.
\end{align*}
By \eqref{zeta-cn}, we obtain
\begin{equation} \label{jgam} \int_{C_n} \Delta(\mathbf j \gamma) = (x_1 \cdots x_m)^2 \sum_{b=0}^m \sum_{z=0}^{\lfloor \frac {m-b} 2 \rfloor} \eta_2(z,b) \chi_{(2^{m-b-2z}, 1^{2z})}^{\Sp(2m)} , \end{equation}
and it follows from \eqref{phi-cn} that  $\mathfrak m_{\tilde \lambda}(D_n)$ are given by
\begin{equation} \label{dn-eta} \frac 1 2 \widetilde\eta_n + \frac 1 2 \eta_2 \qquad  \text{ for } n=2,3,4,6 .\end{equation}
These coincide with the formulae in Table \ref{tab-1} and
prove Theorem \ref{thm-genus-2} for the groups $D_n$.

\subsection{Group $T$} \label{subsec-t}

There are 24 cosets of $C_1$ in $T$, whose representatives are given by
\[ \left  \{ \pm 1, \pm \pmb{\rm i}, \pm \pmb{\rm j}, \pm \pmb{\rm k}, \tfrac 1 2 (\pm 1 \pm \pmb{\rm i} \pm \pmb{\rm j}\pm \pmb{\rm k}) \right \} . \]
From computations of the characteristic polynomials for the elements of each coset, we see the following.
\begin{enumerate} \item[(i)] There are 2 cosets whose characteristic polynomials are
 \[(1-2 xu + x^2u^2) (1-2 xu^{-1} + x^2u^{-2})  \quad (u \in \U(1)).\] By \eqref{det-u}, these are the same as those of $C_1$.
\item [(ii)]  There are 6 cosets whose characteristic polynomials are \[(1+ x^2u^2) (1+ x^2u^{-2})  \quad (u \in \U(1)),\]  which are the same as
those of  $\pmb{\zeta}_4 C_1$ by \eqref{det-u}.
\item[(iii)] The characteristic polynomials of the remaining 16 cosets are
\[ (1+  xu + x^2u^2) (1+  xu^{-1} + x^2u^{-2})  \quad (u \in \U(1)).\]  These are the same as those of $\pmb{\zeta}_6 C_1$ by \eqref{det-u}.
\end{enumerate}

Thus we obtain
\begin{align*}
\int_{T} \Delta(\gamma)&= \frac 2 {24} \int_{C_1} \Delta(\gamma)+ \frac 6 {24} \int_{C_1} \Delta(\pmb{\zeta}_4 \gamma)+ \frac {16}{24} \int_{C_1}\Delta(\pmb{\zeta}_6 \gamma)
\end{align*}
and it follows from \eqref{zeta-cn} that
\begin{equation} \label{res-t} \mathfrak m_{\tilde \lambda}(T) = \tfrac 1 {12} \eta_1 + \tfrac 1 4 \eta_2 + \tfrac 2 3 \eta_3.\end{equation}
This proves
Theorem \ref{thm-genus-2} for the group $T$.

\subsection{Group $O$} \label{subsec-o}

There are 48 cosets of $C_1$ as one can see from \eqref{q2}. We compute the characteristic polynomial for the elements of each coset.
Comparing with \eqref{det-u}, we obtain 2 cosets with $C_1$ polynomials, 18 cosets with $\pmb{\zeta}_4 C_1$ polynomials, 16 cosets with $\pmb{\zeta}_6 C_1$ polynomials and 12 cosets with $\pmb{\zeta}_8 C_1$ polynomials.

Thus  we obtain
\begin{align*}
\int_{O} \Delta(\gamma)& = \frac 2 {48} \int_{C_1}\Delta(\gamma) +\frac {18} {48}  \int_{C_1}\Delta(\pmb{\zeta}_4 \gamma) + \frac {16} {48}  \int_{C_1}\Delta(\pmb{\zeta}_{6} \gamma)+ \frac {12}{48}  \int_{C_1}\Delta(\pmb{\zeta}_{8} \gamma),\end{align*}
and  it follows from \eqref{zeta-cn} that
\begin{equation} \label{res-o} \mathfrak m_{\tilde \lambda}(O)= \tfrac 1 {24} \eta_1 + \tfrac 3 8 \eta_2 + \tfrac 1 3 \eta_3 +\tfrac 1 4 \eta_4. \end{equation}
This proves
Theorem \ref{thm-genus-2} for the group $O$.

\subsection{Groups $J(C_n)$}

Recall $J = \scriptsize { \begin{pmatrix} 0 & 0 & 0 & 1 \\ 0 & 0 & -1&0 \\ 0 &-1&0&0\\ 1& 0&0&0 \end{pmatrix} }$. The group $J(C_n)$ is defined to be the group obtained by adjoining $J$ to $C_n$ for $n=1,2,3,4,6$, and we have the decomposition
\[ J(C_n) = C_n \sqcup J C_n. \]

The number $\mathfrak m_{\tilde \lambda}(J(C_n))$ is equal to the number of independent \emph{$J$-fixed vectors} with weight $\mu$ in the representation $V_{\tilde \lambda}^{\Sp(4)}$
 satisfying the conditions in~\eqref{eq: condition C_n}.
The numbers $\mathfrak m_{\tilde \lambda}(J(C_n))$ are all calculated in Propositions \ref{prop-J-imp} and \ref{prop-J-imp-2}, and they coincide with the formulae in Table \ref{tab-1}.
This prove Theorem \ref{thm-genus-2} for the groups $J(C_n)$.

\medskip

Define
\begin{equation} \label{def-wi-theta}
\begin{aligned}
\widetilde\theta_1 &:= \theta_1,&
\widetilde\theta_2 &:= \tfrac 1 2 \theta_1 + \tfrac 1 2 \theta_2 ,&
\widetilde\theta_3 &:= \tfrac 1 3 \theta_1 + \tfrac 2 3 \theta_3 ,\\ 
\widetilde\theta_4 &:= \tfrac 1 4 \theta_1 + \tfrac 1 4 \theta_2 + \tfrac 1 2 \theta_4 ,&
\widetilde\theta_6 &:= \tfrac 1 6 \theta_1 + \tfrac 1 6 \theta_2 + \tfrac 1 3 \theta_3+ \tfrac 1 3 \theta_6 .
\end{aligned}
\end{equation}
Then the result for $J(C_n)$ can be written as
\begin{align} \label{theta-11} & \int_{J(C_n)} \prod_{i=1}^m \det(I+x_i \gamma) d\gamma = (x_1 \cdots x_m)^2 \sum_{b=0}^m \sum_{z=0}^{\lfloor \frac {m-b} 2 \rfloor} \widetilde\theta_n(z,b) \chi_{(2^{m-b-2z}, 1^{2z})}^{\Sp(2m)}.\end{align}

\medskip

Using the above result, we now prove another main theorem of this paper.
\begin{theorem}[Theorem \ref{thm-identities-1} (1)] \label{thm-identities}
Let $ \kappa_1=-2 , \kappa_2=2, \kappa_3=1, \kappa_4=0, \text{ and } \kappa_6=-1$.
Then,
for any $m \in \mathbb Z_{\ge 1}$ and $n=1,2,3,4,6$, we have the following identities:
\begin{align}
\int_{C_1} \prod_{i=1}^m \det(I+x_iJ \pmb{\zeta}_{2n} \gamma) d\gamma_{C_1} &=\prod_{i=1}^m (x_i^2+\kappa_n+x_i^{-2}) = \sum_{b=0}^m \sum_{z=0}^{\lfloor \frac {m-b} 2 \rfloor} \psi_n(z,b) \chi_{(2^{m-b-2z}, 1^{2z})}^{\Sp(2m)} , \label{psi-kappa}
\end{align}
where $\psi_n(z,b)$ are defined in \eqref{def-psi-1}, \eqref{def-psi-2} and Table \ref{tab-3}.
\end{theorem}

\begin{proof}
We have $J\pmb{\zeta}_{2n} = {\scriptsize \begin{pmatrix} 0&0&0& e^{\pi i/n}\\0&0&-e^{-\pi i/n}&0\\0&-e^{-\pi i/n}&0&0\\e^{\pi i/n}&0&0&0 \end{pmatrix}}$.
For $\gamma \in C_1$ and $n=1,2,3,4,6$,  direct computation shows
\begin{equation} \label{kap} \det (I + xJ\pmb{\zeta}_{2n}\gamma) = 1-(e^{\pi i/n} +e^{-\pi/n})x^2  +x^4 = 1-\kappa_n x^2 + x^4. \end{equation}
This proves the first equality in \eqref{psi-kappa}.

We continue to use the notational convention in \eqref{int-nota}. We have
\[ \int_{J(C_n)} \Delta(\gamma) = \frac 1 2 \int_{C_n} \Delta(\gamma_n) + \frac 1 2 \int_{C_n} \Delta(J\gamma_n) .\]

When $n=1$, we get from \eqref{phi-cn}, \eqref{theta-11} and \eqref{kap}
\begin{align*} &\int_{J(C_1)} \Delta(\gamma)= (x_1 \cdots x_m)^2 \sum_{b=0}^m \sum_{z=0}^{\lfloor \frac {m-b} 2 \rfloor} \widetilde\theta_1(z,b) \chi_{(2^{m-b-2z}, 1^{2z})}^{\Sp(2m)} \allowdisplaybreaks\\
&=\frac 1 2 \int_{C_1} \Delta(\gamma_1) + \frac 1 2 \int_{C_1} \Delta(J\gamma_1) \allowdisplaybreaks \\
&= \frac 1 2 (x_1 \cdots x_m)^2 \sum_{b=0}^m \sum_{z=0}^{\lfloor \frac {m-b} 2 \rfloor} \widetilde\eta_1(z,b) \chi_{(2^{m-b-2z}, 1^{2z})}^{\Sp(2m)}+ \frac 1 2 \prod_{i=1}^m(1-2x_i^2+x_i^4)
. \end{align*}
Since $\widetilde\theta_1=\theta_1= \frac 1 2 \eta_1 +\frac 1 2 \psi_1$ and $\eta_1=\widetilde\eta_1$ by definition, we obtain
\[ \prod_{i=1}^m (x_i^2-2 + x_i^{-2}) = \sum_{b=0}^m \sum_{z=0}^{\lfloor \frac {m-b} 2 \rfloor} \psi_1(z,b) \chi_{(2^{m-b-2z}, 1^{2z})}^{\Sp(2m)},\] as desired.

Assume that $n=2$. Using \eqref{coset-234} along with \eqref{phi-cn}, \eqref{theta-11} and \eqref{kap}, we obtain
\begin{align*} &\int_{J(C_2)}  \Delta(\gamma)  = (x_1 \cdots x_m)^2 \sum_{b=0}^m \sum_{z=0}^{\lfloor \frac {m-b} 2 \rfloor} \widetilde\theta_2(z,b) \chi_{(2^{m-b-2z}, 1^{2z})}^{\Sp(2m)} \allowdisplaybreaks\\
 & =\frac 1 2  \int_{C_2} \Delta(\gamma_2)+\frac 1 2  \int_{C_2} \Delta(J\gamma_2)= \frac 1 2 \int_{C_2} \Delta(\gamma_2) + \frac 1 4 \int_{C_1} \Delta(J\gamma_1)  + \frac 1 4 \int_{C_1} \Delta(J\pmb{\zeta}_4\gamma_1) \allowdisplaybreaks\\
 &= (x_1 \cdots x_m)^2 \sum_{b=0}^m \sum_{z=0}^{\lfloor \frac {m-b} 2 \rfloor} \left( \tfrac 1 2 \widetilde\eta_2+ \tfrac 1 4 \psi_1 \right ) \chi_{(2^{m-b-2z}, 1^{2z})}^{\Sp(2m)}
+\frac 1 4 \prod_{i=1}^m(1+2x_i^2+x_i^4)  . \end{align*}
Since $\widetilde\theta_2 = \frac 1 2 \theta_1 + \frac 1 2 \theta_2 = \frac 1 4 \eta_1 + \frac 1 4 \eta_2 + \frac 1 4 \psi_1 + \frac 1 4 \psi_2$ and $\widetilde\eta_2= \frac 1 2 \eta_1 + \frac 1 2 \eta_2$, we obtain
\[ \prod_{i=1}^m (x^2+2+x^{-2}) = \sum_{b=0}^m \sum_{z=0}^{\lfloor \frac {m-b} 2 \rfloor} \psi_2(z,b) \chi_{(2^{m-b-2z}, 1^{2z})}^{\Sp(2m)}. \]

When $n=3$, similar computation yields
\[ \int_{J(C_3)}\Delta(\gamma_3) = (x_1 \cdots x_m)^2 \sum_{b=0}^m \sum_{z=0}^{\lfloor \frac {m-b} 2 \rfloor} \left ( \tfrac 1 2 \widetilde\eta_3+ \tfrac 1 6 \psi_1 \right ) \chi_{(2^{m-b-2z}, 1^{2z})}^{\Sp(2m)}+\frac 1 3 \prod_{i=1}^m (1+x^2+x^4) .  \]
Since $\widetilde\theta_3 = \frac 1 3 \theta_1 + \frac 2 3 \theta_3 = \frac 1 6 \eta_1 + \frac 1 3 \eta_3 + \frac 1 6 \psi_1 + \frac 1 3 \psi_3$ and $\widetilde\eta_3 = \frac 1 3 \eta_1 + \frac 2 3 \eta_3$, we obtain
\[ \prod_{i=1}^m (x^2+1+x^{-2}) =  \sum_{b=0}^m \sum_{z=0}^{\lfloor \frac {m-b} 2 \rfloor} \psi_3(z,b) \chi_{(2^{m-b-2z}, 1^{2z})}^{\Sp(2m)}. \]

When $n=4$, it follows from
\[ \int_{J(C_4)} \Delta(\gamma_4) =(x_1 \cdots x_m)^2 \sum_{b=0}^m \sum_{z=0}^{\lfloor \frac {m-b} 2 \rfloor} \left ( \tfrac 1 2 \widetilde\eta_4+ \tfrac 1 8 \psi_1+\tfrac 1 8 \psi_2 \right ) \chi_{(2^{m-b-2z}, 1^{2z})}^{\Sp(2m)}    + \frac 1 4 \prod_{i=1}^m (1+x^4) \]
that
\[ \prod_{i=1}^m (x^2+x^{-2}) =  \sum_{b=0}^m \sum_{z=0}^{\lfloor \frac {m-b} 2 \rfloor} \psi_4(z,b) \chi_{(2^{m-b-2z}, 1^{2z})}^{\Sp(2m)}. \]

Finally, when $n=6$, we have
\begin{align*} \int_{J(C_6)} \Delta(\gamma_6) &=(x_1 \cdots x_m)^2 \sum_{b=0}^m \sum_{z=0}^{\lfloor \frac {m-b} 2 \rfloor} \left (  \tfrac 1 2 \widetilde\eta_6+ \tfrac 1 {12} \psi_1+ \tfrac 1 {12} \psi_2+\tfrac 1 6 \psi_3 \right ) \chi_{(2^{m-b-2z}, 1^{2z})}^{\Sp(2m)} \\ & \phantom{LLLLLLLLLLLLLLLLL} +\frac 1 6 \prod_{i=1}^m (1-x^2+x^4),  \end{align*}
and obtain
\[ \prod_{i=1}^m (x^2 -1 + x^{-2}) =\sum_{b=0}^m \sum_{z=0}^{\lfloor \frac {m-b} 2 \rfloor} \psi_6(z,b) \chi_{(2^{m-b-2z}, 1^{2z})}^{\Sp(2m)}. \qedhere \]
\end{proof}

The following identities are obtained from computations in the proof above. These identities will be used in the next subsection.
\begin{corollary} \label{cor-def-psi}
For $n=1,2,3,4,6$, we have \begin{equation} \label{cnj} \int_{C_n} \prod_{i=1}^m \det(I+x_i J\gamma)d\gamma = (x_1 \cdots x_m)^2 \sum_{b=0}^m \sum_{z=0}^{\lfloor \frac {m-b} 2 \rfloor} \widetilde \psi_n(z,b) \chi_{(2^{m-b-2z}, 1^{2z})}^{\Sp(2m)}, \end{equation}
where we define
\begin{align*}
\widetilde\psi_1 &:= \psi_1,&
\widetilde\psi_2 &:= \tfrac 1 2 \psi_1 + \tfrac 1 2 \psi_2 ,&
\widetilde\psi_3 &:= \tfrac 1 3 \psi_1 + \tfrac 2 3 \psi_3 ,\\
\widetilde\psi_4 &:= \tfrac 1 4 \psi_1 + \tfrac 1 4 \psi_2 + \tfrac 1 2 \psi_4 ,&
\widetilde\psi_6 &:= \tfrac 1 6 \psi_1 + \tfrac 1 6 \psi_2 + \tfrac 1 3 \psi_3+ \tfrac 1 3 \psi_6 .
\end{align*}
\end{corollary}

We also obtain the following identity from \eqref{gen-1-id-2}:
\begin{corollary}
For  any  $m \in \mathbb Z_{\ge 1}$, we have
\[  \prod_{i=1}^m (x^2+x^{-2}) =  \sum_{b=0}^m \sum_{z=0}^{\lfloor \frac {m-b} 2 \rfloor} \psi_4(z,b) \chi_{(2^{m-b-2z}, 1^{2z})}^{\Sp(2m)}= \sum_{j=0}^{\lfloor \frac m 2 \rfloor} (-1)^j \chi_{(1^{m-2j})}^{\Sp(2m)}(x_1^{\pm 2}, \dots, x_m^{\pm 2}). \]

\end{corollary}

\subsection{Groups $J(D_n)$}

We have the decompositions
\[  J(D_n) = D_n \sqcup JD_n \quad \text{ and } \quad D_n= C_n \sqcup \mathbf j C_n, \]
where $ \mathbf j = \scriptsize { \begin{pmatrix} 0 & 1 & 0 & 0 \\ -1 & 0 & 0&0 \\ 0 &0&0&1\\ 0& 0&-1&0 \end{pmatrix} }$.
For any $\gamma \in C_n$, $n=2,3,4,6$, direct computation shows
\begin{equation*} 
\det (I + xJ\mathbf{j}\gamma) = 1+2x^2  +x^4 .
\end{equation*}

Using \eqref{dn-eta}, \eqref{psi-kappa} and \eqref{cnj}, we compute
\begin{align*}
\int_{J(D_n)} \Delta(\gamma) &= \frac 1 2 \int_{D_n} \Delta(\gamma) + \frac 1 2 \int_{D_n}\Delta(J\gamma) = \frac 1 2 \int_{D_n} \Delta(\gamma) + \frac 1 4 \int_{C_n}\Delta(J\gamma)+\frac 1 4 \int_{C_n}\Delta(J\mathbf j \gamma) \\ &=(x_1 \cdots x_m)^2 \sum_{b=0}^m \sum_{z=0}^{\lfloor \frac {m-b} 2 \rfloor} \left(\tfrac 1 4  \widetilde\eta_n + \tfrac 1 4 \eta_2 + \tfrac 1 4 \widetilde \psi_n + \tfrac 1 4 \psi_2 \right ) \chi_{(2^{m-b-2z}, 1^{2z})}^{\Sp(2m)}.
\end{align*}
Then we have
\[\tfrac 1 4  \widetilde\eta_n + \tfrac 1 4 \eta_2 + \tfrac 1 4 \widetilde \psi_n + \tfrac 1 4 \psi_2
=\tfrac 1 2 \widetilde\theta_n+\tfrac 1 2  \theta_2, \]
which are the same as the formulae in Table \ref{tab-1} for $n=2,3,4,6$. This proves Theorem \ref{thm-genus-2} for the groups $J(D_n)$.

\subsection{Groups $J(T)$ and $J(O)$}

We have the decompositions
\[  J(T) = T \sqcup JT \quad \text{ and } \quad J(O) = O \sqcup JO .\]
As we have noticed in Sections \ref{subsec-t} and \ref{subsec-o}, elements of each coset $S$ of $C_1$ in $T$ or $O$ have the same characteristic polynomials of the elements of $\pmb{\zeta}_{2n}C_1$ for some $n$.
Then the coset $JS$ in $JT$ or $JO$ has the same characteristic polynomials as those of $J\pmb{\zeta}_{2n}C_1$.
It follows from \eqref{res-t}, \eqref{res-o} and \eqref{psi-kappa} that
\begin{align*} \mathfrak m_{\tilde \lambda}(J(T)) &= \tfrac 1 2 (\tfrac 1 {12} \eta_1 + \tfrac 1 4 \eta_2 + \tfrac 2 3 \eta_3) + \tfrac 1 2 (\tfrac 1 {12} \psi_1 + \tfrac 1 4 \psi_2 + \tfrac 2 3 \psi_3) = \tfrac 1 {12} \theta_1 + \tfrac 1 4 \theta_2 + \tfrac 2 3 \theta_3, \\
\mathfrak m_{\tilde \lambda}(J(O))&= \tfrac 1 2 (\tfrac 1 {24} \eta_1 + \tfrac 3 8 \eta_2 + \tfrac 1 3 \eta_3 +\tfrac 1 4 \eta_4) + \tfrac 1 2 (\tfrac 1 {24} \psi_1 + \tfrac 3 8 \psi_2 + \tfrac 1 3 \psi_3 +\tfrac 1 4 \psi_4)\\ &=\tfrac 1 {24} \theta_1 + \tfrac 3 8 \theta_2 + \tfrac 1 3 \theta_3 +\tfrac 1 4 \theta_4  .
\end{align*}
This proves Theorem \ref{thm-genus-2} for $J(T)$ and $J(O)$.

\subsection{Groups $C_{n,1}$}
Recall $C_{n,1}=\langle \U(1), J\pmb{\zeta}_{2n} \rangle$, $n=2,4,6$. The group $C_{n,1}$ has the subgroup $\langle \U(1), \pmb{\zeta}_{n} \rangle$ of index $2$ which is isomorphic to $C_{n/2}$. Thus we have the decomposition
\[ C_{n,1} = C_{n/2} \sqcup J\pmb{\zeta}_{2n}C_{n/2} . \]
It follows from \eqref{coset-234} that
\[ J \pmb{\zeta}_8 C_2=J \pmb{\zeta}_8 C_1 \sqcup J \pmb{\zeta}_8^3 C_1 \quad \text{ and } \quad J \pmb{\zeta}_{12} C_3=J \pmb{\zeta}_{12} C_1 \sqcup J \pmb{\zeta}_{12}^5 \sqcup J \pmb{\zeta}_4 C_1.\]

Since
\begin{align*}
\int_{C_{n,1}} \Delta(\gamma)  &= \frac 1 2 \int_{C_{n/2}} \Delta(\gamma)+ \frac 1 2 \int_{C_{n/2}}\Delta(J\pmb{\zeta}_{2n}\gamma),
\end{align*}
we use \eqref{phi-cn} and \eqref{psi-kappa} to obtain
\begin{align*}
\mathfrak m_{\tilde \lambda}(C_{2,1}) &= \tfrac 1 2 \widetilde \eta_1 + \tfrac 1 2 \psi_2, \allowdisplaybreaks\\
\mathfrak m_{\tilde \lambda}(C_{4,1}) &= \tfrac 1 2 \widetilde \eta_2 + \tfrac 1 2 \psi_4, \allowdisplaybreaks \\
\mathfrak m_{\tilde \lambda}(C_{6,1}) &= \tfrac 1 2 \widetilde \eta_3 + \tfrac 1 6 \psi_2 + \tfrac 1 3 \psi_6.
\end{align*}
Thus we have proved Theorem \ref{thm-genus-2} for $C_{n,1}$.

\subsection{Groups $D_{n,1}$}
We have the decompositions
\[ D_{n,1} = C_{n,1} \sqcup \mathbf j C_{n,1} \quad \text{ and } \quad \mathbf j C_{n,1} = \mathbf j C_{n/2} \sqcup \mathbf j J \pmb{\zeta}_{2n} C_{n/2} . \]
Then we have
\begin{align*}
\int_{D_{n,1}} \Delta(\gamma) &= \frac 1 2 \int_{C_{n,1}} \Delta(\gamma)  + \frac 1 2 \int_{C_{n,1}}  \Delta(\mathbf j \gamma) \allowdisplaybreaks\\
&= \frac 1 2 \int_{C_{n,1}} \Delta(\gamma) + \frac 1 4 \int_{C_{n/2}} \Delta(\mathbf j \gamma) + \frac 1 4 \int_{C_{n/2}} \Delta(\mathbf j J \pmb{\zeta}_{2n} \gamma).
\end{align*}

For any $\gamma \in \mathbf j J \pmb{\zeta}_{2n} C_{n/2}$, one can check
\[  \det(I+x \gamma) = 1+2x^2+x^4. \]
Thus it follows from \eqref{jgam} and \eqref{psi-kappa} that
\[ \mathfrak m_{\tilde \lambda}(D_{n,1})= \tfrac 1 2 \mathfrak m_{\tilde \lambda} (C_{n,1}) + \tfrac 1 4 \eta_2 + \tfrac 1 4 \psi_2. \]
These are the same as the formulae for $D_{n,1}$ in Theorem \ref{thm-genus-2}.

\subsection{Groups $D_{n,2}$}
We have the decomposition
\[  D_{n,2} = C_n \sqcup J\mathbf j C_n , \quad n=3,4,6. \]
For any $\gamma \in J \mathbf j  C_{n}$, we obtain
\[  \det(I+x \gamma) = 1+2x^2+x^4. \]
Since we have
\begin{align*}
\int_{D_{n,2}} \Delta(\gamma) &= \frac 1 2 \int_{C_{n}} \Delta(\gamma)  + \frac 1 2 \int_{C_{n}} \Delta(J\mathbf j\gamma),
\end{align*}
we get
\[ \mathfrak m_{\tilde \lambda}(D_{n,2})= \tfrac 1 2 \widetilde \eta_n+ \tfrac 1 2 \psi_2. \]
These are the same as the formulae for $D_{n,2}$ in Theorem \ref{thm-genus-2}.

\subsection{Group $O_{1}$}
Recall $O_1 = \langle T, J Q_2 \rangle$  with $Q_2$ in \eqref{q2}.
There are $48$ cosets of $C_1$ in $O_1$. Half of them belong to the subgroup $T$. In the other half, $12$ of the $24$ cosets have the same characteristic polynomial as $J\pmb{\zeta}_4 C_1$ and the remaining $12$ cosets have the same as $J\pmb{\zeta}_8 C_1$. By \eqref{psi-kappa}, we obtain
\[ \mathfrak m_{\tilde \lambda}(O_1) = \tfrac 1 2 \mathfrak m_{\tilde \lambda}(T) + \tfrac 1 4 \psi_2 + \tfrac 1 4 \psi_4. \]
This coincides with the formula in Table \ref{tab-1} for $O_1$.

\subsection{Groups $E_{n}$}

We have $E_n= \langle \SU(2), e^{\pi i/n} \rangle$, $n=1,2,3,4,6$, where $e^{\pi i/n}$ is identified with  \[\mathrm{diag}(e^{\pi i/n},e^{\pi i/n},e^{-\pi i/n},e^{-\pi i/n}).\] The number $\mathfrak m_{\tilde \lambda}(E_n)$ is equal to the number of independent weight vectors $v_\mu$ with weight $\mu$ in $V^{\Sp(4)}_{\tilde \lambda}$ that generate the trivial representation of $\SU(2)$ and satisfy $\mu (\hat h_1+\hat h_2) \equiv 0$ (mod $2n$), equivalently,
the number of independent weight vectors $v_\mu$ such that  $\mu (\hat h_1+\hat h_2) \equiv 0$ (mod $2n$) and $e_1 v_\mu = f_1 v_\mu=0$.
The numbers $\mathfrak m_{\tilde \lambda}(E_n)$ are all calculated in Proposition \ref{prop:en-1}, and they coincide with the formulae in Table \ref{tab-1}. This proves Theorem \ref{thm-genus-2} for $E_n$.

\subsection{Groups $J(E_n)$}

The number $\mathfrak m_{\tilde \lambda}(J(E_1))$ is equal to the number of independent {\em $J$-fixed} weight vectors $v_\mu$ with weight $\mu$ in $V^{\Sp(4)}_{\tilde \lambda}$ such that $e_1 v_\mu = f_1 v_\mu=0$.
The numbers $\mathfrak m_{\tilde \lambda}(J(E_1))$ are calculated in Proposition \ref{e1}, and they are the same as the formulae in Table \ref{tab-1}.

For the other $J(E_n)$, $n=2,3,4,6$, observe that \begin{equation} \label{jen-det} \det \left (I+xJ \begin{pmatrix} tA & 0 \\ 0 & \overline{tA} \end{pmatrix} \right )= 1 +(2-|\tr A|^2)x^2+x^4\end{equation} for any $A \in \SU(2)$ and $t \in \U(1)$.
Thus \[\int_{E_n} \Delta(J\gamma) = \int_{E_1} \Delta(J\gamma)\qquad \text{ for }n=2,3,4,6. \]

Notice from Propositions \ref{e1} that
\[ \mathfrak m_{\tilde \lambda}(J(E_1)) = \frac 1 2 \mathfrak m_{\tilde \lambda}(E_1) + \frac 1 2 (-1)^z \, \delta( b \equiv_2 0).\]
Since $J(E_1)= E_1 \sqcup JE_1$, we obtain
the following:

\begin{corollary}
For $n=1,2,3,4,6$,
\begin{align}
\int_{E_n} \Delta(J\gamma) &=(x_1 \cdots x_m)^2 \sum_{b=0}^m \sum_{z=0}^{\lfloor \frac {m-b} 2 \rfloor} (-1)^z\delta(b \equiv_2 0)  \, \chi_{(2^{m-b-2z}, 1^{2z})}^{\Sp(2m)} \nonumber \\
&=(x_1 \cdots x_m)^2 \sum_{\ell=0}^{\lfloor \frac m 2 \rfloor} \sum_{z=0}^{\lfloor \frac m 2 \rfloor-\ell}  (-1)^z \chi_{(2^{m-2\ell-2z}, 1^{2z})}^{\Sp(2m)}.\label{jen-co}
\end{align}
\end{corollary}

From $J(E_n)= E_n \sqcup JE_n$, $n=2,3,4,6$, we also obtain
\[ \mathfrak m_{\tilde \lambda}(J(E_n)) = \frac 1 2 \mathfrak m_{\tilde \lambda}(E_n) + \frac 1 2 (-1)^z \, \delta( b \equiv_2 0),\]
which are equal to the formulae in Table \ref{tab-1}.

\medskip

Using the above results on $J(E_n)$, we now prove another identity presented in the introduction.

\begin{theorem}[Theorem \ref{thm-identities-1} (2)] \label{thm-identities-e}
For any $m \in \mathbb Z_{\ge 1}$, the following identity holds:
\begin{align*}
\sum_{k=0}^m \binom{m-k}{\lfloor (m-k)/2 \rfloor}
\sum_{1 \le i_1 < \cdots <i_k \le m} \prod_{j=1}^k (x_{i_j}^2+x_{i_j}^{-2})
&=\sum_{\ell=0}^{\lfloor \frac m 2 \rfloor} \sum_{z=0}^{\lfloor \frac m 2 \rfloor-\ell}  (-1)^z \chi_{(2^{m-2\ell-2z}, 1^{2z})}^{\Sp(2m)}.
\end{align*}
\end{theorem}

\begin{proof}
Recall
\begin{equation} \label{trC} \int_{\SU(2)} |\tr A|^{2k} dA= \mathscr C_k, \end{equation}
where $\mathscr C_k= \frac 1 {k+1} \binom{2k}{k}$ is the $k^{\mathrm {th}}$ Catalan number.
It is known that the second inverse binomial transformations of Catalan numbers are central binomial coefficients. Precisely, we have
\begin{equation} \label{bi-Ca} \binom m {\lfloor m/2 \rfloor} = \sum_{k=0}^m (-1)^k\binom m {k} 2^{m-k} \, \mathscr C_{k} . \end{equation}
Using \eqref{trC} and \eqref{bi-Ca}, we obtain
\begin{align*}& \int_{\SU(2)} (2 -|\tr A|^{2})^m dA = \int_{\SU(2)} \sum_{k=0}^m \binom m {k} 2^{m-k} (-1)^k\, |\tr A|^{2k} dA  \\ &= \sum_{k=0}^m (-1)^k\binom m {k} 2^{m-k} \, \int_{\SU(2)} |\tr A|^{2k} dA =  \binom m {\lfloor m/2 \rfloor }. \end{align*}
Then, using \eqref{jen-det}, we have
\begin{align*}
\int_{E_n} \Delta(J \gamma) & = \int_{\SU(2)} \prod_{i=1}^m (1+(2-|\tr A|^2) x_i^2 +x_i^4)\, dA
\\ & =(x_1 \cdots x_m)^2
\sum_{k=0}^m \binom{m-k}{\lfloor (m-k)/2 \rfloor}
\sum_{1 \le i_1 < \cdots <i_k \le m} \prod_{j=1}^k (x_{i_j}^2+x_{i_j}^{-2}).
\end{align*}

On the other hand, we have from \eqref{jen-co}
\begin{align*}
\int_{E_n} \Delta(J\gamma) &=(x_1 \cdots x_m)^2 \sum_{\ell=0}^{\lfloor \frac m 2 \rfloor} \sum_{z=0}^{\lfloor \frac m 2 \rfloor-\ell}  (-1)^z \chi_{(2^{m-2\ell-2z}, 1^{2z})}^{\Sp(2m)},
\end{align*}
and obtain the desired identity.
\end{proof}

\subsection{Groups $\U(2)$ and $N(\U(2))$}

The number $\mathfrak m_{\tilde \lambda}(\U(2))$ is equal to the number of independent weight vectors $v_\mu$ with weight $\mu$ in $V^{\Sp(4)}_{\tilde \lambda}$ such that $e_1 v_\mu = f_1 v_\mu=0$  and  $\mu=0$, and we see from \eqref{eq:wt} that  such a vector $v_\mu$ exists with multiplicity $1$ if and only if $b$ is even. The number $\mathfrak m_{\tilde \lambda}(N(\U(2))$ is equal to the number of such vectors $v_\mu$ which is fixed by $J$, and it follows
 from \eqref{eq:wt-J} that, when it exits, the vector $v_\mu$ is fixed by $J$ if and only if $J$ is even. These match with   the formulae in Table \ref{tab-1}.

\subsection{Groups $F_*$}

Recall that we have the embedding $\U(1) \times \U(1)$ into $\USp(4)$ given by
\[  (u_1, u_2) \mapsto \mathrm{diag}(u_1, u_2, u_1^{-1}, u_2^{-1}) ,\]
and that
\begin{align*}
\mathtt{a} &= {\tiny \begin{pmatrix}  0&0&1&0 \\  0&1&0&0 \\ -1&0&0&0 \\ 0&0&0&1 \end{pmatrix}},  & \mathtt{b}& = {\tiny \begin{pmatrix}  1&0&0&0 \\  0&0&0&1 \\ 0&0&1&0 \\ 0&-1&0&0 \end{pmatrix}},
&\mathtt{c}& = {\tiny \begin{pmatrix}  0&1&0&0 \\  -1&0&0&0 \\ 0&0&0&1 \\ 0&0&-1&0 \end{pmatrix}}, \\ \mathtt{ab}&=
{\tiny \begin{pmatrix}  0&0&1&0 \\  0&0&0&1 \\ -1&0&0&0 \\ 0&-1&0&0 \end{pmatrix}},  & \mathtt{ac}& = {\tiny \begin{pmatrix}  0&0&0&1 \\  -1&0&0&0 \\ 0&-1&0&0 \\ 0&0&-1&0 \end{pmatrix}}.
\end{align*}

\subsubsection{Group $F$}

We have $F \cong \U(1) \times \U(1)$. The number $\mathfrak m_{\tilde \lambda}(F)$ is equal to the number of independent weight vectors $v_\mu$ such that $\mu =0$.
This number is calculated in Proposition \ref{f-form}, and coincides with the formula in Table \ref{tab-1}.
That is,
\begin{equation} \label{feq}  \int_F \Delta(\gamma) = (x_1 \cdots x_m)^2 \sum_{b=0}^m \sum_{z=0}^{\lfloor \frac {m-b} 2 \rfloor} \xi_1(z,b) \chi_{(2^{m-b-2z}, 1^{2z})}^{\Sp(2m)},\end{equation}
where \[ \xi_1(z,b)= z(b+1) + \lfloor b/2 \rfloor +1. \]

\subsubsection{Group $F_{\mathtt{a}}$}

The number $\mathfrak m_{\tilde \lambda}(F_{\mathtt{a}})$ is equal to the number of {\em $\mathtt{a}$-fixed} independent weight vectors $v_\mu$ such that $\mu =0$.
This number is calculated in Proposition \ref{fa-prop}, and we obtain
\begin{equation} \label{faeq}  \int_{F_{\mathtt{a}}} \Delta(\gamma) = (x_1 \cdots x_m)^2 \sum_{b=0}^m \sum_{z=0}^{\lfloor \frac {m-b} 2 \rfloor} (\tfrac 1 2 \xi_1+\tfrac 1 2 \xi_2) \chi_{(2^{m-b-2z}, 1^{2z})}^{\Sp(2m)}.\end{equation}

Using \eqref{faeq}, we can establish the last identity presented in the introduction.
\begin{theorem}[Theorem \ref{thm-identities-1} (3)] \label{thm-identities-f}
For any $m \in \mathbb Z_{\ge 1}$, the following identity holds:
\begin{align*}
\prod_{i=1}^m (x_i+x_i^{-1})  \sum_{j=0}^{\lfloor \frac m 2 \rfloor} \chi_{(1^{m-2j})}^{\Sp(2m)}
&=\sum_{b=0}^m \sum_{z=0}^{\lfloor \frac {m-b} 2 \rfloor} \xi_2(z,b) \chi_{(2^{m-b-2z}, 1^{2z})}^{\Sp(2m)}.
\end{align*}
\end{theorem}

\begin{proof}
We have the decomposition
\[ F_{\mathtt{a}} = F \sqcup \mathtt{a} F.\]
For $\gamma \in \mathtt{a} F$, we have
\[ \det(I+x \gamma) = (1+x^2)(1+(u+u^{-1})x+x^2), \quad u \in \U(1). \]
Thus we obtain
\begin{align*}
\int_{F_{\mathtt{a}}} \Delta(\gamma)  &= \frac 1 2 \int_{F} \Delta(\gamma) + \frac 1 2 \int_{F} \Delta(\mathtt{a}\gamma)\allowdisplaybreaks \\
&= \frac 1 2 \int_{F} \Delta(\gamma) + \frac 1 2 \prod_{i=1}^m (1+x_i^2) \int_{\U(1)} \prod_{i=1}^m (1+(u+u^{-1})x_i+x_i^2) du \allowdisplaybreaks\\
&=  \frac 1 2 \int_{F} \Delta (\gamma)  + \frac 1 2 (x_1 \cdots x_m) \prod_{i=1}^m (1+x_i^2)  \sum_{j=0}^{\lfloor \frac m 2 \rfloor} \chi_{(1^{m-2j})}^{\Sp(2m)}  ,
\end{align*}
where we use the identity in Theorem \ref{thm-genus-1} for $\U(1)$. It follows from \eqref{feq} and \eqref{faeq} that
\begin{align} \label{last-eq}
\int_{F} \Delta(\mathtt{a}\gamma) &= (x_1 \cdots x_m)^2\sum_{b=0}^m \sum_{z=0}^{\lfloor \frac {m-b} 2 \rfloor} \xi_2(z,b) \chi_{(2^{m-b-2z}, 1^{2z})}^{\Sp(2m)}\\ &= (x_1 \cdots x_m) \prod_{i=1}^m (1+x_i^2)  \sum_{j=0}^{\lfloor \frac m 2 \rfloor} \chi_{(1^{m-2j})}^{\Sp(2m)}, \nonumber \end{align}
and we obtain the desired identity.
\end{proof}

\subsubsection{Group $F_{\mathtt{c}}$}
We have $F_{\mathtt{c}} = F \sqcup \mathtt{c} F$.
For $\gamma \in \mathtt{c} F$, we obtain
\[ \det(I+x \gamma) = 1+(u_1u_2+u_1^{-1}u_2^{-1})x^2+x^4. \]
Thus
\begin{align*}
& \int_{F_{\mathtt{c}}} \Delta(\gamma)= \frac 1 2 \int_{F} \Delta(\gamma)+ \frac 1 2 \int_{F}\Delta(\mathtt{c}\gamma) \allowdisplaybreaks\\
&= \frac 1 2 \int_{F} \Delta(\gamma)+ \frac 1 2  \int_{\U(1)} \int_{\U(1)}  \prod_{i=1}^m (1+(u_1u_2+u_1^{-1}u_2^{-1})x_i^2+x_i^4) du_1du_2 \allowdisplaybreaks\\
&= \frac 1 2 \int_{F} \Delta(\gamma) + \frac 1 2  \int_{\U(1)} \int_{\U(1)}  \prod_{i=1}^m (1+(u_1+u_1^{-1})x_i^2+x_i^4) du_1du_2 \allowdisplaybreaks\\
&= \frac 1 2 \int_{F} \Delta(\gamma)+\frac 1 2  \int_{\U(1)}   \prod_{i=1}^m (1+(u+u^{-1})x_i^2+x_i^4) du ,
\end{align*}
since $du_1$ is translation-invariant. It follows from \eqref{zeta-cn} that
\[ \mathfrak m_{\tilde \lambda}(F_{\mathtt{c}}) = \frac 1 2 \xi_1(z,b) +\frac 1 2 \eta_2(z,b), \] which is the same as the formula in Table \ref{tab-1}.

\subsubsection{Group $F_{\mathtt{ab}}$}
We have $F_{\mathtt{ab}} = F \sqcup \mathtt{ab} F$.
For $\gamma \in \mathtt{ab} F$, we have
\[ \det(I+x \gamma) = (1+x^2)^2. \]
Thus
\begin{align*}
\int_{F_{\mathtt{ab}}} \Delta(\gamma)  &= \frac 1 2 \int_{F} \Delta(\gamma)  + \frac 1 2 \int_{F} \Delta(\mathtt{ab} \gamma)\allowdisplaybreaks\\
&= \frac 1 2 \int_{F} \Delta(\gamma)+ \frac 1 2 \prod_{i=1}^m (1+x_i^2)^2.
\end{align*}
It follows from \eqref{f-form} and \eqref{psi-kappa} that
\[ \int_{F_{\mathtt{ab}}} \Delta(\gamma) = (x_1 \cdots x_m)^2\sum_{b=0}^m \sum_{z=0}^{\lfloor \frac {m-b} 2 \rfloor} (\tfrac 1 2 \xi_1 + \tfrac 1 2 \psi_2) \chi_{(2^{m-b-2z}, 1^{2z})}^{\Sp(2m)}. \]
This coincides with the formula in Table \ref{tab-1}.

\subsubsection{Group $F_{\mathtt{ac}}$}

There are four cosets of $F$ in $F_{\mathtt{ac}}$ represented by $1,\mathtt{ac},(\mathtt{ac})^2,(\mathtt{ac})^3$.
For $\gamma \in \mathtt{ac}F$ or $(\mathtt{ac})^3F$, we have
\[ \det(I+x \gamma) = 1+x^4. \]
For $\gamma \in (\mathtt{ac})^2F =\mathtt{ab}F$, we have
\[ \det(I+x \gamma) = (1+x^2)^2. \]
Using \eqref{psi-kappa}, we have
\begin{align*} \int_{F_{\mathtt{ac}}}\Delta(\gamma)&= \frac 1 4 \int_{F} \Delta(\gamma) +  \frac 1 4 \prod_{i=1}^m (1+x_i^2)^2 +  \frac 1 2 \prod_{i=1}^m (1+x_i^4)\allowdisplaybreaks\\
&= (x_1 \cdots x_m)^2\sum_{b=0}^m \sum_{z=0}^{\lfloor \frac {m-b} 2 \rfloor} (\tfrac 1 4 \xi_1 + \tfrac 1 4 \psi_2 + \tfrac 1 2 \psi_4) \chi_{(2^{m-b-2z}, 1^{2z})}^{\Sp(2m)} .\end{align*}
This coincides with the formula in Table \ref{tab-1}.

\subsubsection{Group $F_{\mathtt{a},\mathtt{b}}$}

There are four cosets of $F$ in $F_{\mathtt{a},\mathtt{b}}$ represented by $1,\mathtt{a},\mathtt{b},\mathtt{ab}$.
For $\gamma \in \mathtt{a} F$ or $\mathtt{b} F$, we have
\[ \det(I+x \gamma) = (1+x^2)(1+(u+u^{-1})x+x^2), \qquad u \in \U(1). \]
For $\gamma \in \mathtt{ab}F$, we have
\[ \det(I+x \gamma) = (1+x^2)^2. \]
Thus it follows from \eqref{psi-kappa} and \eqref{last-eq} that
\[ \mathfrak m_{\tilde \lambda}(F_{\mathtt{a},\mathtt{b}}) = \frac 1 4 \xi_1(z,b) +\frac 1 4 \psi_2(z,b) +  \frac 1 2 \xi_2(z,b), \] which is the same as the formula in Table \ref{tab-1}.

\subsubsection{Group $F_{\mathtt{ab},\mathtt{c}}$}

There are four cosets of $F$ in $F_{\mathtt{ab},\mathtt{b}}$ represented by $1,\mathtt{ab},\mathtt{c},\mathtt{abc}$.
For $\gamma \in \mathtt{c}F$ or $\mathtt{abc}F$, we have
\[ \det(I+x \gamma) = 1+(u_1u_2^{-1}+u_1^{-1}u_2)x^2+x^4. \]
Thus using the results on $F_{\mathtt{ab}}$ and $F_{\mathtt{c}}$, we have
\[ \mathfrak m_{\tilde \lambda}(F_{\mathtt{ab},\mathtt{c}}) = \frac 1 4 \xi_1(z,b) +\frac 1 4 \psi_2(z,b) +  \frac 1 2 \eta_2(z,b), \] which is the same as the formula in Table \ref{tab-1}.

\subsubsection{Group $F_{\mathtt{a},\mathtt{b},\mathtt{c}}$}

There are eight cosets of $F$ in $F_{\mathtt{a},\mathtt{b},\mathtt{c}}$ represented by $1,\mathtt{ac}$, $(\mathtt{ac})^2$, $(\mathtt{ac})^3$, $\mathtt{a}$, $\mathtt{b}$, $\mathtt{c}$, $\mathtt{abc}$.
Using the results on $F_{\mathtt{a}}$, $F_{\mathtt{b}}$, $F_{\mathtt{ac}}$ and $F_{\mathtt{abc}}$, we have
\[ \mathfrak m_{\tilde \lambda}(F_{\mathtt{a},\mathtt{b},\mathtt{c}}) = \frac 1 8 \xi_1(z,b) +\frac 1 4 \xi_2(z,b)+ \frac 1 8  \psi_2(z,b) + \frac 1 4 \psi_4(z,b)+ \frac 1 4 \eta_2(z,b), \] which coincides with the formula in Table \ref{tab-1}.

\subsection{Groups $G_{1,3}$ and $N(G_{1,3}
)$}
Recall that $G_{1,3} \cong \U(1) \times \SU(2)$ and
\[ N(G_{1,3}) = \langle G_{1,3}, \mathtt a \rangle, \qquad \text{ where }  \quad \mathtt a = {\tiny \begin{pmatrix}  0&0&1&0 \\  0&1&0&0 \\ -1&0&0&0 \\ 0&0&0&1 \end{pmatrix}} . \] The number $\mathfrak m_{\tilde \lambda}(G_{1,3})$ is equal to the number of independent weight vectors $v_\mu$ in $V^{\Sp(4)}_{\tilde \lambda}$ such that $\mu=0$ and $\hat e_2$ and $\hat f_2$ act trivially, and the number $\mathfrak m_{\tilde \lambda}(N(G_{1,3}))$ is equal to the number of
$\mathtt{a}$-fixed such vectors.
From Proposition \ref{prop-u1-su2}, we obtain
 the same numbers as in the formulae in Table \ref{tab-1} for $G_{1,3}$ and $N(G_{1,3})$.

\subsection{Groups $G_{3,3}$ and $N(G_{3,3})$}
Recall that $G_{3,3} \cong \SU(2) \times \SU(2)$ and
\[ N(G_{3,3}) = \langle G_{3,3}, J \rangle, \qquad \text{ where } \quad J = {\tiny \begin{pmatrix}  0&0&0&1 \\  0&0&-1&0 \\ 0&-1&0&0 \\ 1&0&0&0 \end{pmatrix}} . \]
In this case, the number $\mathfrak m_{\tilde \lambda}(G_{3,3})$ is equal to the number of independent weight vectors $v_\mu$ in $V^{\Sp(4)}_{\tilde \lambda}$ such that $\mu=0$ and $\hat e_i$ and $\hat f_i$ act trivially for $i=1,2$, and the number $\mathfrak m_{\tilde \lambda}(N(G_{1,3}))$ is equal to the number of
$J$-fixed such vectors.
From Proposition \ref{prop-su2-su2}, we obtain
 the same numbers as those in Table \ref{tab-1} for $G_{3,3}$ and $N(G_{3,3})$.

\medskip

We have checked all the formulae in Table \ref{tab-1}. This completes our proof of Theorem \ref{thm-genus-2}. \qed

\section{Branching rules} \label{sec-branching}

In this section we study branching rules that arise in relation to the Sato--Tate groups. The results are essentially used in Section \ref{sec-g-2} for the proof of Theorem \ref{thm-genus-2}. We present these branching rules for Lie algebras. This will be more consistent with standard notations for crystals which are our main tools in this section.
We refer the readers to \cite{BuSch, HK, Kash} for a theory of crystals.

\smallskip

The Cartan types of Lie algebras $\mathfrak{sp}_4$ and $\mathfrak{sl}_2$ will be denoted by $\mathtt C_2$ and $\mathtt A_1$, respectively. A partition $(a,b)$ of length $\le 2$ will be considered as a weight of $\mathtt C_2$ type corresponding to $a \epsilon_1 + b \epsilon_2$. The irreducible representation of $\mathfrak{sp}_4(\mathbb C)$ with highest weight $(a,b)$ will be denoted by $V_{\mathtt C_2}(a,b)$ and its crystal by $B_{\mathtt C_2}(a,b)$. Similar notations will be adopted for other types.
In this section, since we are mainly interested in the Sato--Tate groups,
\[ \text{we assume that $a-b$ is even,  and write $z:= (a-b)/2$} . \]
This assumption is justified by \eqref{con-coeff}.

\subsection{From $\mathtt C_2$ to $\mathtt A_1 \times \mathtt A_1$}
For a partition $(a,b)$ of length $\le 2$, define a set of pairs of integers
\begin{align} \label{eq: Phi}
\Phi_{\mathtt A_1 \times \mathtt A_1}(a,b):=\{ (a-r-s,b-r+s)  \  |   \ 0 \le r \le b, \quad 0 \le s \le a-b  \} .
\end{align}
Clearly, we have
$$|\Phi_{\mathtt A_1 \times \mathtt A_1}(a,b)| =
(a-b+1)(b+1). $$

\begin{example}
\hfill
\begin{enumerate}
\item[{\rm (a)}] $\Pset(2,2)=\{  (2,2),(1,1),(0,0) \}$.
\item[{\rm (b)}] $\Pset(4,2)=\{  (4,2),(3,1),(2,0), \ (3,3),(2,2),(1,1), \ (2,4),(1,3),(0,2)  \}$.
\end{enumerate}
\end{example}

Let $\mathtt{t}_1=(\frac{1}{2},-\frac{1}{2},0,0 )$ and $\mathtt{t}_2=(0,0,\frac{1}{2},-\frac{1}{2})$ be weights of type $\mathtt A_1 \times \mathtt A_1$.
For non-negative integers $p$ and $q$, the crystal $B_{\mathtt A_1 \times \mathtt A_1}(p \mathtt{t}_1+q \mathtt{t}_2)$ can be realized as the set of one-row standard tableaux with entries $1$, $\Oone$, $2$ or $\Otwo$
having the order $1 \prec \Oone \prec 2 \prec \Otwo$ such that
\begin{align}\label{eq: tableaux realization of A1A1 crystal}
\text{the total number of $1$ and $\Oone$ is $p$ and the total number of $2$ and $\Otwo$ is $q$.}
\end{align}
 For instance,
$$ B_{\mathtt A_1 \times \mathtt A_1}(3\mathtt{t}_1+ \mathtt{t}_2) = \{1112,11\Oone2,1\Oone\Oone2,\Oone \Oone \Oone 2, 111\Otwo,11\Oone\Otwo,1\Oone\Oone\Otwo,\Oone\Oone\Oone\Otwo\}.$$
The cardinality of $B_{\mathtt A_1 \times \mathtt A_1}(p \mathtt{t}_1+q \mathtt{t}_2)$ is $(p+1)(q+1)$.

Note that,  $B_{\mathtt A_1 \times \mathtt A_1}(p \mathtt{t}_1+q \mathtt{t}_2)$  is \emph{minuscule}; i.e., each weight space in  $B_{\mathtt A_1 \times \mathtt A_1}(p \mathtt{t}_1+q \mathtt{t}_2)$
is $1$-dimensional.
The crystal graph of $B_{\mathtt A_1 \times \mathtt A_1}(\mathtt{t}_1+\mathtt{t}_2)$ is given by
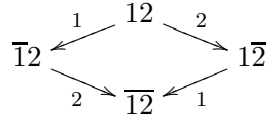
\begin{align*}
\xymatrix@R=0.1em{
 & {\scriptsize 12}  \ar[dr]^2 \ar[dl]_1  \allowdisplaybreaks \\
 {\scriptsize \Oone  2}  \ar[dr]_2 &&  {\scriptsize  1  \Otwo} \ar[dl]^1 \allowdisplaybreaks \\
&  {\scriptsize  \Oone  \Otwo}
}
\end{align*}

Clearly, the Kashiwara operators $\tilde{e}_1$ and $\tilde{e}_2$ commute. For a pair $(a_1,a_2)$ of non-negative integers and $i=1,2$, we have
\begin{equation}\label{eq: A1 A1-1}
\begin{aligned}
& B_{\mathtt A_1 \times \mathtt A_1}(a_1 \mathtt{t}_1+a_2 \mathtt{t}_2) \otimes B_{\mathtt A_1 \times \mathtt A_1}(\mathtt{t}_i) \allowdisplaybreaks\\
& =  B_{\mathtt A_1 \times \mathtt A_1}( (a_1+\delta_{i,1}) \mathtt{t}_1+ (a_2+\delta_{i,2}) \mathtt{t}_2)   \oplus \delta({a_i \ge 1})  B_{\mathtt A_1 \times \mathtt A_1}( (a_1-\delta_{i,1}) \mathtt{t}_1+ (a_2-\delta_{i,2})  \mathtt{t}_2) \\
\end{aligned}
\end{equation}
and
\begin{equation}\label{eq: A1 A1-2}
\begin{aligned}
& B_{\mathtt A_1 \times \mathtt A_1}(a_1 \mathtt{t}_1+a_2 \mathtt{t}_2) \otimes B_{\mathtt A_1 \times \mathtt A_1}(\mathtt{t}_1+\mathtt{t}_2) \allowdisplaybreaks\\
& =  B_{\mathtt A_1 \times \mathtt A_1}( (a_1+1) \mathtt{t}_1+ (a_2+1) \mathtt{t}_2) \oplus \delta({a_2 >0})  B_{\mathtt A_1 \times \mathtt A_1}( (a_1+1) \mathtt{t}_1+ (a_2-1) \mathtt{t}_2) \allowdisplaybreaks\\
& \phantom{LLLLLLL} \oplus\delta({a_1>0})  B_{\mathtt A_1 \times \mathtt A_1}( (a_1-1) \mathtt{t}_1+ (a_2+1) \mathtt{t}_2) \allowdisplaybreaks \\
& \phantom{LLLLLLLLLLLLL} \oplus \delta({a_1a_2>0})  B_{\mathtt A_1 \times \mathtt A_1}( (a_1-1) \mathtt{t}_1+ (a_2-1)\mathtt{t}_2).
\end{aligned}
\end{equation}

On the other hand, for a highest weight crystal $B_{\mathtt C_2}(a,b)$ $(a \ge b \ge 0)$, we have
\begin{equation}\label{eq: C_2-1}
\begin{aligned}
B_{\mathtt C_2}(a,b) \otimes B_{\mathtt C_2}(1,0) = \ & B_{\mathtt C_2}(a+1,b) \oplus  \delta(a>b)B_{\mathtt C_2}(a,b+1) \\
& \hspace{5ex} \oplus  \delta(b>0)B_{\mathtt C_2}(a,b-1)  \oplus \delta(a-1 \ge b)B_{\mathtt C_2}(a-1,b)
\end{aligned}
\end{equation}
and
\begin{equation}\label{eq: C_2-2}
\begin{aligned}
B_{\mathtt C_2}(a,b) \otimes B_{\mathtt C_2}(1,1) = \ & B_{\mathtt C_2}(a+1,b+1) \oplus \delta({b >0}) B_{\mathtt C_2}(a+1,b-1)\\&   \hspace{5ex}  \oplus \delta({a >b})B_{\mathtt C_2}(a,b)
 \oplus \delta({a -2 \ge b})B_{\mathtt C_2}(a-1,b+1)\\  & \hspace{10ex}  \oplus \delta({b >0})B_{\mathtt C_2}(a-1,b-1).
\end{aligned}
\end{equation}

\begin{proposition} \label{prop: A1xA1}
For each partition $(a,b)$, we have
$$  B_{\mathtt C_2}(a,b)\vert_{\mathtt A_1 \times \mathtt A_1} \simeq  \soplus_{(p,q)\in \Pset(a,b)} B_{\mathtt A_1 \times \mathtt A_1}(p \mathtt{t}_1+q \mathtt{t}_2).$$
\end{proposition}

\begin{proof}
For simplicity we write $B_{\mathtt A}(p,q)$ for $B_{\mathtt A_1 \times \mathtt A_1}(p \mathtt{t}_1+q \mathtt{t}_2)$.
By direct calculation, one can see that
$$B_{\mathtt C_2}(1,0)\vert_{\mathtt A_1 \times \mathtt A_1} =  B_{\mathtt A}(1,0) \oplus  B_{\mathtt A}(0,1). $$ 
We first apply an induction argument on $a$ using \eqref{eq: A1 A1-1} and \eqref{eq: C_2-1} to obtain the formula for  $B_{\mathtt C_2}(a+1,b)$.
Namely,
$B_{\mathtt C_2}(a,b) \otimes B_{\mathtt C_2}(1,0) $ can be replaced with
$$ \left( \soplus_{(p,q)\in \Pset(a,b)} B_{\mathtt A}(p,q)  \right) \otimes \left( B_{\mathtt A}(1,0) \oplus  B_{\mathtt A}(0,1) \right) $$
Then, for the composition factors in RHS of \eqref{eq: C_2-1} except $B_{\mathtt C_2}(a+1,b)$, we apply the induction on $a$.
Now we use \eqref{eq: A1 A1-2} and \eqref{eq: C_2-2} to obtain the formula for  $B_{\mathtt C_2}(a+1,b+1)$.
More precisely,
$B_{\mathtt C_2}(a,b) \otimes B_{\mathtt C_2}(1,1) $ can be replaced with
$$ \left( \soplus_{(p,q)\in \Pset(a,b)} B_{\mathtt A}(p,q)  \right) \otimes \left( B_{\mathtt A}(1,1) \oplus  B_{\mathtt A}(0,0) \right).$$
Then, for the composition factors in RHS of \eqref{eq: C_2-2} except $B_{\mathtt C_2}(a+1,b+1)$, we apply the induction on $a+b$.  We obtain
our assertion for $B_{\mathtt C_2}(a+1,b+1)$ by comparing LHS and RHS of \eqref{eq: C_2-2}.
\end{proof}

\begin{corollary} \label{cor-c2-a1a1} \hfill
\begin{enumerate}
\item[{\rm (a)}] For $(p,q)\in \Pset(a,b)$, the
multiplicity of $B_{\mathtt A_1 \times \mathtt A_1}(p \mathtt{t}_1+q \mathtt{t}_2)$ in $ B_{\mathtt C_2}(a,b)\vert_{\mathtt A_1 \times \mathtt A_1}$ is $1$.
\item[{\rm (b)}] $B_{\mathtt A_1 \times \mathtt A_1}(0)$ appears in $ B_{\mathtt C_2}(a,b)\vert_{\mathtt A_1 \times \mathtt A_1}$ if and only if $a=b$.
\end{enumerate}
\end{corollary}

\subsection{Sato-Tate groups $C_n$} \label{sec-H-Cn}
Recall that we set $\hat{h}_1=E_{11}-E_{33}$ and $\hat{h}_2=E_{22}-E_{44}$, where $E_{ij}$  are the $4 \times 4$ elementary matrices.
From the embedding \eqref{embed-su2-su2}, we have the induced Lie algebra embedding $\mathfrak {sl}_2(\mathbb C) \times \mathfrak{sl}_2(\mathbb C) \hookrightarrow \mathfrak {sp}_4(\mathbb C)$ such that
\[ (h,0) \longmapsto \hat{h}_1 \quad \text{ and }\quad (0,h) \longmapsto \hat{h}_2 ,\]
where $h = {\scriptsize \begin{pmatrix} 1&0\\0&-1 \end{pmatrix}}$.

Consider the Sato--Tate groups $C_n$ for $n=1,2,3,4,6$. By definition the number $\mathfrak m_{\tilde \lambda}(C_n)$ is equal to the multiplicity of trivial representations in $\chi_{\tilde \lambda}^{\Sp(4)}|_{C_n}$. Throughout this section, write \[\tilde \lambda = (a,b).\] As observed in Section \ref{subsec-cn}, the number $\mathfrak m_{(a,b)}(C_n)$ is equal to the number of independent weight vectors with weight $\mu$ in the representation $V_{\mathtt C_2}(a,b)$ such that \[ \text{$\mu(\hat h_1+\hat h_2)=0$ \quad and \quad $\mu(\hat h_1-\hat h_2) \equiv 0$ (mod $2n$),\quad  $n=1,2,3,4,6$.} \]
If $a-b$ is odd then $\mu (\hat h_1+ \hat h_2)$ cannot be zero. This verifies \eqref{con-coeff}, and we assume that $a-b$ is even in what follows, and write
\[ z: = (a-b)/2 . \]

\begin{proposition} [$n=1$] \label{prop:cn1}
For a partition $(a,b)$, assume that $a-b$ is even. Then the sum of  weight multiplicities of the weights $\mu$ such that
$\mu(\hat h_1+\hat h_2)=0$
is given by
\begin{align*} \sum_{(p,q) \in \Pset(a,b)}( \min(p,q)+1)&  =(b+1)(z^2 + zb+ 2z +b/2+1) ,
\end{align*} which is equal to $ \widetilde \eta_1(z,b) = \eta_1(z,b)$.
\end{proposition}

\begin{proof}
Since we have the branching decomposition in Proposition \ref{prop: A1xA1}, we look at $B_{\mathtt A_1 \times \mathtt A_1}(p \mathtt{t}_1+q \mathtt{t}_2)$
for $(p,q) \in \Pset(a,b)$. If a tableau in the realization of $B_{\mathtt A_1 \times \mathtt A_1}(p \mathtt{t}_1+q \mathtt{t}_2)$ has weight $\mu$,
then $\mu(\hat h_1+\hat h_2)$ is equal to
$$ \text{(the number of $1$ and $2$)} - \text{(the number of $\Oone$ and $\Otwo$)}.$$
Thus the number of elements in $B_{\mathtt A_1 \times \mathtt A_1}(p \mathtt{t}_1+q \mathtt{t}_2)$ with weights $\mu$ such that $\mu(\hat h_1+\hat h_2)=0$ is equal to $\min(p,q)+1$
(recall \eqref{eq: tableaux realization of A1A1 crystal}).

The elements $(p,q)$ in $\Pset(a,b)$ and the corresponding numbers $\min(p,q)+1$ can be each arranged into an array of size
$(2z+1) \times (b+1)$ as follows, where we put $(p,q)$ in the left and $\min(p,q)+1$ in the right:
\begin{equation} \label{eq: arrangement}
{\scriptsize
\left. \begin{matrix}
(a,b) & (a-1,b-1) & \cdots & (a-b,0) \\
(a-1,b+1) & (a-2,b) & \cdots & (a-b-1,1) \\
\vdots & \vdots & \cdots & \vdots  \\
(a-z+1,b+z-1) & (a-z,b+z-2)& \cdots & (a-b-z+1,z-1) \\
(b+z, b+z) & (b+z-1,b+z-1) & \cdots & (z,z) \\
(b+z-1, a-z+1) & (b+z-2,a-z) & \cdots & (z-1,a-b-z+1) \\
\vdots & \vdots & \cdots & \vdots  \\
(b+1,a-1) & (b,a-2) & \cdots & (1,a-b-1) \\
(b,a) & (b-1,a-1) & \cdots & (0,a-b)
\end{matrix} \ \ \right|
\
\begin{matrix}
b+1 & b & \cdots & 1 \\
b+2 & b+1 & \cdots & 2 \\
\vdots & \vdots & \cdots & \vdots  \\
b+z & b+z-1 & \cdots & z \\
b+z+1 & b+z & \cdots & z+1 \\
b+z & b+z-1 & \cdots & z \\
\vdots & \vdots & \cdots & \vdots  \\
b+2 & b+1 & \cdots & 2 \\
b+1 & b & \cdots & 1
\end{matrix} }
\end{equation}
Taking sums by rows and adding up the results yields
$$
\sum_{(p,q) \in \Pset(a,b)}( \min(p,q)+1)=\dfrac{(b+1)}{2} \left( 2(b+2)+2(b+4)+\cdots+2(b+2z) +(b+2z+2) \right),
$$
which is equal to $\eta_1(z,b)$.
\end{proof}

\begin{remark} \label{rmk: p=q line}
In~\eqref{eq: arrangement}, all $(p,q)$'s with $p=q$ form the row containing $(b+z,b+z)$.
\end{remark}

\begin{proposition}[$n=2$]\label{prop:cn2}
For a partition $(a,b)$, assume that $a-b$ is even. Then the sum of weight multiplicities of the weights $\mu$ such that $\mu(\hat h_1+\hat h_2)=0$ and $\mu(\hat h_1-\hat h_2) \equiv_4 0$ is given by
\begin{equation}
\sum_{\substack{ (p,q) \in \Pset(a,b) \\ q\, \equiv_2 \, 0 }  }  ( \min(p,q) +1) \label{eq: mod 4 first},
\end{equation}
and the sum is equal to
\begin{equation*}
\widetilde \eta_2(z,b)=\tfrac 1 2 \eta_1(z,b) + \tfrac 1 2 \eta_2(z,b). 
\end{equation*}
\end{proposition}

\begin{proof} Write $\mu=r\epsilon_1+s\epsilon_2$ for a weight of  $B_{\mathtt A_1 \times \mathtt A_1}(p \mathtt{t}_1+q \mathtt{t}_2)$ satisfying $r+s=0$. Then we have
\[  r-s=2r \equiv_4 0  \quad \Longleftrightarrow \quad r \equiv_2 s \equiv_2 q \equiv_2 0 . \] With this observation, our first assertion follows from the first paragraph in the proof of Proposition \ref{prop:cn1}.

Now let us consider the second assertion; i.e., the sum is equal to
$$\tfrac 1 2 \eta_1(z,b) + \tfrac 1 2 \eta_2(z,b).$$
Since $p \equiv_2 q$, the sum \eqref{eq: mod 4 first} amounts to the sum of odd integers in the right array of \eqref{eq: arrangement}.
For any block of $4$ integers of the  form
$$ \begin{matrix} i+1 & i \\ i & i-1 \end{matrix}
$$ in~\eqref{eq: arrangement}, the sum of odd integers is $2i$. By decomposing the right array in~\eqref{eq: arrangement} into as many such blocks of $4$ integers as possible in each of the cases
$b \equiv_2 0, 1$ and $z \equiv_2 0,1$, we can check the sum is equal to $\frac 1 2 \eta_1 + \frac 1 2 \eta_2$.

For example, if $b \equiv_2 1$ and $z \equiv_2 0$, the first $z$ rows of the array are decomposed into $\frac {b+1}2 \times \frac z 2$ blocks of $4$ integers, and the sum of odd integers from those blocks is
\[\tfrac 1 4  (b+1)  z (b+z+1). \] The sum of odd integers in the middle row is
\[ \tfrac 1 4 (b+1)  (b+2z+1), \]
and the total sum of odd integers is
\[  2 \times \tfrac 1 4 (b+1)  z (b+z+1) + \tfrac 1 4 (b+1) (b+2z+1) = \tfrac 1 2 (b+1)  (z^2+zb+2z+(b+1)/2), \] which is the same as $\frac 1 2 \eta_1 + \frac 1 2 \eta_2$ in this case.
\end{proof}

\begin{proposition}[$n=3$]\label{prop:cn3}
Assume that $a-b$ is even. Then the sum of weight multiplicities of the weights $\mu$ such that  $\mu(\hat h_1+\hat h_2)=0$ and $\mu(\hat h_1-\hat h_2) \equiv_6 0$ is given by
\begin{align}
\sum_{(p,q) \in \Pset(a,b) } \left(  \lfloor \min(p,q)/3 \rfloor +1 - \delta( \min(p,q)\equiv_3 1 ) \right), \label{eq: mod 6 first}
\end{align}
and the sum is equal to
\begin{equation*}
\widetilde \eta_3(z,b) =\tfrac 1 3 \eta_1(z,b)+ \tfrac 2 3 \eta_3(z,b). 
\end{equation*}
\end{proposition}

\begin{proof} Consider $(p,q) \in \Pset(a,b)$, and without loss of generality assume that $p \ge q$. Write $\mu=r\epsilon_1+s\epsilon_2$ for a weight of  $B_{\mathtt A_1 \times \mathtt A_1}(p \mathtt{t}_1+q \mathtt{t}_2)$ satisfying $r+s=0$.
Explicitly, the set of such weights $(r,s)$ is
\begin{equation} \label{qq}  \{ (-q,q), (-q+2, q-2), \dots , (q-2, -q+2), (q, -q) \}.\end{equation}
Since $r+s=0$, the condition $r-s \equiv_6 0$ is equivalent to $s \equiv_3 0$.
In each case of $q \equiv_3 0,1,2$, we count the number of pairs $(-s,s)$ in \eqref{qq} such that $s \equiv_3 0$, and it is equal to
\[ \begin{cases} q/3 +1 & \text{ if } q \equiv_3 0, \\ (q-1)/3 & \text{ if } q \equiv_3 1, \\ (q+1)/3 & \text{ if } q \equiv_3 2. \end{cases} \]
This justifies the expression in \eqref{eq: mod 6 first}.

Now we have to add the numbers $\lfloor (e-1)/3 \rfloor +1 - \delta( (e-1)\equiv_3 1)$, where $e$ runs over the right array in~\eqref{eq: arrangement}. Observe that, for any consecutive 3 integers $ i+1, \ i, \ i-1 $, we have
$$ i =  \sum^{i+1}_{e=i-1} \lfloor (e-1)/3 \rfloor +1 - \delta( (e-1)\equiv_3 1).$$
By decomposing the right array in~\eqref{eq: arrangement} into blocks of $3$-consecutive integers as many as possible, we can check the sum is equal to $\frac 1 3 \eta_1 + \frac 2 3 \eta_3$.
\end{proof}

\begin{proposition}[$n=4$]\label{prop:cn4}
Assume that $a-b$ is even. Then the sum of weight multiplicities  of the weights $\mu$ such that  $\mu(\hat h_1+\hat h_2)=0$ and $\mu(\hat h_1-\hat h_2) \equiv_8 0$ is given by
\begin{equation}
\sum_{\substack{ (p,q) \in \Pset(a,b) \\ q \, \equiv_2 \, 0  }  }  (    2\times \lfloor \min(p,q)/4 \rfloor  +1) \label{eq: mod 8 first},
\end{equation}
and the sum is equal to
\begin{equation*}
\widetilde \eta_4(z,b)= \tfrac 1 4 \eta_1(z,b)+ \tfrac 1 4 \eta_2(z,b) + \tfrac 1 2 \eta_4(z,b). 
\end{equation*}
\end{proposition}

\begin{proof} As in the proof of Proposition \ref{prop:cn3}, consider $(p,q) \in \Pset(a,b)$ with $p \ge q$ and the set \eqref{qq}.
Since $r+s=0$, the condition $r-s \equiv_8 0$ is equivalent to $s \equiv_4 0$. If $s \equiv_4 0$, then we must have $q \equiv_2 0$. Thus,
in each case of $q \equiv_4 0,2$, we count the number of pairs $(-s,s)$ in \eqref{qq} such that $s \equiv_4 0$, and it is equal to
\[ 2\times \lfloor q/4 \rfloor  +1. \]
This justifies the expression in \eqref{eq: mod 8 first}.

We need to add the numbers $ ( 2 \times \lfloor (e-1)/4 \rfloor +1) \times \delta(  (e-1) \equiv_2 0)$,  where $e$ runs over the right array in~\eqref{eq: arrangement}. Observe that, for 8-integers in any block of the form
$$\begin{matrix}
i+2 & i+1 &  i &  i-1 \\
i+1 & i &  i-1 &  i-2 \end{matrix}$$  in~\eqref{eq: arrangement}, we have
$$ 2i =  \sum_{u=0}^1\sum^{i+1+u}_{e=i-2+u} ( 2 \times \lfloor (e-1)/4 \rfloor +1) \times \delta(  (e-1) \equiv_2 0).$$

By decomposing the right array in~\eqref{eq: arrangement} into as many such blocks of $8$ integers as possible in each of the cases $b \equiv_4 0, 1,2,3$ and $z \equiv_4 0,1,2,3$, we can see that the sum is equal to $\frac 1 4 \eta_1 + \frac 1 4 \eta_2 + \frac 1 2 \eta_4$.
\end{proof}

\begin{proposition}[$n=6$]\label{prop:cn6}
Assume that $a-b$ is even. Then the sum of weight multiplicities of the weights $\mu$ such that  $\mu(\hat h_1+\hat h_2)=0$ and $\mu(\hat h_1-\hat h_2) \equiv_{12} 0$ is given by
\begin{equation}
\sum_{\substack{ (p,q) \in \Pset(a,b) \\ q \, \equiv_2 \, 0 }  }  (    2\times \lfloor \min(p,q)/6 \rfloor  +1)  \label{eq: mod 12 first},
\end{equation}
and the sum is equal to
\begin{equation*}
\widetilde \eta_6(z,b)=\tfrac 1 6 \eta_1(z,b)+ \tfrac 1 6 \eta_2(z,b) + \tfrac 1 3 \eta_3(z,b) +\tfrac 1 3 \eta_6(z,b).  
\end{equation*}
\end{proposition}

\begin{proof} As in the proof of Proposition \ref{prop:cn3}, consider $(p,q) \in \Pset(a,b)$ with $p \ge q$ and the set \eqref{qq}.
Since $r+s=0$, the condition $r-s \equiv_{12} 0$ is equivalent to $s \equiv_6 0$. If $s \equiv_6 0$, then $q \equiv_2 0$. Thus,
in each case of $q \equiv_6 0,2, 4$, we count the number of pairs $(-s,s)$ in \eqref{qq} such that $s \equiv_6 0$, and it is equal to
\[ 2\times \lfloor q/6 \rfloor  +1. \]
This justifies the expression in \eqref{eq: mod 12 first}.

We have to add the numbers $ ( 2 \times \lfloor (e-1)/6 \rfloor +1) \times \delta(  (e-1) \equiv_2 0)$,  where $e$ runs over the right array in~\eqref{eq: arrangement}. Observe that, for any 12-integers of the form
$$ \begin{matrix}
i+3 & i+2 & i+1 &  i &  i-1 & i-2\\
i+2 & i+1 & i &  i-1 &  i-2 & i-3\end{matrix}$$  in~\eqref{eq: arrangement}, we have
$$ 2i=  \sum_{u=0}^1\sum^{i+2+u}_{e=i-3+u} ( 2 \times \lfloor (e-1)/6 \rfloor +1) \times \delta(  (e-1) \equiv_2 0).$$

By decomposing the right array in~\eqref{eq: arrangement} into as many such blocks of $12$ integers as possible in each of the cases, we see that the sum is equal to $\frac 1 6 \eta_1 + \frac 1 6 \eta_2 + \frac 1 3 \eta_3 + \frac 1 3 \eta_6$.
\end{proof}

In the next subsection, we need the following.
\begin{corollary} \label{cor:cn}
Assume that $a-b$ is even. Then we have
\begin{align*}
\widetilde \psi_1(z,b)&= (-1)^b \sum_{(p,p) \in \Pset(a,b)}( p+1) , \allowdisplaybreaks\\
\widetilde \psi_2(z,b)&= (-1)^b \sum_{\substack{ (p,p) \in \Pset(a,b) \\ p\, \equiv_2 \, 0 }  }  ( p +1),\allowdisplaybreaks\\
\widetilde \psi_3(z,b) &= (-1)^b \sum_{(p,p) \in \Pset(a,b) } \left(  \lfloor p/3 \rfloor +1 - \delta( p \equiv_3 1 ) \right), \allowdisplaybreaks\\
\widetilde \psi_4(z,b)&= (-1)^b \sum_{\substack{ (p,p) \in \Pset(a,b) \\ p \, \equiv_2 \, 0  }  }  (    2\times \lfloor p/4 \rfloor  +1), \allowdisplaybreaks\\
\widetilde \psi_6(z,b)&= (-1)^b \sum_{\substack{ (p,p) \in \Pset(a,b) \\ p \, \equiv_2 \, 0 }  }  (    2\times \lfloor p/6 \rfloor  +1),
\end{align*}
where $\widetilde \psi_i(z,b)$ are defined in Corollary \ref{cor-def-psi}.
\end{corollary}

\begin{proof}
After dividing cases according to the congruence classes of $b$ and $z$, the computation becomes straightforward in each case. For example, consider $\widetilde \psi_2(z,b)$ and
assume that $b$ and $z$ are both even. Then the sum in the right hand side is equal to
\[  (z+b/2+1) (b/2+1). \] On the other hand, in this case, $\psi_1(z,b)=(b+1)(z+b/2+1)$ and $\psi_2(z,b)=z+b/2+1$. Thus\[ \widetilde \psi_2(z,b) = \tfrac 1 2 \psi_1(z,b) + \tfrac 1 2 \psi_2(z,b) = (z+b/2+1) (b/2+1).\]
The other cases are similar.
\end{proof}

\subsection{Sato--Tate groups $J(C_n)$} \label{subsec-J-Cn}

Let $J={\scriptsize \begin{pmatrix} 0&0&0&1\\0&0&-1&0\\0&-1&0&0\\1&0&0&0 \end{pmatrix}}$.
For the branching rules for $J(C_n)$, we need more information than the crystal isomorphism in Proposition \ref{prop: A1xA1}.
Let $V$ and $W \subset \bigwedge^2 V$ be the fundamental representations of $\Sp(4, \mathbb C)$. Take a basis $\{ v_1, v_2,v_{\Oone},v_{\Otwo} \}$  of $V$ such that
\begin{equation}
\label{eqn-Jv} Jv_1=v_{\Otwo}, \quad Jv_2= - v_{\Oone}, \quad Jv_{\Oone}=-v_2 \quad \text{ and } \quad  Jv_{\Otwo}=v_1.
\end{equation}
Write
\begin{align*}
w_{12} &= v_1 \wedge v_2, \quad  w_{1\Oone}=v_1 \wedge v_{\Oone}, \quad w_{1\Otwo}=v_1 \wedge v_{\Otwo}, \\
w_{2\Oone}&=v_2 \wedge v_{\Oone}, \quad w_{2\Otwo}=v_2 \wedge v_{\Otwo}, \quad w_{\Oone \; \hspace{-.2em} \Otwo} = v_{\Oone}\wedge v_{\Otwo},
\end{align*} and take the basis  $\{w_{12}, w_{1\Otwo}, w_{2\Oone}, w_{\Oone \Otwo}, w_{2\Otwo}-w_{1\Oone} \}$ for $W$. Then we have
\begin{align*}
J(w_{12}) &=w_{\Oone \; \hspace{-.2em} \Otwo}, \quad J (w_{1\Otwo})= - w_{1 \Otwo}, \quad J(w_{2\Oone})= -w_{2\Oone}, \\
J(w_{\Oone \; \hspace{-.2em} \Otwo}) &= w_{12} \quad \text{ and } \quad J(w_{2\Otwo} -w_{1\Oone})= - w_{2\Otwo} + w_{1\Oone} .
\end{align*}

We realize the representation $V_{\mathtt C_2}(a,b)$ for a partition $(a,b)$ as the irreducible component of $\mathrm{Sym}^{a-b}\, V \otimes \mathrm{Sym}^b\, W$ generated by the highest weight vector $v_1^{a-b} \otimes w_{12}^b$. In particular, $V= V_{\mathtt C_2}(1,0)$ and $W= V_{\mathtt C_2}(1,1)$.
We identify $V_{\mathtt C_2}(a+1,b+1)$ with the image of the embedding \[ \iota_{a+1,b+1} :V_{\mathtt C_2}(a+1,b+1) \hookrightarrow V_{\mathtt C_2}(a,b) \otimes V_{\mathtt C_2}(1,1)\] given by \[v_1^{a-b} \otimes w_{12}^{b+1} \mapsto (v_1^{a-b} \otimes w_{12}^b )\otimes w_{12}. \]

Let $V_{\mathtt A}(p,q)$ be the representation of $\mathtt A_1 \times \mathtt A_1$ with highest weight $p\mathtt t_1 + q \mathtt t_2$. When $a-b$ is even, we inductively specify $\mathtt A_1 \times \mathtt A_1$-highest weight vectors $v_{(p,q;a,b)}$ in $V_{\mathtt C_2}(a,b)$ to describe
the decomposition \[V_{\mathtt C_2}(a,b) \big |_{\mathtt A_1 \times \mathtt A_1}\cong \soplus_{(p,q) \in \Pset(a,b)} V_{\mathtt A}(p,q)\] in what follows.

For $V_{\mathtt C_2}(a,0)$ with $a$ even, let $v_{(p,q;a,0)}\seteq v_1^pv_2^q$ for $(p,q) \in \Pset(a,0)$.
Then one can see that $v_{(p,q;a,0)}$ are $\mathtt A_1 \times \mathtt A_1$-highest weight vectors with highest weights $p \mathtt t_1 + q \mathtt t_2$. For the  induction, assume that
$v_{(p,q;a,b)}$ are $\mathtt A_1 \times \mathtt A_1$-highest weight vectors in $V_{\mathtt C_2}(a,b)$ with weights $p \mathtt t_1 + q \mathtt t_2$ for $(p,q) \in \Pset(a,b)$ with $a-b$ even.

Consider the following subset of $\Pset(a,b)$
\begin{equation} \label{phiA}
\phi_{\mathtt A}(a,b):= \{ (a,b), (a-1,b+1), (a-2,b+2), \dots , (b+1, a-1), (b,a) \},
\end{equation}
by taking $r=0$ in~\eqref{eq: Phi}. Now the lemma below completes the induction process.

\begin{lemma} \label{vpq} \hfill

{\rm (1)} For $(p,q) \in \phi_{\mathtt A}(a,b)$, the vectors \[v_{(p,q;a+1,b+1)} = v_1^{p-(b+1)}v_2^{q-(b+1)} \otimes w_{12}^{b+1}\] are $\mathtt A_1 \times \mathtt A_1$-highest weight vectors in $V_{\mathtt C_2}(a+1,b+1)$ with weights $p \mathtt t_1 + q\mathtt t_2$.

{\rm (2)}
For each $(p,q) \in \Pset(a,b) \setminus \phi_{\mathtt A}(a,b)$, there is an $\mathtt A_1 \times \mathtt A_1$-highest weight vector $v_{(p,q;a+1,b+1)}$ of the form
\[ v_{(p,q;a,b)}\otimes (w_{2\Otwo}-w_{1\Oone})+u_1\otimes w_{12}+u_2 \otimes w_{1\Otwo} +u_3 \otimes w_{2 \Oone}+u_4 \otimes w_{\Oone \Otwo} \]
 for some $u_1, u_2, u_3, u_4 \in V_{\mathtt C_2}(a,b)$, where $V_{\mathtt C_2}(a+1,b+1)$ is identified with the image of $\iota_{a+1,b+1}$ in $V_{\mathtt C_2}(a,b) \otimes V_{\mathtt C_2}(1,1)$.
\end{lemma}

\begin{proof} Note that $v_{(a,b;a,b)} = v_1^{a-b} \otimes w_{12}^{b}$ is the $\mathtt C_2$-highest weight vector of $V_{\mathtt C_2}(a,b)$.

(1) The vectors $v_{(p,q;a+1,b+1)}$ are obtained from the highest weight vector $v_1^{a-b} \otimes w_{12}^{b+1}$ by applying $f_1$ successively. Since there are no $v_{\Oone}, v_{\Otwo}$ factors in the elements, they are clearly $\mathtt A_1 \times \mathtt A_1$-highest weight vectors.

(2)
Starting from the $\mathtt C_2$-highest weight vector $v_{(a,b;a,b)} \otimes w_{12}$, we apply $f_1$'s and $f_2$'s to obtain $v_{(p,q;a,b)} \otimes w_{12} +u$ and then apply $f_1f_2$ to obtain $v_{(p,q;a,b)} \otimes (w_{2 \Otwo} - w_{1 \Oone}) +u'$,
where the terms of $u$ and $u'$ have $w_{12}, w_{1\Otwo}, w_{2\Oone}$ or $w_{\Oone \Otwo}$ as right-most factors.   We can determine constants $c_k$, $d_l$, $M$ and $N$ so that
\[ \left (1 + \sum_{l=1}^N d_l \hat f_2^l \hat e_2^l \right ) \left (1 + \sum_{k=1}^M c_k \hat f_1^k \hat e_1^k \right )\cdot (v_{(p,q;a,b)} \otimes (w_{2 \Otwo} - w_{1 \Oone}) +u')  \] is an $\mathtt A_1 \times \mathtt A_1$-highest vector. Since the action of $\hat e_i$ and $\hat f_i$, $i=1,2$, on $v_{(p,q;a,b)} \otimes (w_{2\Otwo}- w_{1\Oone})$ is trivial, this highest weight vector is of the desired form.
\end{proof}

\begin{proposition} [$n=1$] \label{prop-J-imp}
For a partition $(a,b)$ with $a-b\equiv_2 0$,
the number of linearly independent vectors in $V_{\mathtt C_2}(a,b)$,
\begin{enumerate}
\item[{\rm (i)}] which are fixed by the action of $J$, and
\item[{\rm (ii)}] whose weights  $\mu$ satisfy $\mu(\hat h_1+\hat h_2)=0$,
\end{enumerate}
is equal to
\[ \widetilde \theta_1(z,b)=\theta_1(z,b) = \begin{cases} \tfrac 1 2 z(b+1)(z+b+1)  & \text{ if $b$ is odd}, \\ \tfrac 1 2 (z+1)(b+1)(z+b+2) & \text{ if $b$ is even}. \end{cases}\]
\end{proposition}

\begin{proof}
Suppose that $v$ is a vector of $V_{\mathtt C_2}(a,b)$ which lies inside the $\mathtt A_1 \times \mathtt A_1$-representation generated by the vector $v_{(p,q;a,b)}$ inductively defined in Lemma \ref{vpq} for
$(p,q) \in \Pset(a,b)$. Then the number of $v_1$ and $v_{\Oone}$ factors in each of the terms of $v$ is $p$. After the action of $J$ on $v$ induced from \eqref{eqn-Jv}, the number of $v_1$ and $v_{\Oone}$ factor in each of the terms of $J(v)$ is $q$. In particular, if $p \neq q$ then $v \neq J(v)$. Since $J^2=I$, the vector $v+J(v)$ is fixed by $J$. Furthermore, if $v$ is of weight $\mu$ such that $\mu(\hat h_1 + \hat h_2) =0$, then the weight of $J(v)$ has the same property.
Thus, when $p \neq q$, a $J$-fixed vector of weight $\mu$ such that $\mu(\hat h_1+ \hat h_2)=0$ must be of the form $v+J(v)$.

Recall that the number of independent vectors in the isomorphic copy of $V_{\mathtt A}(p,q)$ in $V_{\mathtt C_2}(a,b)$ with weights $\mu$ such that $\mu(\hat h_1 + \hat h_2)=0$ is equal to $\min(p,q)+1$, and the numbers are displayed in array \eqref{eq: arrangement}. The observation made in the previous paragraph shows that the total number of
$J$-fixed independent vectors in the isomorphic copy of $V_{\mathtt A}(p,q)$ for $p \neq q$ with such weights $\mu$ is given by the sum of the numbers in the first $z$ rows of array \eqref{eq: arrangement}:
\begin{equation*}
{\scriptsize
\left.
\begin{matrix}
(a,b) & (a-1,b-1) & \cdots & (a-b,0) \\
(a-1,b+1) & (a-2,b) & \cdots & (a-b-1,1) \\
\vdots & \vdots & \cdots & \vdots  \\
(a-z+1,b+z-1) & (a-z,b+z-2)& \cdots & (a-b-z+1,z-1)
\end{matrix} \ \ \right|
\quad
\begin{matrix}
b+1 & b & \cdots & 1 \\
b+2 & b+1 & \cdots & 2 \\
\vdots & \vdots & \cdots & \vdots  \\
b+z & b+z-1 & \cdots & z
\end{matrix} }
\end{equation*}
(see Remark~\ref{rmk: p=q line}). Explicitly, the total number for the case $p \neq q$ is
\begin{equation} \label{b+1}  \frac{b+1} 2 \left ( (b+2)+(b+4)+ \cdots + (b+2z) \right ) = \frac 1 2 z(b+1)(z+b+1).\end{equation}

Now let us consider the case $p=q$. Using induction on $a+b$, we will show
\begin{equation*} 
J v = (-1)^b v
\end{equation*}
for any vector $v$ with weight $\mu$ such that $\mu(\hat h_1 + \hat h_2)=0$ in the isomorphic copy of $V_{\mathtt A}(p,p)$ in $V_{\mathtt C_2}(a,b)$.

First assume that $(p,p) \in \phi_{\mathtt A}(a,b)$, where the set $\phi_{\mathtt A}(a,b)$ is defined in \eqref{phiA}. Then by Lemma \ref{vpq} (1) we have a highest weight vector of the form $v_1^{p-b}v_2^{p-b} \otimes w_{12}^b$. A weight vector $v$ in the $\mathtt A_1 \times \mathtt A_1$-representation generated by this highest weight vector has the form
\[ v=\hat f_2^\ell \hat f_1^k (v_1^{p-b}v_2^{p-b}\otimes w_{12}^b), \quad \ell, k \in \mathbb Z_{\ge 0}.\]
If we further assume $v$ has weight $\mu$ such that $\mu(\hat h_1+\hat h_2)=0$, then $\ell + k =p$.

Using the relations \[  J \hat f_2= - \hat e_1 J \quad \text{ and } \quad J \hat f_1 = - \hat e_2 J \] and the symmetry of $\mathtt A_1 \times \mathtt A_1$-representation,
we obtain
\begin{align*}
Jv&=J\hat f_2^\ell \hat f_1^k (v_1^{p-b}v_2^{p-b}\otimes w_{12}^b)=(-1)^{\ell + k} \hat e_1^\ell \hat e_2^k\ J ( v_1^{p-b} v_2^{p-b} \otimes w_{12}^b) \allowdisplaybreaks\\
&= (-1)^{p+p-b} \hat e_1^\ell \hat e_2^k\  ( v_{\Otwo}^{p-b} v_{\Oone}^{p-b} \otimes w_{\Oone \Otwo}^b)= (-1)^b\hat f_2^\ell \hat f_1^k (v_1^{p-b}v_2^{p-b}\otimes w_{12}^b)\allowdisplaybreaks\\
&=(-1)^b v.
\end{align*}

Next assume that $(p,p) \not \in \phi_{\mathtt A}(a,b)$, $a \ge b \ge 1$. By Lemma \ref{vpq} (2) we have a highest weight vector of the form \[v_{(p,p;a-1,b-1)} \otimes (w_{2\Otwo} -w_{1\Oone}) + u,\] where $u=u_1\otimes w_{12}+u_2 \otimes w_{1\Otwo} +u_3 \otimes w_{2 \Oone}+u_4 \otimes w_{\Oone \Otwo}$ for some $u_1, u_2, u_3, u_4 \in V_{\mathtt C_2}(a-1,b-1)$.
Since $\hat f_1$ and $\hat f_2$ act trivially on $(w_{2\Otwo}-w_{1\Oone})$, a weight vector $v$ in the $\mathtt A_1 \times \mathtt A_1$-representation generated by this highest weight vector is of the form
\[ v=\left ( \hat f_2^\ell \hat f_1^k  v_{(p,p;a-1,b-1)} \right ) \otimes (w_{2\Otwo} -w_{1\Oone}) + \hat f_2^\ell \hat f_1^k u .\]
We further assume that $v$ has weight $\mu$ such that $\mu(\hat h_1 + \hat h_2)=0$. By induction hypothesis, we have \[J\hat f_2^\ell \hat f_1^k  v_{(p,p;a-1,b-1)} = (-1)^{b-1}\hat f_2^\ell \hat f_1^k  v_{(p,p;a-1,b-1)}.\] Since $J (w_{2 \Otwo} -w_{1\Oone}) = - (w_{2 \Otwo} - w_{1\Oone})$, we obtain
\[ Jv= (-1)^b\left ( \hat f_2^\ell \hat f_1^k  v_{(p,p;a-1,b-1)} \right ) \otimes (w_{2\Otwo} -w_{1\Oone}) + J \left (\hat f_2^\ell \hat f_1^k u \right ).\]

Recall that $V_{\mathtt A}(p,q)$ is minuscule.
The action of $J$ preserves the weight space containing $v$ since the number of $v_2$ factors is equal to that of $v_{\Oone}$ factors in the terms of $v$.
Thus we must have $J v = (-1)^b v$ as claimed in this case too.

We have just shown that all the weight vectors in the isomorphic copy of $V_{\mathtt A}(p,p)$ with weight $\mu$ such that $\mu(\hat h_1 + \hat h_2) =0$ are fixed by $J$ if $b$ is even, and that none of them are fixed by $J$ if $b$ is odd.
Combining this with the result for $p \neq q$ in \eqref{b+1}, the total number of $J$-fixed independent vectors of $V_{\mathtt C_2}(a,b)$ with weight $\mu$ such that $\mu(\hat h_1 + \hat h_2) =0$ is equal to \eqref{b+1} if $b$ is odd, and to the sum of \eqref{b+1} and the $z+1^{\mathrm{st}}$ row of the array in \eqref{eq: arrangement} if $b$ is even, which is given by
\[  \frac 1 2 z(b+1)(z+b+1)+ \{(b+z+1)+ (b+z)+ \cdots +(z+1)\} = \frac 1 2 (z+1)(b+1)(z+b+2).\]
In either case, we obtain the function $\theta_1 (z,b)$.
\end{proof}

\begin{proposition} [$n=2,3,4,6$] \label{prop-J-imp-2}
Assume that $a-b$ is even. For $n=2,3,4,6$,
the number of independent vectors in $V_{\mathtt C_2}(a,b)$, which are fixed by the action of $J$, with weights $\mu$ such that $\mu(\hat h_1+\hat h_2)=0$ and $\mu(\hat h_1-\hat h_2) \equiv_{2n} 0$, is equal to $\widetilde \theta_n(z,b)$
defined in \eqref{def-wi-theta}.
\end{proposition}

\begin{proof}

Suppose that $v$ is a vector of $V_{\mathtt C_2}(a,b)$ which lies inside the $\mathtt A_1 \times \mathtt A_1$-representation generated by the vector $v_{(p,q;a,b)}$ inductively defined in Lemma \ref{vpq} for
$(p,q) \in \Pset(a,b)$, and assume that $v$ is of weight $\mu$ such that $\mu(\hat h_1 + \hat h_2) =0$. The total numbers of such vectors are equal to $\widetilde \eta_n(z,b)$ by Propositions \ref{prop:cn2}--\ref{prop:cn6}, and  the numbers of such vectors only for $p=q$ are equal to $\widetilde \psi_n(z,b)$ by Corollary \ref{cor:cn}.

As observed in the proof of Proposition \ref{prop-J-imp}, if $p \neq q $ then the $J$-fixed vectors are precisely given by $J(v)+v$; if $p=q$ and $b$ is even then all the vectors $v$ of weight $\mu$
such that $\mu(\hat h_1 + \hat h_2) =0$ are fixed by $J$; if $p=q$ and $b$ is odd then none of such vectors are fixed by $J$. The conditions $\mu(\hat h_1 - \hat h_2) \equiv_{2n} 0$ exactly bring out the restrictions on the sums considered in Propositions \ref{prop:cn2}--\ref{prop:cn6}.

Thus, when $n=2$, the number of $J$-fixed vectors of weight $\mu$ satisfying the conditions is given by
\[   \tfrac 1 2 \widetilde \eta_2(z,b) + (-1)^b \tfrac 1 2  \sum_{\substack{ (p,p) \in \Pset(a,b) \\ p\, \equiv_2 \, 0 }  }  ( p +1)= \tfrac 1 2 \widetilde \eta_2(z,b) + \tfrac 1 2 \widetilde \psi_2(z,b)=\widetilde \theta_2(z,b) , \]
where we use Corollary \ref{cor:cn} and the definition of $\theta_2(z,b)$ in \eqref{def-th} along with the definitions of $\widetilde \eta_2$, $\widetilde \psi_2$ and $\widetilde \theta_2$.
Similarly, when $n=3,4,5$, the numbers of $J$-fixed vectors of weight $\mu$ satisfying the conditions are equal to
\[  \tfrac 1 2 \widetilde \eta_n(z,b) + \tfrac 1 2 \widetilde \psi_n(z,b) = \widetilde \theta_n(z,b). \qedhere \]

\end{proof}

\subsection{From $\mathtt C_2$ to $\mathtt A_1$ via removing the second vertex} \label{subsec-E}

In this section, we shall prove the branching decomposition of $V_{\mathtt C_2}(a,b)$ to $\mathtt A_1$ via Levi rule which removes the second vertex in the Dynkin diagram of type $\mathtt C_2$
$$
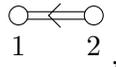
\begin{tikzpicture}[scale=0.5]
\draw (0 cm,0) -- (0 cm,0);
\draw (0 cm, 0.1 cm) -- +(2 cm,0);
\draw (0 cm, -0.1 cm) -- +(2 cm,0);
\draw[shift={(0.8, 0)}, rotate=180] (135 : 0.45cm) -- (0,0) -- (-135 : 0.45cm);
\draw[fill=white] (0 cm, 0 cm) circle (.25cm) node[below=4pt]{$1$};
\draw[fill=white] (2 cm, 0 cm) circle (.25cm) node[below=4pt]{$2$};
\end{tikzpicture},
$$
where we assume $a+b \equiv_2 0$.
We are mainly interested in the composition multiplicity of the trivial representation $V_{\mathtt A_1}(0)$ inside $V_{\mathtt C_2}(a,b)|_{\mathtt A_1}$.
We state the result at the crystal level as in Proposition~\ref{prop: A1xA1}.

\begin{proposition} \label{prop:e1} For a partition $(a,b)$,  set $\epsilon \seteq \delta(a+b\equiv_2 1)$ and $l = \lceil(a-b-1)/2 \rceil$. Then we have
\begin{align*}
B_{\mathtt C_2}(a,b)\vert_{\mathtt A_1} \cong & \left ( \soplus_{i=0}^{l-1} (2i+1+\epsilon)(b+1)B_{\mathtt A_1}(2i+\epsilon) \right ) \\
& \hspace{5ex} \oplus \left ( \soplus_{j=l}^{l+b}  (2l+1+\epsilon)(l+b+1 -j)B_{\mathtt A_1}(2j+\epsilon) \right ).
\end{align*}
\end{proposition}

\begin{proof}
As Proposition~\ref{prop: A1xA1}, the assertion can be proved  by using double induction on $a+b$ and the Clebsch--Gordan formula. 
\end{proof}

\begin{corollary}\label{cor:e1}
The multiplicity of the trivial representation $V_{\mathtt A_1}(0)$ in $V_{\mathtt C_2}(a,b)|_{\mathtt A_1}$ is equal to $$\delta(a+b \equiv_2 0) \times (b+1).$$
\end{corollary}

Now we shall interpret Corollary \ref{cor:e1}  via Kashiwara--Nakashima crystal. The crystal of vector representation
is given as follows:
$$\young(1) \overset{1}{\to} \young(2)  \overset{2}{\to} \young(\Otwo)  \overset{1}{\to} \young(\Oone)$$
with a linear order
$$ 1 \prec 2 \prec \Otwo \prec \Oone.$$

Now let us recall $\mathtt C_2$-type Kashiwara--Nakashima tableaux.

\begin{definition}
Let $Y$ be a Young diagram with at most $2$-rows.
\begin{enumerate}
\item A {\em $\mathtt C_2$-tableaux of shape $Y$} is a tableau obtained from $Y$ by filling the boxes with entries
$\{  1,2, \Otwo, \Oone \}$.
\item A $\mathtt C_2$-tableaux is said to be {\em semistandard}  if the entries in each row are weakly increasing and
the entries in each column are strictly increasing.
\end{enumerate}
\end{definition}

We define $B(Y)$ to be the set of all semistandard $\mathtt C_2$-tableaux $T$ satisfying the following conditions:
\begin{enumerate}
\item[(a)] $T$ does not have a column  $\young(1,\Oone)$;
\item[(b)] $T$ does not have a pair of adjacent columns  $\young(22,\sbullet\Otwo)$ or  $\young(2\sbullet,\Otwo\Otwo)$.
\end{enumerate}

\begin{definition} Assume that $(a,b)$ is a partition with $a+b \equiv_2 0$. For each $0 \le k \le b$, define $T_k$ be the semistandard $\mathtt C_2$-tableaux of shape $(a,b)$ such that
\begin{enumerate}
\item the first row of the tableau $T_k$  is filled with the sequence of entries $(1^k,2^z,\Otwo^y)$, where
$$z=\dfrac{a-b}{2} \quad \text{ and } \quad y=\dfrac{a+b}{2}-k;$$

\item the second row of the tableau $T_k$  is filled with the sequence of entries $(2^k,\Oone^{b-k})$.
\end{enumerate}
\end{definition}

\begin{lemma} For each  $0 \le k \le b$, the tableaux $T_k$ is contained in $B(Y)$.
\end{lemma}

\begin{proof} The condition (a) is obviously satisfied. Let us check whether $T_k$ satisfies the condition (b). Note that, for $1\le s \le b$,
if an entry placed in the position $(1,s)$ is $2$, then the entry placed in the position $(2,s)$ is $\Oone$, if it exists, by definition. Thus  the condition (b) is satisfied.
\end{proof}

\begin{proposition} \label{prop:exhaust}
For each $0 \le k \le b$, we have
$$\tilde{e}_1T_k=\tilde{f}_1T_k = 0,$$
where $\tilde e_1$ and $\tilde f_1$ are the Kashiwara operators.
\end{proposition}

\begin{proof}
We use the far eastern reading \cite{HK}. Then, since the $1$-signature of $\young(1,2)$ is $\sbullet$,
\begin{align} \label{eq: first k}
&\text{we do not need to read the first $k$-columns for $1$-signature.}
\end{align}
Set
\begin{itemize}
\item[] $s_{1} \seteq \text{(the number of $1$'s in $T_k$)} - \text{ (the number of $\Oone$'s in $T_k$)}$,
\item[] $s_{2}\seteq \text{(the number of $2$'s in $T_k$)} -   \text{ (the number of $\Otwo$'s in $T_k$)}$.
\end{itemize}
 By definition $s_1=s_2=2k-b$.  By an induction and \eqref{eq: first k}, it suffices to consider when $k=0$. Then, by the far eastern reading, $T_0$ can be read as follows:
$$
\young(\Otwo)^{\otimes b-a}  \otimes (\young(\Otwo) \otimes \young(\Oone))^{\otimes \frac{3b-a}{2}}
  \otimes (\young(2) \otimes \young(\Oone))^{\otimes \frac{a-b}{2}}
$$
Since
{\rm (i)} the $1$-signature of $\young(\Otwo)$ is $+$,
{\rm (ii)} the $1$-signature of $\young(\Otwo) \otimes \young(\Oone)$ is $\sbullet$,
{\rm (iii)} the $1$-signature of $\young(2) \otimes \young(\Oone)$ is $--$,
our assertion follows.
\end{proof}

We keep the assumption that $a-b$ is even. By Corollary \ref{cor:e1} , the $\mathtt C_2$-tableaux $\{T_k \}_{0 \le k \le b}$ exhaust all the trivial representations $V_{\mathtt A_1}(0)$ inside
$V_{\mathtt C_2}(a,b)\vert_{\mathtt A_1}$.
Note that the weight of $T_k$ is
\begin{align}\label{eq:wt}
(2k-b)(\epsilon_1+\epsilon_2), \qquad 0 \le k \le b.
\end{align}

\smallskip

For the Sato--Tate
groups $E_n$, $n=1,2,3,4,6$,
the number $\mathfrak m_{(a,b)}(E_n)$ is equal to the number of independent weight vectors $v_\mu$  in $V_{\mathtt C_2}(a,b)$ with weight $\mu$ such that  $\mu (\hat h_1+\hat h_2) \equiv 0$ (mod $2n$) and
$e_1 v_\mu = f_1 v_\mu=0$. By Proposition \ref{prop:exhaust}, the number $\mathfrak m_{(a,b)}(E_n)$ for each $n$ is equal to the number of tableaux $T_k$ such that
\begin{equation*}
2k-b \equiv_n 0.
\end{equation*}

We count the number of such tableaux for each $n=1,2,3,4,6$, and summarize the results in the following proposition.
\begin{proposition}\label{prop:en-1} Assume that $a-b$ is even. Then the number of independent weight vectors $v_\mu$  in $V_{\mathtt C_2}(a,b)$ with weight $\mu$ such that  $\mu (\hat h_1+\hat h_2) \equiv_{2n} 0$  and $e_1 v_\mu = f_1 v_\mu=0$ is equal to
\[\begin{cases}
b+1 &\text{if } n=1, \\
(b+1) \delta(b \equiv_2 0)  &\text{if } n=2 ,\\
 \lfloor b /3 \rfloor +1
-\delta(b \equiv_3 1) & \text{if } n=3, \\
(2\lfloor b/4 \rfloor +1)  \delta(b \equiv_2 0) &\text{if } n=4,\\
(2\lfloor b/6 \rfloor +1 )  \delta(b \equiv_2 0)&\text{if }n=6.\\
\end{cases}\]
\end{proposition}

\subsection{Sato--Tate groups $J(E_n)$} \label{subsec-JE}

We keep the notations of Section \ref{subsec-J-Cn}.
We start with a lemma on representations of $\mathfrak{sl}_2(\mathbb C)$.
\begin{lemma}
Assume that $v$ is a vector of weight $0$ in a finite dimensional representation of $\mathfrak{sl}_2(\mathbb C)$ with the standard basis $\{e,h,f\}$. Suppose that $e^{N+1} v =0$ for $N \in \mathbb Z_{\ge 0}$. Then the vector \[ \left ( \sum_{k=0}^N \frac{(-1)^k}{k+1} f^{(k)} e^{(k)} \right ) v \]
either vanishes or generates the trivial representation of $\mathfrak{sl}_2(\mathbb C)$, where $f^{(k)}= f^k/k!$ and $e^{(k)}= e^k/k!$.
\end{lemma}

\begin{proof}
Since \[ f^{(k+1)} e^{(k)} = \sum_{t=0}^k \frac 1 {t!} e^{(k-t)}f^{(k+1-t)}\prod_{s=1}^t (s+1-t-h) \]
(see \cite[Exercise 1.3 (c)]{HK}),
we have \[ f^{(k+1)} e^{(k)} =e^{(k)} f^{(k+1)}+e^{(k-1)} f^{(k)} \] on the weight $0$ space.
Similarly, since \[ef^{(k)} = f^{(k)} e+ f^{(k-1)} (h-k+1), \] we have \[ \frac 1 {k+1} e f^{(k)} e^{(k)} = f^{(k)} e^{(k+1)} + f^{(k-1)} e^{(k)} \] on the weight $0$ space.

Since $e^{N+1}v =0$ and the weight of $v$ is zero, we also have $f^{N+1}v=0$. Then
\begin{align*}
 f  \left ( \sum_{k=0}^N \frac{(-1)^k}{k+1} f^{(k)} e^{(k)} \right ) v &= \left ( \sum_{k=0}^N (-1)^k f^{(k+1)} e^{(k)} \right ) v \allowdisplaybreaks\\
&=fv+\left ( \sum_{k=1}^N (-1)^k (e^{(k)} f^{(k+1)}+e^{(k-1)} f^{(k)}) \right ) v\allowdisplaybreaks\\
&= fv - fv + (-1)^{N}e^{(N)}f^{(N+1)} v =0 .\end{align*}
Similarly, \begin{align*}
e  \left ( \sum_{k=0}^N \frac{(-1)^k}{k+1} f^{(k)} e^{(k)} \right ) v &= \left ( \sum_{k=0}^N (-1)^k \frac 1 {k+1} e f^{(k)} e^{(k)} \right ) v\allowdisplaybreaks\\
&=ev+\left ( \sum_{k=1}^N (-1)^k (f^{(k)} e^{(k+1)}+f^{(k-1)} e^{(k)}) \right ) v \allowdisplaybreaks\\
&= ev - ev + (-1)^{N}f^{(N)}e^{(N+1)} v =0 .\qedhere \end{align*}
\end{proof}

Now we consider $V_{\mathtt C_2}(a,b)$ for a partition $(a,b)$ with $a-b$ even.
Assume that $e_1v=f_1v=0$ for a vector $v$ in $V_{\mathtt C_2}(a,b)$.
Since $J$ commutes with $e_1$ and $f_1$, we have
\begin{equation}  \label{J-fix}
e_1Jv=f_1Jv=0.
\end{equation}

By Proposition \ref{prop:exhaust} and \eqref{eq:wt}, the vectors on which $e_1$ and $f_1$ act trivially have weights $t(\epsilon_1+\epsilon_2)$ where $t=2k-b$ for $0 \le k \le b$.
The action of $J$ sends weight $t(\epsilon_1+\epsilon_2)$ to $-t(\epsilon_1 + \epsilon_2)$.
If $v$ is a vector such that $e_1v=f_1v=0$, then $Jv+v$ is fixed by $J$ and $e_1(Jv+v)=f_1(Jv+v)=0$ by \eqref{J-fix}. Conversely, any $J$-fixed vector of non-zero weight on which $e_1$ and $f_1$ act trivially must be of the form $Jv+v$.

Therefore, if $b$ is odd then $t$ cannot be zero and the number of $J$-fixed vectors $v$ such that $e_1v=f_1v=0$ is equal to $(b+1)/2$ from Proposition \ref{prop:en-1}.
If $b$ is even then the number of $J$-fixed vectors of non-zero weight such that $e_1v=f_1v=0$ is equal to $b/2$. Now the remaining case is when $b$ is even and the weight is $0$. There is only one vector of weight $0$ on which $e_1$ and $f_1$ act trivially. We will determine whether $J$ fixes this vector or not.

\begin{lemma} \label{non-term}
Assume that $b$ is even. Suppose that $v \in V_{\mathtt C_2}(a,b) \subset \mathrm{Sym}^{a-b}\, V \otimes \mathrm{Sym}^b\, W$ is a vector of weight $0$ such that $e_1v=f_1v=0$. Then $v$ has a scalar multiple of
\[  (v_1^z v_{\Oone}^z + v_2^z v_{\Otwo}^z) \otimes w_{12}^{b/2}w_{\Oone \Otwo}^{b/2} \] as a non-zero term, where $z= (a-b)/2$.
\end{lemma}

\begin{proof}
Let $$\tilde v\seteq (v_1^z v_{\Oone}^z + v_2^z v_{\Otwo}^z) \otimes w_{12}^{b/2}w_{\Oone \Otwo}^{b/2}.$$
If we apply $e_2^k$ on the vector $\tilde v$ for $1 \le k \le z+ b/2$, the result for each $k$ can be written in terms of
\begin{align*}
&(v_1^z v_{\Oone}^z + v_2^z v_{\Otwo}^z) \otimes w_{12}^{b/2}w_{2\Oone}^l w_{\Oone \Otwo}^{b/2-l}, && v_2^{z+m} v_{\Otwo}^{z-m} \otimes w_{12}^{b/2}w_{2\Oone}^lw_{\Oone \Otwo}^{b/2-l} , \\
& v_2^{2z} \otimes w_{12}^{b/2}w_{2\Oone}^l w_{\Oone \Otwo}^{b/2-l},&& v_2^{z+l} v_{\Otwo}^{z-l} \otimes w_{12}^{b/2}w_{2\Oone}^{b/2}
\end{align*} for some $l$ and $m$, and the signs of the coefficients of these terms are determined by the parity of the exponent of $w_{2\Oone}$.
In particular, the result of the action of $e_2$ on $v_2^{z+m} v_{\Otwo}^{z-m} \otimes w_{12}^{b/2}w_{2\Oone}^lw_{\Oone \Otwo}^{b/2-l}$ is not canceled out with other terms.
Thus if we apply $e_2^{z+b/2}$ on $\tilde v$, all the terms are combined and the result is a scalar multiple of
\[ v_2^{2z} \otimes w_{12}^{b/2}w_{2\Oone}^{b/2}.\] (One can check that the coefficient is $(-1)^{b/2} (z+b/2)!$.)

Similarly, the action of $e_1^{2z+b}$ on $v_2^{2z} \otimes w_{12}^{b/2}w_{2\Oone}^{b/2}$ results in a scalar multiple of
\[ v_1^{2z} \otimes w_{12}^{b/2}w_{1\Otwo}^{b/2} \]
and the following action of $e_2^{b/2}$ brings it to a scalar multiple of the highest weight vector $v_1^{2z} \otimes w_{12}^b$. Thus we have shown that
the highest vector can be obtained from the vector $\tilde v$ by applying $e_1$'s and $e_2$'s. This implies that if we apply $f_1$'s and $f_2$'s on the highest weight vector appropriately then we obtain
$$\text{a vector $\hat v$ of weight $0$ which has a non-zero scalar multiple of $\tilde v$ as a term.}$$

Set $$v \seteq \left ( \sum_{k=0}^N \frac{(-1)^k}{k+1} f_1^{(k)} e_1^{(k)} \right )\hat v, \qquad \text{where $e_1^{N+1} \hat v =0$}.$$  Since $e_1 \tilde v= f_1 \tilde v=0$, the vector $v$ is non-zero and still has a scalar multiple of $\tilde v$ as a term. By Lemma \ref{non-term}, we have $e_1 v = f_1 v =0$. Since every vector of weight $0$ on which $e_1$ and $f_1$ act trivially is a scalar multiple of $v$, we are done.
\end{proof}

Continue to assume that $b$ is even, and consider a vector $v$ of
weight $0$ such that $e_1v=f_1v=0$. Since the space of weight $0$
vectors on which $e_1$ and $f_1$ act trivially is one-dimensional,
we have $Jv= \pm v$ from the fact that $J^2=1$. Furthermore, such a
vector has a non-zero term $(v_1^z v_{\Oone}^z + v_2^z v_{\Otwo}^z)
\otimes w_{12}^{b/2}w_{\Oone \Otwo}^{b/2}$ by Lemma \ref{non-term}.
Since $J w_{12} =w_{12}$ and $J w_{\Oone \Otwo} = w_{\Oone \Otwo}$,
we see that
\[  J \left ( (v_1^z v_{\Oone}^z + v_2^z v_{\Otwo}^z) \otimes w_{12}^{b/2}w_{\Oone \Otwo}^{b/2}\right )
= (-1)^z \left ((v_1^z v_{\Oone}^z + v_2^z v_{\Otwo}^z) \otimes w_{12}^{b/2}w_{\Oone \Otwo}^{b/2}\right ).\]
It follows that
\begin{equation} \label{eq:wt-J} Jv=(-1)^z v .\end{equation}

Combining this with the observations made in the paragraph right before Lemma \ref{non-term}, we have proved the following proposition.

\begin{proposition} \label{e1}
Assume that $a-b$ is even, and set $z:= (a-b)/2$. Then
 the number of $J$-fixed vectors $v$ with weight $\mu$ such that $e_1v=f_1v=0$ is equal to
\[\tfrac 1 2 (b+1) +\tfrac 1 2 (-1)^z \delta(b \equiv_2 0) .\]
\end{proposition}

\subsection{Sato--Tate group $F$} \label{subsec-F}

Recall that we have the embedding $\U(1) \times \U(1)$ into $\USp(4)$ given by
\[  (u_1, u_2) \mapsto \mathrm{diag}(u_1, u_2, u_1^{-1}, u_2^{-1}) ,\]
and that the group $F$ is the image of this embedding.
The number $\mathfrak m_{(a,b)}(F)$ is equal to the number of independent weight vectors $v_\mu$ in $V_{\mathtt C_2}(a,b)$ such that $\mu =0$.

Recall the array in \eqref{eq: arrangement}, which lists the elements of $\Pset(a,b)$



\begin{proposition} \label{f-form}
Assume that $a-b$ is even. Then
the multiplicity of  weight $0$ in  $V_{\mathtt C_2}(a,b)$  is equal to the number of elements $(p,q) \in \Pset(a,b)$ such that $p \equiv q \equiv 0 \mod 2$. Explicitly, the number is equal to
\[  \xi_1(z,b) := z(b+1)+ \lfloor b/2\rfloor +1. \]
\end{proposition}

\begin{proof} A weight zero space appears in $B_{\mathtt A_1 \times \mathtt A_1}(p \mathtt{t}_1+q \mathtt{t}_2)$ only if $p \equiv q \equiv 0 \mod 2$.
Then our assertion follows from the fact that $B_{\mathtt A_1 \times \mathtt A_1}(p \mathtt{t}_1+q \mathtt{t}_2)$ is minuscule,  and that
the Cartan subalgebra for $\mathtt A_1 \times \mathtt A_1$ is the Cartan subalgebra for $\mathtt C_2$.
One can count the number of such pairs $(p,q)$ in \eqref{eq: arrangement} to see that it is equal to $\xi_1(z,b)$.
\end{proof}

\subsection{Sato--Tate group $F_{\mathtt{a}}$} \label{subsec-Fa}
Since we have \[\mathtt a v_1= -v_{\Oone}, \quad \mathtt a v_2 =v_2, \quad \mathtt a v_{\Oone} = v_1, \quad \mathtt a v_{\Otwo}=v_{\Otwo}, \]
it follows from Proposition \ref{f-form} that the number of $\mathtt a$-fixed vectors of weight $0$ in $V_{\mathtt C_2}(a,b)$ is equal to
 the number of elements $(p,q) \in \Pset(a,b)$ such that $p \equiv_4 0$ and $q \equiv_2 0$.

\begin{proposition} \label{fa-prop}
The number of $\mathtt a$-fixed vectors of weight $0$ in $V_{\mathtt C_2}(a,b)$ is equal to
\[  \frac 1 2 \xi_1 (z,b) + \frac 1 2 \xi_2 (z,b), \] where $\xi_2(z,b)$ is defined
 on  the congruence classes of $z$ and $b$ by
\begin{center} {\scriptsize
\begin{tabular}{|c||c|c|c|c|}
\hline
$z\backslash b$ & $0$&$1$&$2$&$3$\\ \hline \hline
$0$&$1$ & $1$&$0$& $0$ \\ \hline
$1$&$0$&$-1$ & $-1$ & $0$ \\ \hline
\end{tabular}\,.} \end{center}
\end{proposition}

\begin{proof}
There are 8 cases according to the congruence classes of $z$ and $b$. Since all the cases are similar, we only prove the case when $b \equiv_4 0$ and $z \equiv_2 0$. In this case, $a \equiv_4 0$,
and the pairs $(p,q) \in \Pset(a,b)$ satisfying $p \equiv_4 0$ and $ q \equiv_2 0$ can be arranged as follows:
\begin{equation*}
{\scriptsize
\begin{matrix}
(a,b) &  &            && (a-4,b-4)&\cdots & (a-b,0) \\
(a-4,b+4) & (a-4,b+2) & (a-4, b) &(a-4,b-2)&(a-8,b)&\cdots & (a-b-4,4) \\
(a-8,b+8) & (a-8,b+6) & (a-8,b+4) & (a-8,b+2) & (a-12, b+4) &\cdots & (a-b-8,8) \\
\vdots & \vdots & \cdots & \vdots  \\
(b,a) & (b,a-2) &(b,a-4)&(b,a-8)&(b-4,a-4)& \cdots & (0,a-b)
\end{matrix}}
\end{equation*}
The number of pairs in the array is
\[ \frac z 2 (b+1) + \frac b 4 + 1, \] which is equal to $\frac 1 2 \xi_1(z,b)+ \frac 1 2 \xi_2(z,b)$. This proves our assertion in this case.
\end{proof}

\subsection{From $\mathtt C_2$ to $\mathtt A_1$ via removing the first vertex}
In this subsection, we shall prove the branching decomposition of $V_{\mathtt C_2}(a,b)$ with $a+b \equiv 0 \mod 2$ to $\mathtt A_1$ via
$\mathsf b := (\mathtt B_2 \simeq \mathtt C_2) \times \text{Levi rule}$ which removes the first vertex in the Dynkin diagram of $\mathtt C_2$.
Specifically we want to count the composition multiplicity of $V_{\mathtt A_1}(0)$ inside $V_{\mathtt C_2}(a,b)\vert^{\mathsf b }_{\mathtt A_1}$.

\begin{proposition} \label{prop:e2} For a partition $(a,b)=(k+l,k)$, we have
$$B_{\mathtt C_2}(k+l,k)\vert^{\mathsf b }_{\mathtt A_1} =\left(\soplus_{i=0}^k  (l+1)(i+1)B_{\mathtt A_1}(i) \right) \oplus \left( \soplus_{i=k+1}^{l+k}  (k+1)(l+k+1-i)B_{\mathtt A_1}(i) \right).$$
\end{proposition}

\begin{proof}
As Proposition~\ref{prop: A1xA1} and Proposition~\ref{prop:e1},  our assertion follows from the induction on $a+b$ and the Clebsch--Gordan formula.
\end{proof}

\begin{corollary}\label{cor:e2}
The composition multiplicity of $V_{\mathtt A_1}(0)$ inside $V_{\mathtt C_2}(a,b)|^{\mathsf b }_{\mathtt A_1}$ is $b-a+1$.
\end{corollary}

Now we investigate on which crystal elements in $B_{\mathtt C_2}(a,b)$ the operators $\tilde{e}_2$ and $\tilde{f}_2$ act trivially.
For a partition $(a,b)$, set $c:=a-b$. For $0 \le k \le c$, we define the semistandard $\mathtt C_2$-tableaux $T'_k$ as follows:
$$ T'_k =
\begin{cases}
\scalebox{0.8}[0.8]{\tablone} & \text{ if } k \le b \text{ and } b-k=2t+1, \\[3ex]
\scalebox{0.8}[0.8]{\tabltwo}  & \text{ if } k \le b \text{ and } b-k=2t, \\[3ex]
\scalebox{0.8}[0.8]{\tablthree}  & \text{ if } k \ge b.
\end{cases}
$$
Here $\scalebox{0.8}[0.8]{\tablas}$ denotes $\underbrace{\scalebox{0.8}[0.8]{\tablasd}}_{\text{$s$-times}}$.

Then one can easily check that they are contained in $B(Y)$ and
$$ \tilde{e}_2 T'_k =\tilde{f}_2 T'_k =0.$$
Note that the weight of $ T'_k$ is
\begin{align}\label{eq: wt2}
(2k-c)\epsilon_1, \qquad 0 \le k \le c=a-b.
\end{align}
Since the number of such crystal elements is $b-a+1$, it follows from Corollary \ref{cor:e2} that they exhaust all the crystal elements in $B_{\mathtt C_2}(a,b)$ on which the operators $\tilde{e}_2$ and $\tilde{f}_2$ act trivially.

\medskip

For the Sato--Tate subgroups $G_{1,3}$ and $N(G_{1,3})$, we record the following.
\begin{proposition} \label{prop-u1-su2}
Assume that $a-b$ is even, and let $z=(a-b)/2$. Then the dimension of weight $0$ space in $V_{\mathtt C_2}(a,b)$ on which $e_2$ and $f_2$ act trivially is one, and the space is fixed by the action of
$\mathtt a$ 
precisely when $z$ is even.
\end{proposition}

\begin{proof}

It follows from \eqref{eq: wt2} that the weight of $T'_k$ is zero if and only if $k = (a-b)/2$. This proves the first assertion. For the second assertion, let $v$ be a vector in the one-dimensional weight $0$ space on which $e_2$ and $f_2$ act trivially. Then the vector lies in the $\mathtt A_1 \times \mathtt A_1$-representation generated by the highest weight vector $v_{(a-b,0; a,b)}$ defined in Lemma \ref{vpq}.
It follows from the definition of $v_{(a-b,0;a,b)}$ that  the vector $v$ has a scalar multiple of
\[  v_1^z v_{\Oone}^z \otimes (w_{2\Otwo} - w_{1\Oone})^b \] as a term.
Since $\mathtt a v_1 = - v_{\Oone}$, $\mathtt a v_{\Oone} = v_1$ and $\mathtt a(w_{2\Otwo} - w_{1\Oone}) = w_{2\Otwo} - w_{1\Oone}$, the vector $v$ is fixed by $\mathtt a$ if and only if $z$ is even.
\end{proof}

Finally, for the Sato--Tate subgroups $G_{3,3}$ and $N(G_{3,3})$, we have the following.
\begin{proposition} \label{prop-su2-su2}
The dimension of weight $0$ space in $V_{\mathtt C_2}(a,b)$ on which $e_i$ and $f_i$ act trivially, $i=1,2$, is one if $a=b$ and is zero otherwise. In the case $a=b$, the space is fixed by the action of $J$ precisely when $b$ is even.
\end{proposition}

\begin{proof}
The first assertion follows from Corollary \ref{cor-c2-a1a1} (b). For the second assertion, let $v$ be a vector in the one-dimensional weight $0$ space on which $e_i$ and $f_i$ act trivially, $i=1,2$. Then the vector $v$ is a scalar multiple of $v_{(0,0;a,b)}$ defined in Lemma \ref{vpq}.
It follows from the definition of $v_{(0,0;a,b)}$ that  the vector $v$ has a scalar multiple of
\[  (w_{2\Otwo} - w_{1\Oone})^b \] as a term.
Since $J(w_{2\Otwo} - w_{1\Oone}) =-( w_{2\Otwo} - w_{1\Oone})$, the vector $v_\mu$ is fixed by $J$ if and only if $b$ is even.
\end{proof}

\newcommand{\etalchar}[1]{$^{#1}$}


\end{document}